\pdfoutput=1
\documentclass[10pt]{amsart}
\usepackage{amssymb,amscd,graphicx,enumerate,epsfig}
\usepackage[bookmarks=true]{hyperref}
\newtheorem{theorem}{Theorem}[section]

\newtheorem{corollary}[theorem]{Corollary}
\newtheorem{lemma}[theorem]{Lemma}
\theoremstyle{definition}
\newtheorem{definition}[theorem]{Definition}

\newtheorem{claim}{Claim}[]
\newcommand{\D}{\mathcal{D}}
\newcommand{\DV}{\mathcal{D}_\mathcal{V}}
\newcommand{\DW}{\mathcal{D}_\mathcal{W}}
\newcommand{\DVW}{\mathcal{D}_{\mathcal{VW}}}

\newenvironment{proofN}[1]{\noindent\textit{Proof of #1.}}{\hfill$\square$\\}
\newcommand{\Case}[1]{\textbf{Case #1.}}
\newcommand{\V}{\mathcal{V}}
\newcommand{\W}{\mathcal{W}}
\newcommand{\Thick}{\operatorname{Thick}}
\newcommand{\Thin}{\operatorname{Thin}}
\begin{document}

\title[]{Generalized Heegaard splittings and the disk complex}

\author{Jungsoo Kim}
\date{October 8, 2016}

\begin{abstract}
Let $M$ be an orientable, irreducible $3$-manifold and $(\V,\W;F)$ a weakly reducible, unstabilized Heegaard splitting of $M$ of genus at least three.
In this article, we define an equivalent relation $\sim$ on the set of the generalized Heegaard splittings obtained by weak reductions and find special subsets of the disk complex $\D(F)$ named by the ``\textit{equivalent clusters}'', where we can find a canonical function $\Phi$ from the set of equivalent clusters to the set of the equivalent classes for the relation $\sim$.
As an application, we prove that if $F$ is topologically minimal and the topological index of $F$ is at least three, then there is a $2$-simplex in $\D(F)$ formed by two weak reducing pairs such that the equivalent classes of the generalized Heegaard splittings obtained by weak reductions along the weak reducing pairs for the relation $\sim$ are different.
In the last section, we prove $\Phi$ is a bijection if the genus of $F$ is three.
Using it, we prove there is a canonical function $\Omega$ from the set of components of $\DVW(F)$ to the set of the isotopy classes of the generalized Heegaard splittings obtained by weak reductions and describe what $\Omega$ is. 
\end{abstract}

\address{\parbox{4in}{
	BK21 PLUS SNU Mathematical Sciences Division,\\ Seoul National University\\ 
	1 Gwanak-ro, Gwanak-Gu, Seoul 08826, %Korea\\
}} 
	
\email{pibonazi@gmail.com}
\subjclass[2000]{57M50}

\maketitle
\section{Introduction and Result}
Let $M$ be an orientable, irreducible $3$-manifold and $(\V,\W;F)$ an unstabilized Heegaard splitting of $M$ whose genus is at least three.

In this article, we will define an equivalent relation $\sim$ on the set of the generalized Heegaard splittings obtained by weak reductions, where two generalized Heegaard splittings $
\mathbf{H}_1$ and $\mathbf{H}_2$ obtained by weak reductions are \textit{equivalent} or $\mathbf{H}_1\sim \mathbf{H}_2$ if the following hold:
\begin{enumerate}
\item $
\mathbf{H}_1$ is isotopic to $\mathbf{H}_2$ in $M$ as two sets of surfaces,
\item $\Thick(\mathbf{H}_1)\cap\V$ is isotopic to $\Thick(\mathbf{H}_2)\cap\V$  in $\V$, 
\item $\Thick(\mathbf{H}_1)\cap\W$ is isotopic to $\Thick(\mathbf{H}_2)\cap\W$ in $\W$, and 
\item $\Thin(\mathbf{H}_1)\cap\textrm{int}(M)$ is isotopic to $\Thin(\mathbf{H}_2)\cap\textrm{int}(M)$ in $M$,
\end{enumerate} where $\Thick(\mathbf{H}_i)$ and $\Thin(\mathbf{H}_i)$ are the thick and thin levels of $\mathbf{H}_i$ for $i=1,2$, respectively.
Then we will prove the following theorem.

\begin{theorem}\label{theorem-main}
Suppose $M$ is an orientable, irreducible $3$-manifold and $(\V,\W;F)$ is a weakly reducible, unstabilized Heegaard splitting of $M$ of genus at least three.
Then there are special subsets of $\D(F)$ named by the ``\textit{equivalent clusters}''  such that there is a canonical function $\Phi$ from the set of the equivalent clusters to the set of the equivalent classes for the relation $\sim$.
The equivalent clusters satisfy the following:
\begin{enumerate}
\item Each equivalent cluster is connected.
\item For a given equivalent cluster, the generalized Heegaard splittings obtained by weak reductions along the weak reducing pairs in the cluster are all equivalent.
Moreover, the equivalent clusters are classified into ten types (Definition \ref{definition-3-3-detail}, Definition \ref{definition-3-5-detail}, Definition \ref{definition-3-5-ii-detail} and Definition \ref{definition-3-6}).

\item If there is a simplex of $\D(F)$ such that the generalized Heegaard splittings obtained by weak reductions along the weak reducing pairs in the simplex are all equivalent, then this simplex belongs to a uniquely determined equivalent cluster (Theorem \ref{theorem-DVWF}).
\end{enumerate}
\end{theorem}

In \cite{Bachman2010}, Bachman introduced the \textit{topological index theory} for surfaces in $3$-manifolds, where it is the generalization of the \textit{critical surface theory} in his former articles \cite{Bachman2002} and \cite{Bachman2008}.
In this article, he defined a \textit{topologically minimal surface} and the \textit{topological index} of a tologically minimal surface is defined as the homotopy index of $\D(F)$.
A topologically minimal surface intersects an incompressible surface in essential curves on both surfaces up to isotopy in an irreducible $3$-manifold [\cite{Bachman2010}, Corollary 3.8].
Moreover, if there is a topologically minimal Heegaard splitting in a $3$-manifold, then the boundary of the $3$-manifold is incompressible [\cite{Bachman2010}, Corollary 4.4].
As well as the topological minimality itself, the topological index was turned out to be very useful in his subsequent works giving a resemblance between  topologically minimal surfaces and  geometrically minimal surfaces (see  \cite{Bachman2012-1}, \cite{Bachman2012-2} and \cite{Bachman2013-1}).
For a topologically minimal surface $F$, if $F$ is incompressible,  strongly irreducible, or weakly reducible, then its topological index  is zero, one, or at least two, respectively.
The question is what would be a distinctive property of surfaces of topological index at least three?
As an answer for this question, we will prove the following theorem giving a sufficient  condition that $F$ cannot be of topological index at least three.

\begin{theorem}[the contraposition of Theorem \ref{theorem-not-minimal-critical}]\label{theorem-main-a}
Suppose $M$ is an orientable, irreducible $3$-manifold and $(\V,\W;F)$ is an unstabilized Heegaard splitting of $M$ of genus at least three.
If $F$ is topologically minimal and the topological index of $F$ is at least three, then there is a $2$-simplex in $\D(F)$ formed by two weak reducing pairs such that the equivalent classes of the generalized Heegaard splittings obtained by weak reductions along the weak reducing pairs for the relation $\sim$ are  different.
\end{theorem}

Note that Theorem \ref{theorem-not-minimal-critical} is the generalization of Theorem 1.1 of \cite{JungsooKim2014} (see Corollary \ref{lemma-equivalent-genus3}).

In the last section, we will consider the case where the genus of $F$ is three and prove that each component of $\DVW(F)$ is just an equivalent cluster and the canonical function $\Phi$ is bijective. 
Using it, we will prove there is a naturally induced function $\Omega$ from the set of components of $\DVW(F)$ to set of the isotopy classes of the generalized Heegaard splittings obtained by weak reductions and describe what $\Omega$ is, where $\DVW(F)$ is the union of all simplices of $\D(F)$ intersecting both $\DV(F)$ and $\DW(F)$.

\begin{theorem}[Theorem \ref{theorem-1-1-main}]\label{theorem-1-1}
Let $(\V,\W;F)$ be a weakly reducible, unstabilized Heegaard splitting of genus three  in an orientable, irreducible $3$-manifold $M$.
Then the domain of $\Phi$ is the set of components of $\DVW(F)$, $\Phi$ is bijective, and there is a canonically induced function $\Omega$ from the set of components of $\DVW(F)$ to the set of the isotopy classes of the generalized Heegaard splittings obtained by weak reductions from $(\V,\W;F)$.
The number of components of the preimage of an isotopy class of $\Omega$ is the number of ways to embed the thick level contained in $\V$ into $\V$ (or in $\W$ into $\W$).
This means if we consider a generalized Heegaard splitting $\mathbf{H}$ obtained by weak reduction from $(\V,\W;F)$, then the way to embed the thick level of $\mathbf{H}$ contained in $\V$ into $\V$ determines the way to embed the thick level of $\mathbf{H}$ contained in $\W$ into $\W$ up to isotopy and vise versa.
\end{theorem}

This article is constructed as follows.
First, we will classify the generalized Heegaard splittings obtained by weak reductions into five types in Lemma \ref{lemma-five-GHSs} simply by the shapes of the two compression bodies intersecting the inner thin level, where the \textit{inner thin level} is the union of the components of the thin level contained in $\mathrm{int}(M)$.
Then we will find a neccessary and sufficient condition that all weak reducing pairs in a $2$-simplex in $\D(F)$ induce equivalent generalized Heegaard splittings after weak reductions in Lemma \ref{lemma-equivalent}.
Using this lemma, we will reclassify the generalized Heegaard splittings obtained by weak reductions into ten types in Definition \ref{definition-GHSs}.
After this reclassification, we will define the ``\textit{equivalent clusters}''.
Of course, a component of $\DVW(F)$ might not be an equivalent cluster in general (see Theorem \ref{theorem-DVWF}).
But, we will find a necesary and sufficient condition that each component of $\DVW(F)$ is an equivalent cluster in Lemma \ref{lemma-equivalent-GHSs}.
Using this idea, we will prove Theorem  \ref{theorem-not-minimal-critical}.
In the last section, we will find a sufficient condition that two equivalent clusters corresponding to the same equivalent class by $\Phi$ are the same in Lemma \ref{lemma-3-11}.
By applying Lemma \ref{lemma-3-11} to the genus three case, we will prove Theorem \ref{theorem-1-1-main}.

\section*{Acknowledgments}
The author is grateful to Professor Jongil Park and the other professors of the BK21 PLUS SNU Mathematical Sciences Division for giving me the opportunity to be a postdoctoral researcher. 
This research was supported by BK21 PLUS SNU Mathematical Sciences Division.

\section{Preliminaries\label{section2}}

A \textit{compression body} is a $3$--manifold which can be obtained by starting with a closed, orientable, connected surface $F$, forming the product $F\times I$, attaching some number of $2$-handles to $F\times\{1\}$ and capping off all  resulting $2$--sphere boundary components that are not contained in $F\times\{0\}$ with $3$-balls. 
The boundary component $F\times\{0\}$ is referred to as $\partial_+$ and the rest of the boundary is referred to as $\partial_-$. 
For a compression body $\V$ with $\partial_-\V=\emptyset$, we call it a \textit{handlebody}.
For a compression body $\V$, we call the genus of $\partial_+\V$ the \textit{genus} of $\V$.
For a compression body $\V$ with $\partial_-\V\neq\emptyset$, we can obtain $\V$ from $\partial_-\V\times I$ by attaching $1$-handles to $\partial_-\V\times \{1\}$.
This means the genus of $\partial_+\V$ is at least the sum of genera of the components of $\partial_-\V$.
For a surface $F$, if there is a embedded disk $D$ in $M$ such that $\partial D\subset F$ is an essential curve in $F$ and $\mathrm{int}(D)\cap F=\emptyset$, then we call $D$ a \textit{compressing disk} for $F$.
For a compression body $\V$, if $D$ is properly embedded in $\V$ and it is a compressing disk for $\partial_+\V$, then we call $D$ a \textit{compressing disk} in $\V$. 
A \textit{Heegaard splitting} of a $3$-manifold $M$ is an expression of $M$ as a union $\V\cup_F \W$, denoted   as $(\V,\W;F)$,  where $\V$ and $\W$ are compression bodies that intersect in a transversally oriented surface $F=\partial_+\V=\partial_+\W$. 
We say $F$ is the \textit{Heegaard surface} of this splitting. 
If $\V$ or $\W$ is homeomorphic to a product, then we call such compression body \textit{trivial} and we say the splitting  is \textit{trivial}. 
If there are compressing disks $V\subset \V$ and $W\subset \W$ such that $V\cap W=\emptyset$, then we say the splitting is \textit{weakly reducible} and call the pair $(V,W)$ a \textit{weak reducing pair}. 
If $(V,W)$ is a weak reducing pair and $\partial V$ is isotopic to $\partial W$ in $F$, then we call $(V,W)$ a \textit{reducing pair}.
If the splitting is not trivial and we cannot take a weak reducing pair, then we call the splitting \textit{strongly irreducible}. 
If there is a pair of compressing disks $(\bar{V},\bar{W})$ such that $\bar{V}$ intersects $\bar{W}$ transversely in a point in $F$, then we call this pair a \textit{canceling pair} and say the splitting is \textit{stabilized}. 
Otherwise, we say the splitting is \textit{unstabilized}.

Let $F$ be a surface of genus at least two in a compact, orientable $3$-manifold $M$. 
Then the \emph{disk complex} $\D(F)$ is defined as follows: 
\begin{enumerate}[(i)]
\item Vertices of $\D(F)$ are isotopy classes of compressing disks for $F$.
\item A set of $m+1$ vertices forms an $m$-simplex if there are representatives for each
that are pairwise disjoint.
\end{enumerate}

Consider a Heegaard splitting $(\V,\W;F)$ of genus at least two in  an orientable, irreducible $3$-manifold $M$. 
Let $\DV(F)$ and $\DW(F)$ be the subspaces of $\D(F)$ spanned by compressing disks in $\V$ and $\W$, respectively. 
We call these subspaces \textit{the disk complexes of $\V$ and $\W$}, respectively.
Let $\DVW(F)$ be the subset of $\D(F)$ consisting of all simplices with at least one vertex from $\DV(F)$ and at least one vertex from  $\DW(F)$.
By definition, if there is a simplex $\sigma$ in $\DVW(F)$ such that $\sigma\subset\DV(F)$ or $\DW(F)$, then it must be a proper subsimplex of a simplex  intersecting both $\DV(F)$ and $\DW(F)$.

If there is no confusion, we will use the disk $V$ instead of the isotopy class $[V]\in\D(F)$ for the sake of convenience.
Hence, we will say (1) $D_1=D_2$ in $\D(F)$ instead of $[D_1]=[D_2]$ and (2) if there is a $k$-simplex $\Delta$ in $\D(F)$, then we will denote $\Delta$ as a set of $k+1$ mutually disjoint, non-isotopic compressing disks for $F$.
Note that if there is a compressing disk in a compression body $\V$, then the isotopy class of the boundary of the disk in $\partial_+\V$ determines the isotopy class of the compressing disk in $\V$.
(Since $\partial F=\emptyset$, the pair $(\V,F)$ is \textit{spotless}. Therefore, the natural map $\nu:\DV(F)\to\mathcal{C}(F)$ induced by taking a compressing disk in $\V$ to its boundary is injective, where $\mathcal{C}(F)$ is the \textit{curve complex} of $F$, see \cite{MS2013} for the details.)
This gives an important lemma.
\begin{lemma}\label{lemma-isotopic-bd}
Let $\mathcal{V}$ be a compression body.
If there are two compressing disks $V_1$ and $V_2$ in $\V$ such that $\partial V_1$ is isotopic to $\partial V_2$ in $\partial_+\V$, then $V_1$ is isotopic to $V_2$ in $\V$.
\end{lemma}

We get the following corollary directly from Lemma \ref{lemma-isotopic-bd}.

\begin{corollary}\label{corollary-isotopic-bd}
Let $\mathcal{V}$ be a compression body and $\mathcal{V}'\subset \mathcal{V}$ a compression body such that $\partial_+\mathcal{V}=\partial_+\mathcal{V}'$.
If $V_1$ and $V_2$ are compressing disks in $\mathcal{V}'$ such that $V_1$ is isotopic to $V_2$ in $\V'$, then $V_1$ is isotopic to $V_2$ in $\V$.
\end{corollary}

Note that $\DV(F)$ and $\DW(F)$ are contractible (see \cite{8}).
Let $(\V,\W;F)$ be a Heegaard splitting of genus at least two.
Since $\dim(\mathcal{C}(F))$ is finite and there exist at most two disks in $\D(F)$ which correspond to an element of $\mathcal{C}(F)$ (such two disks exist if and only if there exists a reducing pair), $\dim(\D(F))$ is also finite.

\begin{definition}[Bachman, \cite{Bachman2010}]
The \emph{homotopy index} of a complex $\Gamma$ is defined to be 0 if $\Gamma=\emptyset$, and the smallest $n$ such that $\pi_{n-1}(\Gamma)$ is non-trivial, otherwise. 
We say a separating surface $F$ with no torus components is \emph{topologically minimal} if $\mathcal{D}(F)$ is either empty or non-contractible. 
When $F$ is topologically minimal, we say its \emph{topological index} is the homotopy index of $\mathcal{D}(F)$. 
\end{definition}

From now on, we will consider only unstabilized Heegaard splittings of an irreducible $3$-manifold. 
If a Heegaard splitting of a compact $3$-manifold is reducible, then the manifold is reducible or the splitting is stabilized (see \cite{SaitoScharlemannSchultens2005}).
Hence, we can exclude the possibilities of reducing pairs among weak reducing pairs.

Suppose $W$ is a compressing disk for a properly embedded surface $F\subset M$. 
Then there is a subset of $M$ that can be identified with $W\times I$ so that $W=W\times\{\frac{1}2\}$ and $F\cap(W\times I)=(\partial W)\times I$. 
We form the surface $F_W$, obtained by \textit{compressing $F$ along $W$}, by removing $(\partial W)\times I$ from $F$ and replacing it with $W\times(\partial I)$. 
We say the two disks $W\times(\partial I)$ in $F_W$ are the $\textit{scars}$ of $W$.

\begin{lemma}[Lustig and Moriah, Lemma 1.1 of \cite{7}] \label{lemma-2-8}
Suppose $M$ is an irreducible $3$-manifold and $(\V,\W;F)$ is an unstabilized Heegaard splitting of $M$. 
If $F'$ is obtained by compressing $F$ along a collection of pairwise disjoint disks, then no $S^2$ component of $F'$ can have scars from disks in both $\V$ and $\W$. 
\end{lemma}

The next is the definition of ``\textit{generalized Heegaard splitting}'' originated from \cite{ScharlemannThompson1994}.

\begin{definition}[Definition 4.1 of \cite{Bachman2008}]\label{definition-2-9}
A \textit{generalized Heegaard splitting} (we will use \textit{GHS} as abbreviation from now on) $\mathbf{H}$ of a $3$-manifold $M$ is a pair of sets of pairwise disjoint, transversally oriented, connected surfaces, $\operatorname{Thick}(\mathbf{H})$ and $\operatorname{Thin}(\mathbf{H})$ (resp called the \textit{thick levels} and \textit{thin levels}), which satisfies the following conditions.
\begin{enumerate}
\item Each component $M'$ of $M-\operatorname{Thin}(\mathbf{H})$ meets a unique element $H_+$ of $\operatorname{Thick}(\mathbf{H})$ and $H_+$ is a Heegaard surface in the closure of $M'$, say  $M(H_+)$.
\item As each Heegaard surface $H_+\subset M(H_+)$ is transversally oriented, we can consistently talk about the points of $M(H_+)$ that are ``above''  $H_+$ or ``below'' $H_+$.
Suppose $H_-\in\operatorname{Thin}(\mathbf{H})$.
Let $M(H_+)$ and $M(H_+')$ be the submanifolds on each side of $H_-$.
Then $H_-$ is below $H_+$ if and only if it is above $H_+'$.
\item There is a partial ordering on the elements of $\operatorname{Thin}(\mathbf{H})$ which satisfies the following: Suppose $H_+$ is an element of $\operatorname{Thick}(\mathbf{H})$, $H_-$ is a component of $\partial M(H_+)$ above $H_+$ and $H_-'$ is a component of $\partial M(H_+)$ below $H_+$.
Then $H_->H_-'$.
\end{enumerate}
We denote the maximal subset of $\operatorname{Thin}(\mathbf{H})$ consisting of surfaces in $\mathrm{int}(M)$ as $\overline{\operatorname{Thin}}(\mathbf{H})$ and call it  the \textit{inner thin level}.
\end{definition}

\begin{definition}[a restricted version of Definition 5.2, Definition 5.3, and Definition 5.6 of \cite{Bachman2008}]\label{definition-WR}
Let $M$ be an orientable, irreducible 3-manifold.
Let $\mathbf{H}$ be an unstabilized Heegaard splitting of $M$, i.e. $\operatorname{Thick}(\mathbf{H})=\{F\}$ and $\operatorname{Thin}(\mathbf{H})$ consists of $\partial M$.
Let $V$ and $W$ be disjoint compressing disks for $F$ from the opposite sides of $F$.
Here, Lemma \ref{lemma-2-8} guarantees ${F}_{VW}$ does not have a $2$-sphere component.
Define
$$\operatorname{Thick}(\mathbf{G'})=(\operatorname{Thick}(\mathbf{H})-\{F\})\cup\{{F}_{V}, {F}_{W}\}, \text{ and}$$
$$\operatorname{Thin}(\mathbf{G'})=\operatorname{Thin}(\mathbf{H})\cup\{{F}_{VW}\},$$
where we assume each element of $\operatorname{Thick}(\mathbf{G'})$ belongs to $\mathrm{int}(\V)$ or $\mathrm{int}(\W)$ by slightly pushing $F_V$ and $F_W$ off into $\mathrm{int}(\V)$ and $\mathrm{int}(\W)$, respectively, and then also assume they miss $F_{VW}$. 
(In the proof of Lemma 2.16 of \cite{JungsooKim2014}, the author used the term $F_V$ for denoting $\operatorname{Thick}(\mathbf{G'})\cap \V$ for the sake of convenience.
In this article, we will only use this term for this definition and we will use $\tilde{F}_V$ for general cases to avoid confusion.)
We say the GHS $\mathbf{\mathbf{G'}}=\{\operatorname{Thick}(\mathbf{G'}),\operatorname{Thin}(\mathbf{G'})\}$ is obtained from $\mathbf{H}$ by \textit{preweak reduction} along $(V,W)$. 
The relative position of the elements of $\operatorname{Thick}(\mathbf{G'})$ and $\operatorname{Thin}(\mathbf{G'})$ follows the order described in Figure 10 of \cite{Bachman2008}.
If there are elements $S\in \operatorname{Thick}(\mathbf{G'})$ and $s\in \overline{\operatorname{Thin}}(\mathbf{G'})$ that cobound a product region $P$ of $M$ such that $P\cap \operatorname{Thick}(\mathbf{G'})=S$ and $P\cap\operatorname{Thin}(\mathbf{G'})=s$, then remove $S$ from $\operatorname{Thick}(\mathbf{G'})$ and $s$ from $\operatorname{Thin}(\mathbf{G'})$.
If we repeat this procedure until there are no such two elements of $\operatorname{Thick}(\mathbf{G'})$ and $\overline{\operatorname{Thin}}(\mathbf{G'})$, then we get the resulting GHS $\mathbf{G}$ of $M$ from  the GHS $\mathbf{G'}$ (see Lemma 5.4 of \cite{Bachman2008}) and we say $\mathbf{G}$ is obtained from $\mathbf{G'}$ by \textit{cleaning}.
We say the GHS $\mathbf{G}$ of $M$ given by preweak reduction along $(V,W)$, followed by cleaning, is obtained from $\mathbf{H}$ by \textit{weak reduction} along $(V,W)$.
From now on, we will denote the GHS $\mathbf{G}$ obtained by weak reduction from the Heegaard surface $F$ along a weak reducing pair $(V,W)$ as the ordered triple $(\bar{F}_{V},\bar{F}_{V W},\bar{F}_{W})$ without denoting $\partial M$ for the sake of convenience, where (i) $\Thick(\mathbf{G})\cap\V=\bar{F}_V$, $\Thick(\mathbf{G})\cap\W=\bar{F}_W$, and $\overline{\Thin}(\mathbf{G})=\bar{F}_{VW}$ and (ii) the notations $\bar{F}_V$, $\bar{F}_W$ and $\bar{F}_{VW}$ imply they come from $F_V$, $F_W$ and $F_{VW}$, respectively.
Note that we will confirm the GHSs obtained by  weak reductions in Lemma \ref{lemma-five-GHSs}.
\end{definition}

\begin{definition}\label{definition-HS-GHS}
Let $F$ be a weakly reducible Heegaard surface in an orientable, irreducible $3$-manifold $M$.
We say the GHSs $\mathbf{H}_1$ and $\mathbf{H}_2$ are \textit{isotopic} if there is an ambient isotopy $h_t$ defined on $M$ such that $h_t$ sends  $\Thick(\mathbf{H}_1)\cup\Thin(\mathbf{H}_1)$ into $\Thick(\mathbf{H}_2)\cup\Thin(\mathbf{H}_2)$.
Let $\mathcal{GHS}_F$ be the set of isotopy classes of the GHSs obtained by weak reductions from $(\V,\W;F)$.
We will say two GHSs $\mathbf{H}_1=(\bar{F}_{V_1},\bar{F}_{V_1 W_1},\bar{F}_{W_1})$ and $\mathbf{H}_2=(\bar{F}_{V_2},\bar{F}_{V_2 W_2},\bar{F}_{W_2})$ obtained by weak reductions are \textit{equivalent} if (i) $\mathbf{H}_1$ is isotopic to $\mathbf{H}_2$, i.e. $[\mathbf{H}_1]=[\mathbf{H}_2]$ in $\mathcal{GHS}_F$ (ii)
$\bar{F}_{V_1}$ is isotopic to $\bar{F}_{V_2}$ in $\V$, (iii) $\bar{F}_{W_1}$ is isotopic to $\bar{F}_{W_2}$ in $\W$, and (iv) $\bar{F}_{V_1 W_1}$ is isotopic to $\bar{F}_{V_2 W_2}$ in $M$.
\end{definition}

We get the following lemma from Definition \ref{definition-HS-GHS} and omit the trivial proof.

\begin{lemma}\label{lemma-equivalent-class}
Let $\mathcal{GHS}_F^{\ast}$ be the set of GHSs obtained by weak reductions from $(\V,\W;F)$ and $\sim$ the relation defined on $\mathcal{GHS}_F^\ast$ such that $\mathbf{H}_1\sim\mathbf{H}_2$ if $\mathbf{H}_1$ and $\mathbf{H}_2$ are equivalent.
Then $\sim$ is an equivalent relation on $\mathcal{GHS}^\ast_F$.
\end{lemma}

Considering Lemma \ref{lemma-equivalent-class}, $\mathcal{GHS}^\ast_F/\sim$ is the set of the equivalent classes.
We will denote the equivalent class of a GHS $\mathbf{H} \in \mathcal{GHS}^\ast_F$ as $[\mathbf{H}]^\ast$ to distinguish it from the isotopy class $[\mathbf{H}]$.

The next two lemmas give  natural, but important observations for a genus $n\geq 2$ compression body  having a special  compressing disk such that it divides the compression body into one or two products.

\begin{lemma}[Ido, Jang and Kobayashi, Lemma 3.3 of \cite{IdoJangKobayashi2014}]\label{lemma-2-19}
Let $\V$ be a genus $n\geq 2$ compression body such that $\partial_-\V$ consists of a genus $n-1$ connected surface.
Then there is a unique nonseparating compressing disk in $\V$ up to isotopy.
Since any separating compressing disk in $\V$ cuts off a uniquely determined solid torus from $\V$, every nonseparating compressing disk in $\V$ is isotopic to a meridian disk of the solid torus in $\V$.
\end{lemma}

Note that the phrase ``\textit{any separating compressing disk in $\V$ cuts off a uniquely determined solid torus from $\V$}'' of Lemma \ref{lemma-2-19} is guaranteed by Lemma 1.3 of \cite{ScharlemannThompson1993} by considering the genus of each component.
The next lemma deals with the case where the negative boundary is disconnected.
For the sake of convenience, if we say $\V'$ is a \textit{disconnected compression body}, then it means $\V'$ is a collection of mutually disjoint compression bodies, where $\partial_+ \V'$ (resp $\partial_-\V'$) is the union of the positive (resp negative) boundaries of them.

\begin{lemma}\label{lemma-2-20}
Let $\V$ be a genus $n\geq 2$ compression body such that $\partial_-\V$ consists of a genus $m\geq 1$ connected surface and a genus $(n-m)\geq 1$ connected surface.
Then there is a unique compressing disk in $\V$ up to isotopy and it is separating in $\V$.
\end{lemma}

\begin{proof}
Let $S_m$ and $S_{n-m}$ be the genus $m$ and $(n-m)$ components of $\partial_-\V$, respectively.
Since $\V$ is not trivial, there is a compressing disk $\bar{V}$ in $\V$.
Then the negative boundary of the (possibly disconnected) compression body $\V'$ obtained by cutting $\V$ along $\bar{V}$ consists of $S_m$ and $S_{n-m}$ by Lemma 1.3 of \cite{ScharlemannThompson1993}.
If $\bar{V}$ is nonseparating in $\V$, then $g(\partial_+\V')=n-1$ and therefore it is less than the sum of genera of components  of $\partial_-\V'$, leading to a contradiction.
Hence, $\bar{V}$ is separating in $\V$ and $\V'$ consists of two compression bodies such that one contains $S_m$ and the other contains $S_{n-m}$ in their negative boundaries, respectively.
Here, the positive boundaries of them must be of genera $m$ and $n-m$, respectively, where they are of the smallest genera, otherwise the genus of $F$ would be larger than $n$.
This means one is $S_m\times I$ and the other is  $S_{n-m}\times I$.
We assume $(S_m\times \{0\})\cap(S_{n-m}\times \{0\})=\bar{V}$.
We will prove every compressing disk $V$ in $\V$ is isotopic to $\bar{V}$ in $\V$.

Suppose $V\cap \bar{V}=\emptyset$.
Then we can assume $V\subset S_m\times I$ without loss of generality.
Since $V$ is a properly embedded disk in $S_m\times I$ such that $\partial V\subset S_m\times \{0\}$ and there is no compressing disk in the trivial compression body $S_m \times I$, there is a disk $D\subset S_m\times\{0\}$ such that $\partial V=\partial D$.
By using the sphere $V\cup D$, we can see $V$ cuts off a $3$-ball $B$ from $S_m\times I$ such that $\partial B= V\cup D$.
If $D$ does not contain $\bar{V}$, then we get $D\subset\partial_+\V$, violating the assumption that $V$ is a compressing disk in $\V$.
Hence, $D$ contains $\bar{V}$, i.e. $\partial V$ is isotopic to $\partial \bar{V}$ in $\partial_+\V$ through the annulus between them.
This means $V$ is isotopic to $\bar{V}$ in $\V$ by Lemma \ref{lemma-isotopic-bd}, leading to the result.

Suppose $V\cap \bar{V}\neq\emptyset$.
Then we can assume $V$ intersects $\bar{V}$ transversely and  $V\cap \bar{V}$ has no loop components by standard innermost disk arguments.
Considering $V\cap\bar{V}$ in $V$, there is an outermost disk $\Delta$ in $V$.
Without loss of generality, assume $\Delta\subset S_m\times I$.
Using the same argument in the previous paragraph, $\Delta$ cuts off a $3$-ball $B$ from $S_m\times I$.
Here, $B\cap (S_m\times\{0\})$ is a disk $D$ such that (i) $D\cap \bar{V}$ is a disk $\bar{\Delta}$, (ii) $D\cap \partial_+\V$ is the disk $\bar{D}=\mathrm{cl}(D-\bar{\Delta})$, and (iii) $\bar{\Delta}\cap\bar{D}$ is the arc $\gamma=\bar{\Delta}\cap\partial_+\V$.
Hence, we can push off $B$ into $S_{n-m}\times I$ without affecting $V\cap ((S_m\times I) -B)$ so that the isotoped image of $B$ would miss $\bar{V}$, where the intersection $\bar{D}=B\cap (\partial_+ \V\cap(S_m\times \{0\}))$ shrinks into $\gamma$ and then disappears (see Figure \ref{ballisotopy}).
\begin{figure}
	\includegraphics[width=9cm]{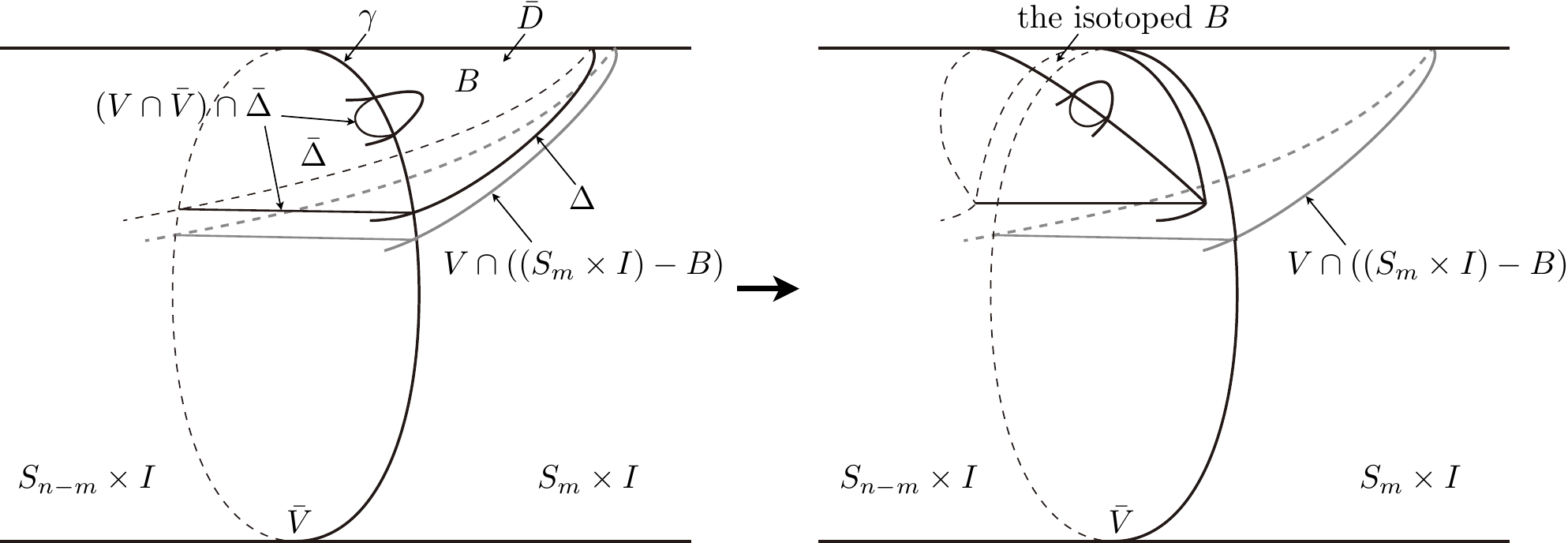}
	\caption{the isotopy pushing off $B$ into $S_{n-m}\times I$ \label{ballisotopy}}
\end{figure}
This realizes an isotopy of $V$ removing $(V\cap \bar{V})\cap\bar{\Delta}$ from $V\cap \bar{V}$ without affecting the other intersections.

Repeating the previous argument until $V\cap\bar{V}=\emptyset$, we conclude $V$ is isotopic to $\bar{V}$ in $\V$ by using the case where $V\cap\bar{V}=\emptyset$.

This completes the proof.
\end{proof}

\begin{lemma}\label{lemma-region}
Let $\V$ be a genus $n\geq 2$ compression body, $V$ a compressing disk in $\V$, and $F=\partial_+\V$.
Let $\tilde{F}_V$ be the surface obtained by pushing $F_V$ off slightly into $\mathrm{int}(\V)$.
Let $\tilde{F}'_V$ be a component of $\tilde{F}_V$ and $\tilde{\V}'$ the closure of the component of $\V-\tilde{F}_V'$ intersecting $F$. 
Then $\tilde{\V}'$ is a compression body such that $\partial_+{\tilde{\V}}'=\partial_+\V$.
\end{lemma}

\begin{proof}
If $V$ is nonseparating, then $\tilde{\V}'$ is obtained from $\tilde{F}_V\times I$ by attaching a $1$-handle whose cocore disk is $V$ to $\tilde{F}_V\times \{1\}$, i.e. $\tilde{\V}'$ is a compression body.

Suppose $V$ is separating in $\V$.
Let $\V'$ and $\V''$ be the $3$-manifolds obtained by cutting $\V$ along a small product neighborhood of $V$ in $\V$ and assume $\V'$ contains $\tilde{F}'_V$.
Then $\V'$ and $\V''$ are compression bodies such that $\partial_-\V'\cup\partial_-\V''=\partial_-\V$ by Lemma 1.3 of \cite{ScharlemannThompson1993}.
Moreover, the region between $\tilde{F}'_V$ and $\partial_+\V'$ in $\V'$ is the trivial compression body $\tilde{F}'_V\times I$ and say $\partial_-(\tilde{F}'_V\times I)=\tilde{F}'_V$.
Then $\tilde{\V}'$ is obtained from the union of $\tilde{F}'_V\times I$ and $\V''$ by attaching a $1$-handle connecting $\partial_+(\tilde{F}'_V\times I)$ and $\partial_+\V''$ whose cocore disk is $V$, i.e. $\tilde{\V}'$ is a compression body.

This completes the proof. 
\end{proof}

\begin{lemma}\label{lemma-2-21}
Let $\V$ be a compression body of genus $n\geq 3$, $V_1$ and $V_2$  compressing disks in $\V$, and $F=\partial_+\V$.
Let  $\tilde{F}_{V_1}$ and $\tilde{F}_{V_2}$ be the surfaces obtained by pushing $F_{V_1}$ and $F_{V_2}$ slightly into $\mathrm{int}(\V)$, respectively.
Then the following hold:
	\begin{enumerate}
		\item Suppose both $V_1$ and $V_2$ are nonseparating in $\V$.
		Then $\tilde{F}_{V_2}$ is isotopic to $\tilde{F}_{V_1}$ in $\V$ if and only if $V_2$ is isotopic to $V_1$ in $\V$.\label{lemma-2-21-1}
		\item Suppose $V_1$ is separating and $V_2$ is nonseparating in $\V$.
		Then $\tilde{F}_{V_2}$ is isotopic to a component of $\tilde{F}_{V_1}$ in $\V$ if and only if $V_1$ cuts off a solid torus from $\V$ and $V_2$ is isotopic to a meridian disk of the solid torus in $\V$.\label{lemma-2-21-2} 
	\end{enumerate}
\end{lemma}

\begin{proof}
\begin{enumerate}
	\item Suppose both $V_1$ and $V_2$ are nonseparating in $\V$.
	
	$(\Leftarrow)$ 
	We can see the region $\tilde{\V}$ cobounded by $F_{V_2}$ and $\tilde{F}_{V_2}$ is a product and we will observe the images of $\tilde{\V}$ of isotopies defined on $\V$.
	Assume $V_2$ is isotopic to $V_1$ in $\V$ by an isotopy $h_t:\V\to\V$, $t\in[0,1]$, i.e. $h_1(V_2)=V_1$.
	After the isotopy $h_t$, we push off $h_1(\tilde{F}_{V_2})$ into a small neighborhood of $F\cup h_1(V_2)$ in $\V$ preserving $h_1(V_2)$ by an isotopy $g_t$ defined on $\V$  (we can refer to (b) and (c) of Figure \ref{fig-2-10-2} ignoring the separating $V_1$ in the figure).
	Then we can assume $g_1(h_1(\tilde{F}_{V_2}))$ is $\tilde{F}_{V_1}$, leading to the result.
	
	$(\Rightarrow)$ 
	Assume $\tilde{F}_{V_2}$ is isotopic to $\tilde{F}_{V_1}$ in $\V$ by an isotopy $h_t:\V\to\V$, $t\in[0,1]$.
	Let $\tilde{\V}$ be the closure of the component of $\V-\tilde{F}_{V_1}$ intersecting $F$.
	Then $\tilde{\V}$ is a genus $n$ compression body with negative boundary consisting of a genus $(n-1)$ surface by Lemma \ref{lemma-region}, where $\partial_+\tilde{\V}=\partial_+\V$.
	Here, (i) $h_1(V_2)$ and $V_1$ are properly embedded in $\tilde{\V}$ and (ii) $\partial h_1(V_2)$ and $\partial V_1$ are nonseparating curves in $F=\partial_+\tilde{\V}$, i.e. both $h_1(V_2)$ and $V_1$ are nonseparating compressing disks in $\tilde{\V}$.
	This means $h_1(V_2)$ is isotopic to $V_1$ in $\tilde{\V}$ by Lemma \ref{lemma-2-19} and therefore $h_1(V_2)$ is isotopic to $V_1$ in $\V$ by an isotopy $g_t$ defined on $\V$ by Corollary \ref{corollary-isotopic-bd}.
	Hence, the sequence of isotopies consisting of $h_t$ and $g_t$ realizes the isotopy from $V_2$ to $V_1$ in $\V$.
	\item Suppose $V_1$ is separating and $V_2$ is nonseparating in $\V$.
	
	$(\Leftarrow)$ 
	Let $\tilde{\V}$ be as in the ($\Leftarrow$) part of the proof of (\ref{lemma-2-21-1}) (see (a) of Figure \ref{fig-2-10-2}).
	\begin{figure}
	\includegraphics[width=12cm]{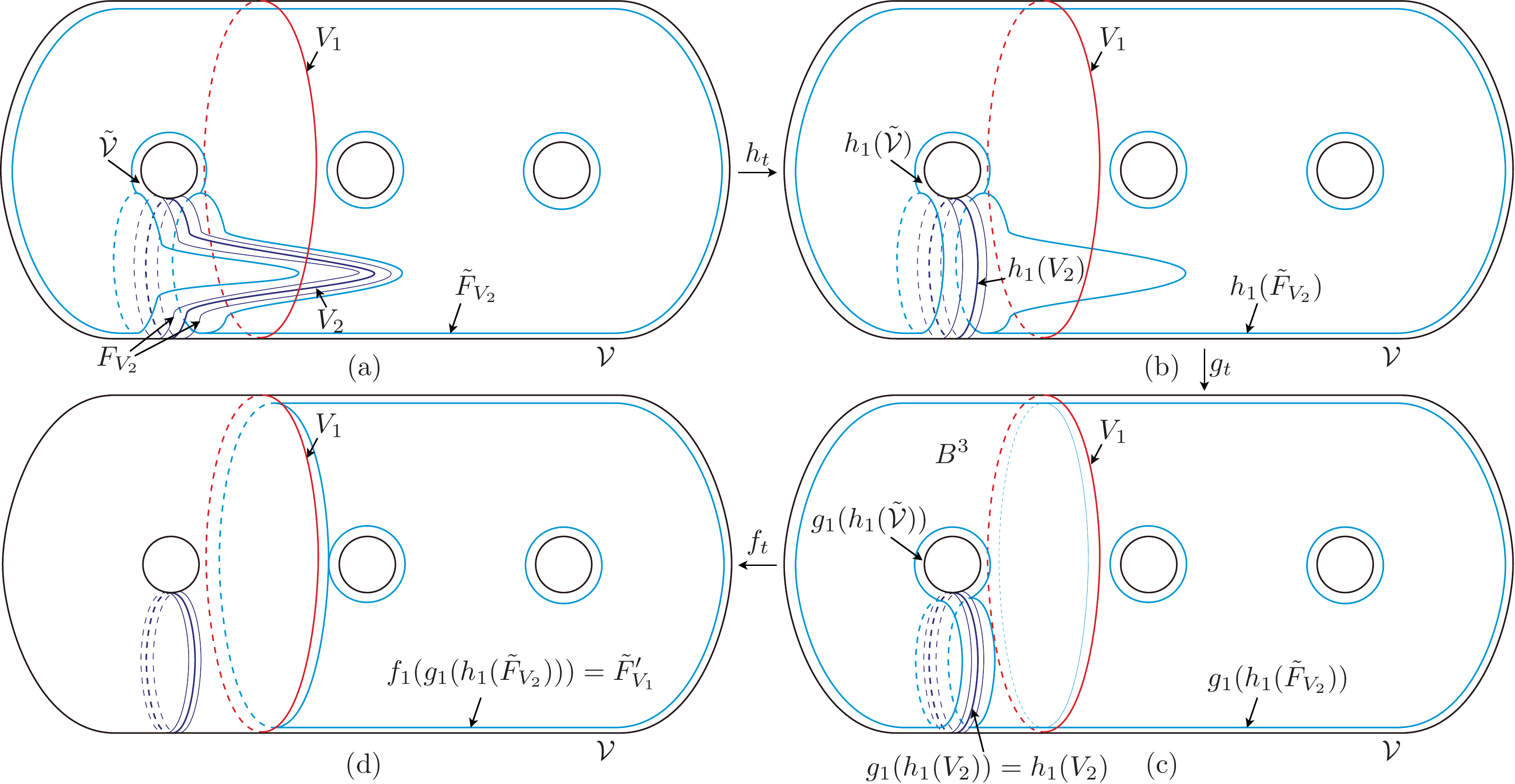}
	\caption{the isotopy taking $\tilde{F}_{V_2}$ into $\tilde{F}_{V_1}'$. \label{fig-2-10-2}}
	\end{figure}
	Assume $V_1$ cuts off a solid torus from $\V$ and there is an isotopy $h_t:\V\to\V$, $t\in[0,1]$ such that $h_1(V_2)$ is a meridian disk of the solid torus.
	Then we can also assume $h_1(V_2)$ misses $V_1$ without loss of generality (see (b) of Figure \ref{fig-2-10-2}).
	Using an additional isotopy $g_t:\V\to\V$, $t\in [0,1]$ such that it preserves $h_1(V_2)$ and pushes $h_1(\tilde{F}_{V_2})$ into a small neighborhood of $F\cup h_1(V_2)$ in $\V$, we can assume $g_1(h_1(\tilde{\V}))$ is a sufficiently thin product (see (c) of Figure \ref{fig-2-10-2}).
	Let $\tilde{F}'_{V_1}$ be the genus $(n-1)$-component of $\tilde{F}_{V_1}$.
	Then we can see $g_1(h_1(\tilde{F}_{V_2}))$ is isotopic to $\tilde{F}_{V_1}'$ in $\V$ by an isotopy $f_t$ defined on $\V$, where the reason is that there is a $3$-ball in $\V$ bounded by the $S^2$ consisting of a disk in  $\mathrm{int}(V_1)$ and a disk in  $g_1(h_1(\tilde{F}_{V_2}))$ such that these two disks share the common boundary and therefore we can push $g_1(h_1(\tilde{F}_{V_2}))$ into the interior of the component of $\V-V_1$ containing $\tilde{F}'_{V_1}$ through the $3$-ball (see (c) and (d) of Figure \ref{fig-2-10-2}).
	Hence, the sequence of isotopies consisting of $h_t$, $g_t$ and $f_t$ realizes the isotopy from $\tilde{F}_{V_2}$ to $\tilde{F}_{V_1}'$ in $\V$.
	
	$(\Rightarrow)$
	Let $\tilde{\V}'$ be the region cobounded by $F$ and $\tilde{F}_{V_2}$ in $\V$.
	Then $\tilde{\V}'$ is a genus $n$ compression body with negative boundary consisting of the genus $(n-1)$ surface $\tilde{F}_{V_2}$ by Lemma \ref{lemma-region}, where $\partial_+\tilde{\V}'=\partial_+\V$.
		Assume there is an isotopy $h_t:\V\to\V$, $t\in[0,1]$ such that $h_1(\tilde{F}_{V_2})=\tilde{F}_{V_1}'$ for a component $\tilde{F}'_{V_1}$ of $\tilde{F}_{V_1}$.
	Then $\tilde{\V}=h_1(\tilde{\V}')$ is also a compression body homeomorphic to $\tilde{\V}'$, where it is the region cobounded by $F$ and $\tilde{F}_{V_1}'$ in $\V$.
	Therefore, the genus of $\partial_-\tilde{\V}=\tilde{F}_{V_1}'$ must be $(n-1)$.
	Moreover, considering $\partial V_1$ and $\partial h_1(V_2)$  in $F=\partial_+\tilde{\V}$, $V_1$ is a separating compressing disk and $h_1(V_2)$ is a nonseparating compressing disk in $\tilde{\V}$.
	This means $h_1(V_2)$ is isotopic to a meridian disk $D$ of the solid torus that $V_1$ cuts off from $\tilde{\V}$ in $\tilde{\V}$ by Lemma \ref{lemma-2-19}.
	Therefore, $h_1(V_2)$ is isotopic to $D$ in  $\V$ by an isotopy $g_t$ defined on $\V$ by Corollary \ref{corollary-isotopic-bd}.
	Hence, the sequence of isotopies consisting of $h_t$ and $g_t$ realizes the isotopy from $V_2$ to $D$ in $\V$.
	\end{enumerate}
\end{proof}

\begin{lemma}\label{lemma-2-22-pre}
Let $\V$ be a compression body of genus $n\geq 2$ and  $V_1$ and $V_2$  separating compressing disks in $\V$ such that the geometric intersection number $i(\partial V_1,\partial V_2)=0$ in $\partial_+\V$.
If each of $V_1$ and $V_2$ cuts off $S\times I$ from $\V$ for a component $S$ of $\partial_-\V$ such that $S$ is the common $0$-level, then $V_1$ is isotopic to $V_2$ in $\V$.
\end{lemma}

\begin{proof}
Since the geometric intersection number $i(\partial V_1,\partial V_2)=0$ in $\partial_+\V$, moving $V_1$ near $\partial_+\V$ continuously, there is an isotopy $h_t:\V\to\V$, $t\in[0,1]$ such that $\partial h_1(V_1)\cap \partial V_2=\emptyset$.
Moreover, we can assume that $h_1(V_1)$ intersects $V_2$ transversely and therefore each component of $h_1(V_1)\cap V_2$ is an inessential loop in both $\mathrm{int}(h_1(V_1))$ and $\mathrm{int}(V_2)$.
This means we can find a compressing disk $V$ in $\V$ such that $\partial V=\partial h_1(V_1)$ and $V\cap V_2=\emptyset$ by using innermost disk arguments (see the proof of Claim 1 of Lemma 3.1.5 of \cite{SaitoScharlemannSchultens2005}).
Therefore, Lemma \ref{lemma-isotopic-bd} implies $h_1(V_1)$ is isotopic to $V$ in $\V$ by an isotopy $g_t$ defined on $\V$, i.e. $g_1(h_1(V_1))\cap V_2=\emptyset$.
Let $\V_1$ and $\V_2$ be the $S\times I$s such that $g_1(h_1(V_1))$ and $V_2$ cut off from $\V$, respectively, guaranteed by the assumption of this lemma.
Considering the $1$-level of $\V_1$ (resp $\V_2$), it consists of a once-punctured surface $F_1$ (resp $F_2$) in $\partial_+\V$ and $g_1(h_1(V_1))$ (resp $V_2$), where $F_1\cap g_1(h_1(V_1))=\partial F_1=\partial g_1(h_1(V_1))$, $F_2\cap V_2=\partial F_2=\partial V_2$.
Since $\V_1$ shares the $0$-level $S$ with $\V_2$, the assumption $g_1(h_1(V_1))\cap V_2=\emptyset$ implies one of $\V_1$ and $\V_2$ contains the other.
Assume $\V_1\subset \V_2$ without loss of generality, i.e. $F_1\subset F_2$.
Since $g(F_1)=g(F_2)=g(S)\geq 1$, $\partial F_1$ cuts off a uniquely determined planar surface from $F_2$.
Moreover, $g_1(h_1(V_1))$ is a compressing disk in $\V$ and therefore $\partial F_1=\partial g_1(h_1(V_1))$ cannot bound a disk in $F_2$.
This means $\partial F_1$ cuts off an annulus from $F_2$, i.e. $\partial F_1$ $(=\partial g_1(h_1(V_1)))$ is isotopic to $\partial F_2$ $(=\partial V_2)$ in $F_2\subset F$.
Therefore, $g_1(h_1(V_1))$ is isotopic to $V_2$ in $\V$ by an isotopy $f_t:\V\to\V$, $t\in[0,1]$ by Lemma \ref{lemma-isotopic-bd}.
Hence, the sequence of isotopies consisting of $h_t$, $g_t$ and $f_t$ realizes the isotopy from $V_1$ to $V_2$ in $\V$.

This completes the proof.
\end{proof}

\begin{lemma}\label{lemma-2-22}
Let $\V$ be a compression body of genus $n\geq 3$, $V_1$ and $V_2$  mutually disjoint separating compressing disks in $\V$, and $F=\partial_+\V$.
Let $\tilde{F}_{V_1}$ and $\tilde{F}_{V_2}$ be the surfaces obtained by pushing $F_{V_1}$ and $F_{V_2}$ slightly into $\mathrm{int}(\V)$.
Then at least one component of $\tilde{F}_{V_2}$ is isotopic to a component of $\tilde{F}_{V_1}$ in $\V$ if and only if $V_2$ is isotopic to $V_1$ in $\V$.
\end{lemma}

\begin{proof}
($\Leftarrow$) Assume that $V_2$ is isotopic to $V_1$ in $\V$.
Then considering the argument in the $(\Leftarrow)$ part of the proof of (\ref{lemma-2-21-1}) of Lemma \ref{lemma-2-21}, we conclude $\tilde{F}_{V_2}$ is isotopic to $\tilde{F}_{V_1}$ in $\V$.

($\Rightarrow$)
Let $\tilde{F}_{V_1}=\tilde{F}'_{V_1}\cup\tilde{F}''_{V_1}$ and $\tilde{F}_{V_2}=\tilde{F}'_{V_2}\cup\tilde{F}''_{V_2}$ by denoting each $\tilde{F}_{V_i}$ as the union of two components for $i=1,2$.
Assume $\tilde{F}'_{V_2}$ is isotopic to $\tilde{F}'_{V_1}$ by an isotopy $h_t:\V\to\V$, $t\in[0,1]$.
Let $\tilde{\V}$ be the closure of the component of $\V-\tilde{F}'_{V_1}$ intersecting $F$.
Then $\tilde{\V}$ is a genus $n$ compression body by Lemma \ref{lemma-region}, where $\partial_+\tilde{\V}=\partial_+\V$, and $V_1$ and $h_1(V_2)$ are compressing disks in $\tilde{\V}$.
By the assumption that $V_1\cap V_2=\emptyset$, the geometric intersection number $i(\partial V_1,\partial h_1(V_2))=0$ in $F$.
Here, we can see the product cobounded by $\tilde{F}'_{V_2}$, $V_2$ and  a  subsurface $F_0\subset F$ in $\V$ has been isotoped into the region cobounded by $\tilde{F}'_{V_1}$, $h_1(V_2)$ and $h_1(F_0)\subset F$ in $\V$ and therefore the latter region is also homeomorphic to a product.
Since $\partial_+\tilde{\V} = \partial_+\V=F$, $\tilde{F}_{V_1}'\subset\partial_-\tilde{\V}$, and each of $V_1$ and $h_1(V_2)$ cuts off $\tilde{F}_{V_1}'\times I$ from $\tilde{\V}$ such that the one corresponding to $V_1$ shares the common $0$-level $\tilde{F}_{V_1}'$ with that corresponding to $h_1(V_2)$, Lemma \ref{lemma-2-22-pre} implies $h_1(V_2)$ is isotopic to $V_1$ in $\tilde{\V}$.
Therefore, $h_1(V_2)$ is isotopic to $V_1$ by an isotopy $g_t:\V\to\V$, $t\in[0,1]$ by Corollary \ref{corollary-isotopic-bd}.
Hence, the sequence of isotopies consisting of $h_t$ and $g_t$ realizes the isotopy from $V_2$ to $V_1$ in $\V$.

This completes the proof.
\end{proof}

\begin{lemma}\label{lemma-pre}
Let $F$ be a closed surface embedded in a $3$-manifold.
Let $V$ and $W$ be mutually disjoint compressing disks for $F$ such that $\partial V$ is separating and $\partial W$ is nonseparating in $F$.
Then $\partial W$ is nonseparating in the relevant component of $F_V$.
\end{lemma}

\begin{proof}
Let $F'_V$ be the component of $F_V$ containing $\partial W$.
For the sake of contradiction, assume $\partial W$ is separating in $F'_V$.
Then one component of $F'_V-\partial W$ has a scar of $V$ and the other doesn't because there is only one scar of $V$ in $F'_V$ by the assumption that $\partial V$ is separating in $F$.
This means $\partial W$ cuts off the latter component from $F$, violating the assumption that $\partial W$ is nonseparating in $F$.
This completes the proof.
\end{proof}

\section{the proof of Theorem \ref{theorem-main}}

Lemma \ref{lemma-five-GHSs} characterizes the possible GHSs obtained by weak reductions from $(\V,\W;F)$ into five types, where the five types depend on the shapes of the compression bodies  intersecting the inner thin level.

\begin{lemma}[the generalization of Appendix of \cite{JungsooKim2014}]\label{lemma-five-GHSs}
Suppose $M$ is an irreducible $3$-manifold and $(\V,\W;F)$ is a weakly reducible, unstabilized Heegaard splitting of $M$ of genus $n\geq 3$.
Let $\mathbf{H}$ be the GHS obtained by weak reduction along a weak reducing pair $(V,W)$ from $(\V,\W;F)$.
Then $\mathbf{H}$ is of the form $(\V_1,\V_2;\bar{F}_V)\cup_{\bar{F}_{VW}}(\W_1,\W_2;\bar{F}_W)$ for four compression bodies $\V_1$, $\V_2$, $\W_1$ and $\W_2$, where $\partial_-\V_2\cap \partial_-\W_1=\bar{F}_{VW}$, $\bar{F}_V\subset \V$, and $\bar{F}_W\subset \W$.
Moreover, we can classify $\mathbf{H}$ into the following five types ((\ref{lemma-3-1-b}) is divided into two types) by the shapes of $\partial_-\V_2$ and $\partial_-\W_1$ (see Figure \ref{fig-Heegaard-a}).
\begin{enumerate}[(a)]
\item Each of $\partial_-\V_2$ and $\partial_-\W_1$ is connected, where either\label{GHS-a}
	\begin{enumerate}[(i)]
	\item $V$ and $W$ are nonseparating in $\V$ and $\W$, respectively, and $\partial V\cup\partial W$ is nonseparating in $F$,\label{lemma-3-1-a-i}
	\item $V$ cuts off a handlebody $\V'$ of genus at least one from $\V$  such that $\partial W\cap\V'=\emptyset$ and $W$ is nonseparating in $\W$,\label{lemma-3-1-a-ii}
	\item $W$ cuts off a handlebody $\W'$ of genus at least one from $\W$ such that $\partial V\cap\W'=\emptyset$ and $V$ is nonseparating in $\V$, or\label{lemma-3-1-a-iii}
	\item $V$ and $W$ cut off handlebodies $\V'$ and $\W'$ of genus at least one from $\V$ and $\W$, respectively, such that $\partial V\cap \W'=\emptyset$ and $\partial W\cap\V'=\emptyset$.\label{lemma-3-1-a-iv}
	\end{enumerate}
	We call it a ``\textit{type (a) GHS}''.
\item One of $\partial_-\V_2$ and $\partial_-\W_1$ is connected and the other is disconnected, where either\label{lemma-3-1-b}
	\begin{enumerate}[(i)]
	\item ($\partial_-\V_2$ is connected)
	$V$ cuts off a compression body $\V'$ of genus at least one  from $\V$ such that $\partial_- \V'\neq\emptyset$ and $\partial W\cap\V'=\emptyset$ 	and $W$ is nonseparating in $\W$,\label{lemma-3-1-b-i}
	\item ($\partial_-\V_2$ is connected) $V$ cuts off a compression body $\V'$ of genus at least one from $\V$ and $W$ cuts off a handlebody $\W'$ of genus at least one from $\W$ such that $\partial_-\V'\neq\emptyset$, $\partial V\cap\W'=\emptyset$, and $\partial W\cap\V'=\emptyset$,\label{lemma-3-1-b-ii}
	\item ($\partial_-\W_1$ is connected) $W$ cuts off a compression body $\W'$ of genus at least one from $\W$ such that $\partial_-\W'\neq\emptyset$, $\partial V\cap\W'=\emptyset$, and $V$ is nonseparating in $\V$, or\label{lemma-3-1-b-iii}
	\item ($\partial_-\W_1$ is connected) $W$ cuts off a compression body $\W'$ of genus at least one  from $\W$ and $V$ cuts off a handlebody $\V'$ of genus at least one from $\V$ such that $\partial_-\W'\neq\emptyset$, $\partial V\cap\W'=\emptyset$, and $\partial W\cap\V'=\emptyset$.\label{lemma-3-1-b-iv}
	\end{enumerate}	
	We call it a ``\textit{type (b)-$\W$ GHS}'' for the cases (\ref{lemma-3-1-b-i}) and (\ref{lemma-3-1-b-ii}) or ``\textit{type (b)-$\V$ GHS}'' for the cases (\ref{lemma-3-1-b-iii}) and (\ref{lemma-3-1-b-iv}), respectively.
\item Each of $\partial_-\V_2$ and $\partial_-\W_1$ is disconnected but $\partial_-\V_2\cap \partial_-\W_1$ is connected, where $V$ and $W$ cut off  compression bodies $\V'$ and $\W'$ of genus at least one from $\V$ and $\W$, respectively, such that $\partial_-\V'\neq\emptyset$, $\partial_-\W'\neq\emptyset$, $\partial V\cap\W'=\emptyset$, and $\partial W\cap\V'=\emptyset$.\label{lemma-3-1-c}
We call it a ``\textit{type (c) GHS}''.
\item Each of $\partial_-\V_2$ and $\partial_-\W_1$ consists of two components and $\partial_-\V_2=\partial_-\W_1$, where $V$ and $W$ are nonseparating in $\V$ and $\W$, respectively, and $\partial V\cup\partial W$ is separating in $F$.\label{lemma-3-1-d}
We call it a ``\textit{type (d) GHS}''.
\end{enumerate}
Note that $\bar{F}_{VW}$ is disconnected only for type (d) GHS and connected otherwise.
This means if $\partial_-\V_2$ (resp $\partial_-\W_1$) is disconnected, then $\partial_-\V_2\cap\partial M\neq\emptyset$ (resp $\partial_-\W_1\cap\partial M\neq\emptyset$), excluding the case (\ref{lemma-3-1-d}).
\end{lemma}

\begin{figure}
\includegraphics[width=10cm]{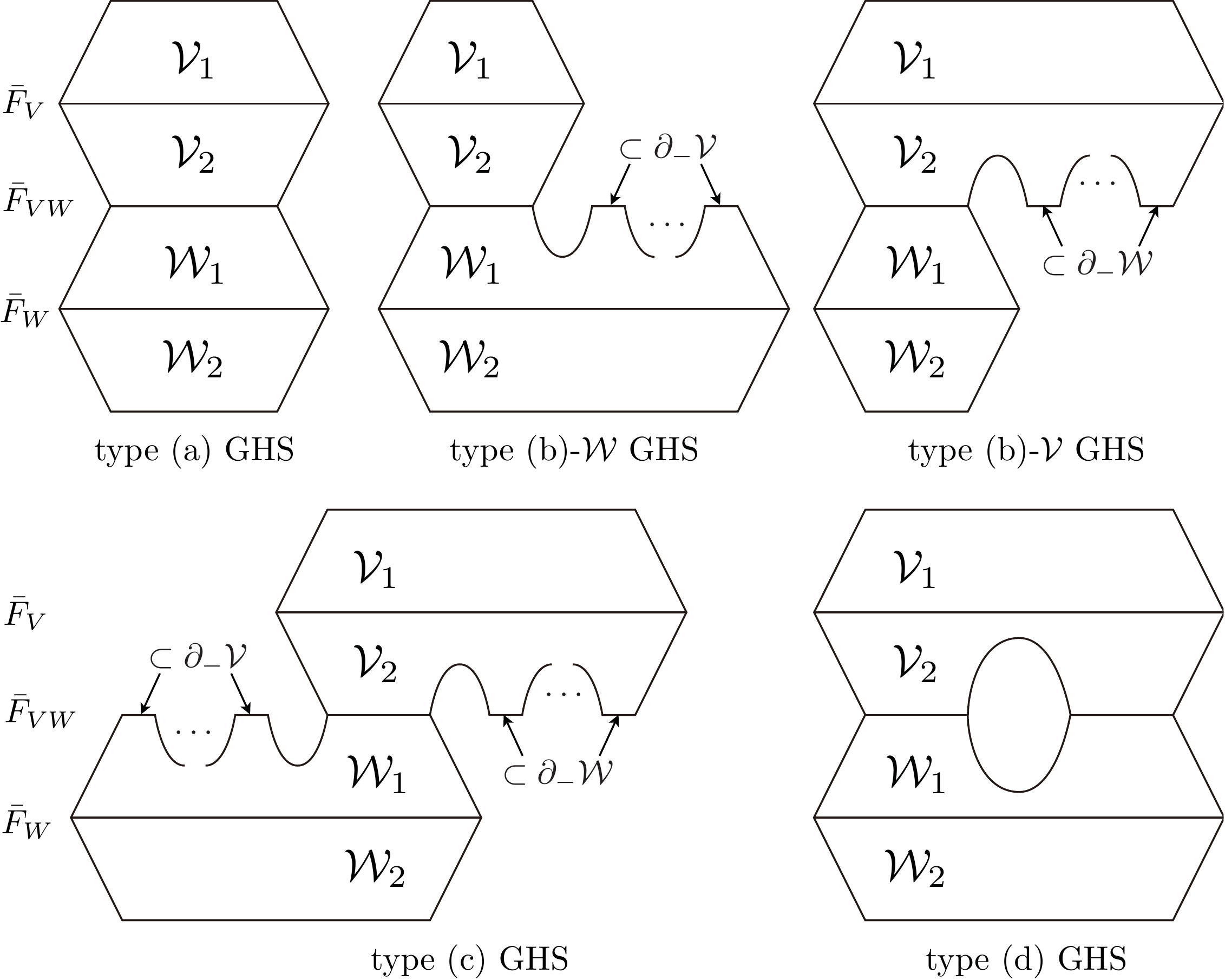}
\caption{the five types of GHSs obtained by weak reductions \label{fig-Heegaard-a}}
\end{figure}

\begin{proof}
Recall that $F_{VW}$ cannot have a $S^2$-component by Lemma \ref{lemma-2-8}.
Therefore, $\partial W$ (resp $\partial V$) is essential in the relevant component of $F_V$ (resp $F_W$).

Let $\mathbf{H}'$  be the GHS obtained by preweak reduction along $(V,W)$ from $(\V,\W;F)$.
We will denote $\Thick(\mathbf{H}')\cap \V$ and $\Thick(\mathbf{H}')\cap \W$ as $\tilde{F}_V$ and $\tilde{F}_W$, respectively, and $\Thin(\mathbf{H}')$ as $F_{VW}$.

Since $F_{VW}$ consists of at most three components, we divide the proof into three cases according to the number of components.
In each case, we will (i) do the preweak reduction along $(V,W)$ and then clean each product region in $\mathbf{H}'$, (ii) considering the remaining components of $\Thick(\mathbf{H}')$ and $\overline{\Thin}(\mathbf{H}')$ in $\mathbf{H}'$ after cleaning, confirm the thick levels $\bar{F}_V$ and $\bar{F}_{W}$ and the inner thin level $\bar{F}_{VW}$, and (iii) finally determine the shapes of the region between $\bar{F}_V$ and $\bar{F}_{VW}$ which would be $\V_2$ and that between $\bar{F}_W$ and $\bar{F}_{VW}$ which would be $\W_1$.
Note that the idea of (iii) is inspired from the proof of Theorem 3.1 of \cite{Morimoto2015}.
\\[2ex]
\Case{a} $F_{VW}$ consists of one component.

In this case, $V$ and $W$ are nonseparating in the relevant compression bodies and $\partial V\cup \partial W$ is also nonseparating in $F$.
Hence, each of $\tilde{F}_V$ and  $\tilde{F}_W$ consists of a genus $(n-1)$ surface and $F_{VW}$ consists of a genus $(n-2)$ surface.
Therefore, there is no product region in $\mathbf{H}'$ to clean, i.e. $\bar{F}_V=\tilde{F}_V$, $\bar{F}_W=\tilde{F}_W$, $\bar{F}_{VW}=F_{VW}$ and they divide $M$ into four compression bodies $\V_1$, $\V_2$, $\W_1$ and $\W_2$ with respect to the order of the surfaces.

Considering $\V_2$, it is the union of the product region $\tilde{\V}$ between $F_V$ and $\bar{F}_V$ in $\V$ and the $2$-handle corresponding to a product neighborhood of $W$ in $\W$, say $N_{\W}(W)$, where $N_{\W}(W)\cap \tilde{\V}$ is a nonseparating annulus in $F_V$ by the assumption that $\partial V\cup\partial W$ is nonseparating in $F$, i.e. $\partial_-\V_2$ is connected.
Similarly, we can use the symmetric argument for $\W_1$ and therefore $\partial_-\W_1$ is also connected.
This means $\partial_-\V_2=\partial_-\W_1=\bar{F}_{VW}$.
Therefore, we reach (\ref{lemma-3-1-a-i}) (see the first one of Figure 8 of \cite{JungsooKim2014}).\\[2ex]
\Case{b} $F_{VW}$ consists of two components.\\

\Case{b-i} $V$ is separating in $\V$ and $W$ is nonseparating in $\W$.

Let $F'_V$ be the component of $F_V$ containing $\partial W$ and $F''_V$ the other component.
Then we can assume $F''_V$ is a component of $\overline{\operatorname{Thin}}(\mathbf{H}')$ ($=F_{VW}$).
Let $\tilde{F}'_V$ and $\tilde{F}''_V$ be the components of $\operatorname{Thick}(\mathbf{H}')$ obtained by pushing $F'_V$ and $F''_V$ off into $\mathrm{int}(\V)$, respectively, and let $m\,(< n)$ be the genus of $F'_V$.
Since $\partial W$ is essential in $F'_V$ and $F_{VW}$ cannot have a $S^2$ component, we get $m\geq 2$.

Here, $\tilde{F}''_V$ and $F''_V$ cobound a product region and therefore $\tilde{F}''_V$ and $F''_V$ disappear after cleaning the relevant portion of $\mathbf{H}'$.
Hence, only $\tilde{F}'_V$ remains in $\operatorname{Thick}(\mathbf{H}')\cap \V$.

Considering Lemma \ref{lemma-pre}, $\partial W$ is nonseparating in $F'_V$ and therefore  $(F'_V)_W$ is a connected surface of genus $(m-1)$, i.e. only one component $(F'_V)_W$  remains in $\overline{\operatorname{Thin}}(\mathbf{H}')$ at this time.

Moreover, $\tilde{F}_W$ consists of a genus $(n-1)$ surface because $W$ is nonseparating in $\W$.
Therefore, considering $(m-1)<(n-1)$, the region between $(F'_V)_W$ and $\tilde{F}_W$ cannot be homeomorphic to a product.
Similarly, considering $(m-1)<m$, the region between $(F'_V)_W$ and $\tilde{F}'_V$ cannot be homeomorphic to a product.
Hence, there is no product region to clean $\mathbf{H}'$ at this moment.

Therefore, we get $\bar{F}_V=\tilde{F}'_V$, $\bar{F}_{VW}=(F'_V)_W$, and $\bar{F}_W=\tilde{F}_W$ and they divide $M$ into four compression bodies $\V_1$, $\V_2$, $\W_1$ and $\W_2$ with respect to the order of the surfaces.

Considering $\V_2$, it is the union of the product region $\tilde{\V}$ between $F'_V$ and $\tilde{F}'_V$ in $\V$ and the $2$-handle corresponding to a product neighborhood of $W$ in $\W$, say $N_{\W}(W)$, where $N_{\W}(W)\cap \tilde{\V}$ is a nonseparating annulus in $F'_V$ as we saw previously, and therefore $\partial_-\V_2$ is connected.\\

\quad \Case{b-i-1} $V$ cuts off a handlebody $\V'$ from $\V$ such that $\partial W\cap \V'=\emptyset$.\\
Let us expand $\V'$ in $\V$ until $\V'$ intersects $\bar{F}_{VW}=(F'_V)_W$ in the scar of $V$ and let $\V'_\ast$ be the resulting one.
Then we can see $\W_1$ is the union of the product region $\tilde{\W}$ between $F_W$ and $\tilde{F}_W$ in $\W$ and $\V'_\ast$ along a once-punctured surface in $F$ and therefore $\V'_\ast$ does not add another  component to $\partial_-\W_1$ other than $(F'_V)_W$ itself, i.e. $\partial_-\W_1$ is connected.
Therefore, we reach (\ref{lemma-3-1-a-ii}) (see the second one of Figure 8 of \cite{JungsooKim2014}).\\

\quad\Case{b-i-2} $V$ cuts off a compression body $\V'$ from $\V$ such that $\partial_- \V'\neq\emptyset$ and $\partial W\cap\V'=\emptyset$.

Using the similar argument as in Case b-i-1, we can see $\partial_- \W_1$ is disconnected because $\V'_\ast$ adds another component to $\partial_-\W_1$ other than $(F'_V)_W$.
Therefore, we reach (\ref{lemma-3-1-b-i}).
Moreover, $\partial_-\W_1\cap\partial M\subset \partial_-\V$ (see the first one of Figure 9 of \cite{JungsooKim2014}).\\

\Case{b-ii} $V$ is nonseparating in $\V$ and $W$ is separating in $\W$.

In this case, we use the symmetric argument of Case b-i and therefore $\bar{F}_{VW}$ is connected and $\partial_-\W_1$ is connected.\\

\quad\Case{b-ii-1} $W$ cuts off a handlebody $\W'$ from $\W$ such that $\partial V\cap\W'=\emptyset$.

In this case, $\partial_-\V_2$ is connected and therefore we reach (\ref{lemma-3-1-a-iii}) by using the symmetric argument of Case b-i-1.\\

\quad\Case{b-ii-2} $W$ cuts off a compression body $\W'$ such that $\partial_-\W'\neq\emptyset$ and $\partial V\cap\W'=\emptyset$.

In this case, $\partial_-\V_2$ is disconnected and therefore we reach (\ref{lemma-3-1-b-iii}) by using the symmetric argument of Case b-i-2.
Moreover, $\partial_-\V_2\cap\partial M\subset \partial_-\W$.\\

\Case{b-iii} $V$ and $W$ are nonseparating in $\V$ and $\W$, respectively, and $\partial V\cup\partial W$ is separating in $F$.

In this case, each of $F_{V}$ and $F_{W}$ consists of a genus $(n-1)$ surface and $F_{VW}$ is disconnected, where each component of $F_{VW}$ is of genus at most $(n-2)$ because $\partial V$ (resp $\partial W$) is an essential separating curve in the genus $(n-1)$ surface $F_W$ (resp $F_V$) by the assumption that $\partial V\cup\partial W$ is separating in $F$ and there is no $S^2$ component of $F_{VW}$.
Therefore, we cannot find a product region in $\mathbf{H}'$ to clean, i.e. $\bar{F}_V=\tilde{F}_V$, $\bar{F}_W=\tilde{F}_W$, $\bar{F}_{VW}=F_{VW}$ and they divide $M$ into four compression bodies $\V_1$, $\V_2$, $\W_1$ and $\W_2$ with respect to the order of the surfaces.

Let us consider $\V_2$.
Then it is the union of the product region $\tilde{\V}$ between $F_V$ and $\bar{F}_V$ in $\V$ and the $2$-handle corresponding to a product neighborhood of $W$ in $\W$, say $N_{\W}(W)$, where $N_{\W}(W)\cap \tilde{\V}$ is a separating annulus in $F_V$.
Hence, $\partial_-\V_2$  consists of exactly two components and therefore $\partial_-\V_2$ is $\bar{F}_{VW}$ itself.
By the symmetric argument, $\partial_-\W_1=\bar{F}_{VW}$.
Hence, we reach (\ref{lemma-3-1-d}) (see Figure 11 of \cite{JungsooKim2014}).\\[2ex]
\Case{c} $F_{VW}$ consists of three components.

If both $V$ and $W$ are nonseparating, then $F_{VW}$ consists of at most two components.
If one of $V$ and $W$ is nonseparating and the other is separating, say $W$ is nonseparating, then $\partial W$ must be separating in the component of $F_V$ containing $\partial W$ because $F_{VW}$ consists of three components, violating Lemma \ref{lemma-pre}.

Therefore, $V$ and $W$ are separating in $\V$ and $\W$, respectively.
Let $F'_V$ (resp $F'_W$) be the component of $F_V$ (resp $F_W$) containing $\partial W$ (resp $\partial V$) and $F''_V$ (resp $F''_W$) the other component, i.e. we can assume $F''_V$ (resp $F''_W$) is a component of $\overline{\operatorname{Thin}}(\mathbf{H}')$ ($=F_{VW}$).
Let $m_{F'_V}$ $(<n)$ (resp $m_{F'_W}$ $(<n)$) be the genus of $F'_V$ (resp $F'_W$).
Then we can see the following:
\begin{enumerate}
\item $F_{VW}$ consists of three components $F''_V$, $F''_W$ and $F''_{VW}$, where $F''_{VW}$ is the component of $F_{VW}$ having scars of both $V$ and $W$ and $F''_V$ (resp $F''_W$) has only one scar of $V$ (resp $W$).
Here, the genus of $F''_{VW}$ is 
$$n-g(F''_V)-g(F''_W)=n-(n-m_{F'_V})-(n-m_{F'_W})=m_{F'_V}+m_{F'_W}-n.$$
\item $\tilde{F}_V$ (resp $\tilde{F}_W$) consists of two components $\tilde{F}'_V$ and $\tilde{F}''_V$ (resp $\tilde{F}'_W$ and $\tilde{F}''_W$) obtained by pushing $F'_V$ and $F''_V$  off into $\mathrm{int}(\V)$ (resp $F'_W$ and $F''_W$ off into $\mathrm{int}(\W)$), respectively.
\end{enumerate}
Then $\tilde{F}''_V$ and $F''_V$ cobound a product region in $\V$ and therefore $\tilde{F}''_V$ and $F''_V$ disappear after cleaning the relevant portion of $\mathbf{H}'$.
Similarly, $\tilde{F}''_W$ and $F''_W$ disappear after cleaning the relevant portion of $\mathbf{H}'$.
Hence, only $F_{VW}''$ remains in $\overline{\operatorname{Thin}}(\mathbf{H}')$  and only $\tilde{F}'_V$ and $\tilde{F}'_W$ remain in $\operatorname{Thick}(\mathbf{H}')$ after these steps.
Considering $(m_{F'_V}+m_{F'_W}-n)-m_{F'_V}=m_{F'_W}-n<0$, we get $g(F_{VW}'')<g(\tilde{F}'_V)$ and also we get  $g(F_{VW}'')<g(\tilde{F}'_W)$ likewise.
Therefore, we conclude there is no product region in $\mathbf{H}'$ to clean at this moment, i.e. $\bar{F}_V=\tilde{F}'_V$, $\bar{F}_{VW}=F''_{VW}$, and $\bar{F}_W=\tilde{F}'_W$.

Here, we can see $V$ and $W$ cut off compression bodies $\V'$ and $\W'$ from $\V$ and $\W$, respectively, such that $\partial W\cap \V'=\emptyset$ and $\partial V\cap \W'=\emptyset$.\\

\quad\Case{c-i} $\V'$ and $\W'$ are handlebodies.

Let us expand $\V'$ (resp $\W'$) in $\V$ (resp in $\W$) until $\V'$ (resp $\W'$) intersects $\bar{F}_{VW}=F''_{VW}$ in the scar of $V$ (resp the scar of $W$) and let $\V'_\ast$ (resp $\W'_\ast$) be the resulting one.
Then $\W_1$ is the union of the product region $\tilde{\W}$ between $F'_W$ and $\bar{F}_W$ and $\V'_\ast$ along a once-punctured surface in $F$ and therefore $\V'_\ast$ does not add another  component to $\partial_-\W_1$ other than $\bar{F}_{VW}$ itself, i.e. $\partial_-\W_1$ is connected.
The symmetric argument also holds for $\V_2$ by using $\W'_\ast$ and therefore $\partial_-\V_2$ is connected.
This means $\partial_-\V_2=\partial_-\W_1=\bar{F}_{VW}$ and this leads to  (\ref{lemma-3-1-a-iv}) (see the last one of Figure 8 of \cite{JungsooKim2014}).\\

\quad\Case{c-ii} $\V'$ is a compression body such that $\partial_-\V'\neq\emptyset$ and $\W'$ is a handlebody.

Using the similar argument of Case c-i, we can see $\partial_- \V_2$ is connected.
But $\partial_-\W_1$ is disconnected because $\V'_\ast$ adds another component to $\partial_-\W_1$ other than $\bar{F}_{VW}$.
Therefore, we reach (\ref{lemma-3-1-b-ii}).
Moreover, $\partial_-\W_1\cap\partial M\subset \partial_-\V$ (see the second one of Figure 9 of \cite{JungsooKim2014}).\\

\quad\Case{c-iii} $\V'$ is a handlebody and $\W'$ is a compression body such that $\partial_-\W'\neq\emptyset$.

Using the symmetric argument of Case c-ii, we can see $\partial_-\W_1$ is connected and $\partial_- \V_2$ is disconnected.
Therefore, we reach (\ref{lemma-3-1-b-iv}).
Moreover, $\partial_-\V_2\cap\partial M\subset \partial_-\W$.\\

\quad\Case{c-iv} $\V'$ and $\W'$ are compression bodies such that $\partial_-\V'\neq\emptyset$ and $\partial_-\W'\neq\emptyset$.

Using the similar argument of the previous cases, we can see both $\partial_- \V_2$ and $\partial_-\W_1$ are disconnected.
Therefore, we reach (\ref{lemma-3-1-c}).
Moreover, $\partial_-\W_1\cap\partial M\subset \partial_-\V$ and $\partial_-\V_2\cap\partial M\subset \partial_-\W$ (see Figure 10 of \cite{JungsooKim2014}).\\

This completes the proof.
\end{proof}

As summary of the proof of Lemma \ref{lemma-five-GHSs}, we get the following corollary directly.

\begin{corollary}\label{corollary-five-GHSs}
Let $M$, $F$ and $\mathbf{H}$ be as in Lemma \ref{lemma-five-GHSs}.
Then the GHS $\mathbf{H}=(\V_1,\V_2;\bar{F}_V)\cup_{\bar{F}_{VW}}(\W_1,\W_2;\bar{F}_W)$ is obtained as follows.
\begin{enumerate}
\item The thick level $\bar{F}_V$ (resp $\bar{F}_W$) is obtained by pushing the  component of $F_V$ (resp $F_W$) containing $\partial W$ (resp $\partial V$) off slightly into $\mathrm{int}(\V)$ (resp of $\mathrm{int}(\W)$). \label{lemma-3-1-2nd-1}
\item The inner thin level $\bar{F}_{VW}$ is the union of components of $F_{VW}$ having scars of both $V$ and $W$.\label{lemma-3-1-2nd-2}
Moreover, if $\partial_-\V_2\cap\partial M\neq \emptyset$ (resp $\partial_-\W_1\cap\partial M\neq \emptyset$), then it belongs to $\partial_-\W$ (resp $\partial_-\V$). 
\end{enumerate}
\end{corollary}

\begin{lemma}\label{lemma-isotopic-same-type}
Let $M$ and $F$ be as in Lemma \ref{lemma-five-GHSs}.
If two GHSs $\mathbf{H}_1$ and $\mathbf{H}_2$ obtained by weak reductions from $(\V,\W;F)$ are isotopic in $M$, then they are of the same type in the sense of Lemma \ref{lemma-five-GHSs}.
\end{lemma}

\begin{proof}
Let $\mathbf{H}_i=(\V_1^i,\V_2^i;\bar{F}_{V_i})\cup_{\bar{F}_{V_i W_i}}(\W_1^i,\W_2^i;\bar{F}_{W_i})$, where $\V_2^i\cap\W_1^i=\bar{F}_{V_i W_i}$ for $i=1,2$.
Suppose there is an isotopy $h_t:M\to M$, $t\in[0,1]$ such that  $h_1(\mathbf{H}_1)=\mathbf{H}_2$.
Since an isotopy preserves each component of $\partial M$, we will not consider such components of $\Thin(\mathbf{H}_1)$ and $\Thin(\mathbf{H}_2)$.

If there is a collection of mutually disjoint surfaces $\mathcal{C}$ (assume each element of $\mathcal{C}$ is connected), then we can divide $\mathcal{C}$ into two collections $\mathcal{C}_1$ and $\mathcal{C}_2$ such that (i) each element of $\mathcal{C}_1$ cuts off a submanifold of $M$ intersecting only itself among the elements of $\mathcal{C}$ and missing the others  and (ii) each element of $\mathcal{C}_2$ does not.
Then we will say ``\textit{$\mathcal{C}_1$ is the type (i) subset of $\mathcal{C}$ and  $\mathcal{C}_2$ is the type (ii) subset of $\mathcal{C}$}''.
We can see the type (i) subset of $\mathcal{C}$ misses the type (ii) subset of $\mathcal{C}$ and the union of these two subsets is the entire collection $\mathcal{C}$.
Let $\mathbf{H}$ be a GHS obtained by weak reduction from $(\V,\W;F)$.
Putting  $\mathcal{C}=\Thick(\mathbf{H})\cup\overline{\Thin}(\mathbf{H})$,  each element of $\Thick(\mathbf{H})$ cuts off a compression body intersecting only itself among the elements of $\mathcal{C}$ and each element of $\overline{\Thin}(\mathbf{H})$  does not cut off such a submanifold by Lemma \ref{lemma-five-GHSs}.
This means $\Thick(\mathbf{H})$ is the type (i) subset of $\mathcal{C}$ and $\overline{\Thin}(\mathbf{H})$ is the type (ii) subset of $\mathcal{C}$.
Note that if there is an ambient isotopy $g_t$ defined on $M$, then the type (i) and type (ii) subsets of $g_t(\mathcal{C})$ are the images of the type (i) and type (ii) subsets of $\mathcal{C}$ of $g_t$, respectively.

Therefore, $h_1(\Thick(\mathbf{H}_1))$ and $h_1(\overline{\Thin}(\mathbf{H}_1))$ are the type (i) and type (ii) subsets of $h_1(\Thick(\mathbf{H}_1)\cup\overline{\Thin}(\mathbf{H}_1))$, respectively.
But the assumption that $h_1(\mathbf{H}_1)=\mathbf{H}_2$ implies $h_1(\Thick(\mathbf{H}_1)\cup\overline{\Thin}(\mathbf{H}_1))=\Thick(\mathbf{H}_2)\cup\overline{\Thin}(\mathbf{H}_2)$.
Since $\Thick(\mathbf{H}_2)$ and $\overline{\Thin}(\mathbf{H}_2)$ are the type (i) and type (ii) subsets of $\Thick(\mathbf{H}_2)\cup\overline{\Thin}(\mathbf{H}_2)$, respectively, we conclude $h_1(\Thick(\mathbf{H}_1))=\Thick(\mathbf{H}_2)$ and $h_1(\overline{\Thin}(\mathbf{H}_1))=\overline{\Thin}(\mathbf{H}_2)$.

This means $h_1(\V_2^1,\W_1^1)=(\V_2^2,\W_1^2)$ or $(\W_1^2,\V_2^2)$, i.e. at least one of $\partial_-\V_2^1\cap \partial M$ and $\partial_-\W_1^1\cap \partial M$ is not empty if and only if at least one of $\partial_-\V_2^2\cap \partial M$ and $\partial_-\W_1^2\cap \partial M$ is not empty.\\

\Case{a} $\partial_-\V_2^i\cap \partial M=\emptyset$ and $\partial_-\W_1^i\cap \partial M=\emptyset$ for $i=1,2$.

In this case, considering Lemma \ref{lemma-five-GHSs}, $\mathbf{H}_1$ and $\mathbf{H}_2$ are of type (a) or type (d).
Since $h_1(\bar{F}_{V_1 W_1})=\bar{F}_{V_2 W_2}$, the number of components of $\bar{F}_{V_1 W_1}$ is the same as that of $\bar{F}_{V_2 W_2}$.
This means both $\mathbf{H}_1$ and $\mathbf{H}_2$ are either of type (a) or of type (d).\\

\Case{b}
At least one of $\partial_-\V_2^i\cap \partial M$ and $\partial_-\W_1^i\cap \partial M$ is not empty for $i=1,2$.

In this case, $\mathbf{H}_1$ and $\mathbf{H}_2$ are of type (b)-$\W$, type (b)-$\V$ or type (c).

Without loss of generality, assume $\partial_-\V_2^1\cap \partial M\neq\emptyset$.
In this case, $\V_2^1\cap \partial M=\partial_-\V_2^1\cap \partial M\subset\partial_-\W$ by Corollary \ref{corollary-five-GHSs}.
Recall $h_1(\V_2^1)$ is either $\V_2^2$ or $\W_1^2$.
In the latter case, $h_1(\V_2^1)\cap \partial M=\W_1^2\cap \partial M\subset \partial_-\V$ by Corollary \ref{corollary-five-GHSs}, leading to a contradiction because an isotopy preserves each component of $\partial M$.
Hence, $h_1(\V_2^1)=\V_2^2$ and  $h_1(\W_1^1)=\W_1^2$.
Therefore, we conclude the shapes of $\partial_-\V_2^1$ and $\partial_-\W_1^1$ (including the property whether each of them intersects $\partial M$ or not) are the same as those of $\partial_-\V_2^2$ and $\partial_-\W_1^2$, respectively.
Hence, $\mathbf{H}_1$ and $\mathbf{H}_2$ are of the same type in the sense of Lemma \ref{lemma-five-GHSs}.

This completes the proof.
\end{proof}

Here, we introduce important ideas in the author's result \cite{JungsooKim2012} according to the context.

\begin{definition}[modification of Definition 2.12 of \cite{JungsooKim2012}]
Let $M$ and $F$ be as in Lemma \ref{lemma-five-GHSs} and recall the equivalent relation $\sim$ defined on $\mathcal{GHS}_F^\ast$ in Lemma \ref{lemma-equivalent-class}.
In a weak reducing pair for the Heegaard splitting $(\V,\W;F)$, if a disk belongs to $\V$, then we call it a \emph{$\V$-disk}.
Otherwise, we call it a \emph{$\W$-disk}.	
We call a $2$-simplex in $\D(F)$ represented by two vertices in $\DV(F)$ and one vertex in $\DW(F)$ a \textit{$\V$-face}, and also define a \textit{$\W$-face} symmetrically.
If there is a $\V$- or $\W$-face $\Delta$, then there exist exactly two weak reducing pairs in $\Delta$ and one shares a disk with the other.
Suppose  there is a $\V$-face $\Delta$ such that the GHSs obtained by weak reductions from $(\V,\W;F)$ along the two weak reducing pairs forming $\Delta$ are equivalent.
Then we call this $\V$-face a ``\textit{$\V$-face having one equivalent class}'' and also we define a ``\textit{$\W$-face having one equivalent class}'' likewise.

Let us consider a graph as follows.
\begin{enumerate}
\item We assign a vertex to each $\V$-face having one equivalent class.
\item If a $\V$-face having one equivalent class shares a weak reducing pair with another $\V$-face having one equivalent class, then we assign an edge between these two vertices in the graph.
\end{enumerate}
We call this graph the \emph{graph of $\V$-faces having one equivalent class}.
If there is a maximal subset $\varepsilon_\V$ of $\V$-faces having one equivalent class which represents a connected component of the graph of $\V$-faces having one equivalent class and the component is not an isolated vertex, then we call $\varepsilon_\V$ a \emph{$\V$-facial cluster having one equivalent class}.
Similarly, we define the \emph{graph of $\W$-faces having one equivalent class} and  a \textit{$\W$-facial cluster having one equivalent class}.
We will clarify the meaning of the phrase ``\textit{having one equivalent class}'' in Lemma \ref{lemma8}.
\end{definition}

\begin{lemma}\label{lemma-equivalent}
Let $M$ and $F$ be as in Lemma \ref{lemma-five-GHSs}.
Let $\Delta$ be a $\V$-face.
Then the GHSs obtained by weak reductions from $(\V,\W;F)$ along the two weak reducing pairs forming $\Delta$ are equivalent if and only if one $\V$-disk of $\Delta$ cuts off a solid torus $\V'$ from $\V$ and the other $\V$-disk is a meridian disk of $\V'$.
\end{lemma}

\begin{proof}
Let $\Delta=\{V_1,V_2,W\}$ for $V_1,V_2\subset\V$ and $W\subset \W$, and $\mathbf{H}_i=(\bar{F}_{V_i},\bar{F}_{V_i W},\bar{F}_W^i)$ the GHS obtained by weak reduction  along $(V_i,W)$ for $i=1,2$.

$(\Leftarrow)$ Suppose $V_1$ cuts off a solid torus $\V'$ from $\V$ and $V_2$ is a meridian disk of $\V'$ (so $V_2$ is nonseparating in $\V$).
Let $F'_{V_1}$ be the genus $(n-1)$-component of $F_{V_1}$.
Since $\partial W\cap\V'=\emptyset$ by Lemma \ref{lemma-2-8},  $F'_{V_1}$ contains $\partial W$.
This means the component of $F_W$ containing $\partial V_1$, say $F_W'$, also contains the once-punctured torus that $\partial V_1$ cuts off from $F$.
Hence, $F_W'$ also contains $\partial V_2$ and therefore we can assume $\bar{F}_{W}^1=\bar{F}_{W}^2$ by Corollary \ref{corollary-five-GHSs}.

Here, $\bar{F}_{V_1}$ and $\bar{F}_{V_2}$ are obtained by pushing $F'_{V_1}$ and $F_{V_2}$ off slightly into $\mathrm{int}(\V)$, respectively, by Corollary \ref{corollary-five-GHSs}.
Moreover, considering (\ref{lemma-2-21-2}) of Lemma \ref{lemma-2-21} and the genus of $\bar{F}_{V_2}$, $\bar{F}_{V_2}$ is isotopic to $\bar{F}_{V_1}$ in $\V$.
Indeed, considering (c) and (d) of Figure \ref{fig-2-10-2},  we can say the isotopy taking $\bar{F}_{V_2}$ into $\bar{F}_{V_1}$ pushes off the $3$-ball part in (c) into $\V-\V'$ without affecting a small neighborhood of $V_2$ and that of $W$ in $M$.
Hence, considering the component of $F_{V_2 W}$ having scars of both $V_2$ and $W$,  it is isotopic to the component of $F_{V_1 W}$ having scars of both $V_1$ and $W$ in $M$ by  the similar way sending $\bar{F}_{V_2}$ into $\bar{F}_{V_1}$.
This means $\bar{F}_{V_2 W}$ is isotopic to $\bar{F}_{V_1 W}$ in $M$ by Corollary \ref{corollary-five-GHSs}.
Therefore, we can isotope $\mathbf{H}_2$ into $\mathbf{H}_1$ by isotoping $\bar{F}_{V_2}$ into $\bar{F}_{V_1}$ first and then isotoping  $\bar{F}_{V_2 W}$ into $\bar{F}_{V_1 W}$ next.

Hence, summing up the above observations, we conclude $\mathbf{H}_2$ is equivalent to $\mathbf{H}_1$.\\

$(\Rightarrow)$ Suppose $\mathbf{H}_1$ is equivalent to $\mathbf{H}_2$.
Hence, $\bar{F}_{V_1}$ is isotopic to $\bar{F}_{V_2}$ in $\V$.\\

\Case{a} Both $V_1$ and $V_2$ are nonseparating in $\V$.

In this case, $\bar{F}_{V_1}$ and $\bar{F}_{V_2}$ are obtained by pushing  $F_{V_1}$ and $F_{V_2}$ off slightly into $\mathrm{int}(\V)$, respectively, by Corollary \ref{corollary-five-GHSs}.
Therefore, considering the assumption that $\bar{F}_{V_1}$ is isotopic to $\bar{F}_{V_2}$ in $\V$, $V_1$ is isotopic to $V_2$ in $\V$ by (\ref{lemma-2-21-1}) of Lemma \ref{lemma-2-21}, violating the assumption that $\Delta$ is a $2$-simplex.\\

\Case{b} One of $V_1$ and $V_2$ is nonseparating and the other is separating in $\V$.

Assume $V_1$ is nonseparating.
In this case, $\bar{F}_{V_1}$ is obtained by pushing  $F_{V_1}$ off slightly into $\mathrm{int}(\V)$ and $\bar{F}_{V_2}$ is obtained by pushing the component of $F_{V_2}$  containing $\partial W$ off slightly into $\mathrm{int}(\V)$ by Corollary \ref{corollary-five-GHSs}.
Since $V_1$ is nonseparating and $V_2$ is separating in $\V$, the assumption that $\bar{F}_{V_1}$ is  isotopic to $\bar{F}_{V_2}$ in $\V$ induces that $V_2$ cuts off a solid torus $\V'$ from $\V$ and $V_1$ is isotopic to a meridian disk $D$ of $\V'$ in $\V$ by (\ref{lemma-2-21-2}) of Lemma \ref{lemma-2-21} and we can assume that $D\cap V_2=\emptyset$.

Considering $V_1\cap V_2=\emptyset$, either $V_1\subset \V'$ or $V_1\cap \V'=\emptyset$.
Assume that  $V_1\cap \V'=\emptyset$, i.e. the component of $F-\partial V_2$ containing $\partial D$ is different from that containing $\partial V_1$.
Then $\partial V_1$ and $\partial D$ bound an annulus $A$ in $F$ because they are disjoint, isotopic, essential simple closed curves in $F$ (see Lemma 3.3 of \cite{Gelca2014}).
If $A\cap\partial V_2=\emptyset$, then  both $\partial V_1$ and $\partial D$ belong to one component of $F-\partial V_2$, leading to a contradiction.
Therefore, considering $(D\cup V_1)\cap V_2=\emptyset$, $\partial V_2\subset A$ and it is essential in $A$ because $V_2$ is a compressing disk for $F$.
This means $\partial V_2$ is isotopic to $\partial V_1$ in $A$ and therefore $V_2$ is  nonseparating in $\V$, violating the assumption.

Therefore, $V_1\subset \V'$.
Hence, considering Lemma \ref{lemma-pre},  $V_1$ is  a meridian disk of $\V'$.\\

\Case{c} Both $V_1$ and $V_2$ are separating in $\V$.

In this case, $\bar{F}_{V_1}$ and $\bar{F}_{V_2}$ are obtained by pushing the components of $F_{V_1}$ and $F_{V_2}$ containing $\partial W$ off slightly into $\mathrm{int}(\V)$, respectively, by Corollary \ref{corollary-five-GHSs}.
Moreover, $V_1$ and $V_2$ are mutually disjoint separating disks in $\V$.
Therefore, the assumption that $\bar{F}_{V_1}$ is isotopic to $\bar{F}_{V_2}$ in $\V$ induces $V_1$ is isotopic to $V_2$ in $\V$ by Lemma \ref{lemma-2-22}.
This violates the assumption that $\Delta$ is a $2$-simplex.\\

Therefore, we can see only Case b holds, leading to the result.

This completes the proof.
\end{proof}

Considering Lemma 2.9 of \cite{JungsooKim2013}, we get the following corollary directly from Lemma \ref{lemma-equivalent}.

\begin{corollary}\label{lemma-equivalent-genus3}
Suppose $M$ is an irreducible $3$-manifold and $(\V,\W;F)$ is a weakly reducible, unstabilized Heegaard splitting of $M$ of genus three.
Then every $\V$- or $\W$-face  has one equivalent class.
\end{corollary}

\begin{lemma}[analogue to Lemma 2.13 of \cite{JungsooKim2012}]\label{lemma8}
Let $M$ and $F$ be as in Lemma \ref{lemma-five-GHSs}.
If there are two different $\V$-faces having one equivalent class $\Delta_1=\{V_0,V_1,W\}$ and $\Delta_2=\{V_1, V_2, W\}$, then $V_1$ is nonseparating and $V_0$ and $V_2$ are separating in $\V$.
Therefore, there is a unique weak reducing pair in a $\V$-facial cluster having one equivalent class which can belong to two or more faces in the $\V$-facial cluster having one equivalent class.
This means every GHS obtained by weak reduction in a $\V$-facial cluster having one equivalent class is equivalent to the GHS obtained by using  the commonly shared weak reducing pair.
\end{lemma}

\begin{proof}
For the sake of contradiction, assume $V_1$ is separating in $F$.
Then $V_0$ and $V_2$ are meridian disks of the solid torus $\V'$ that $V_1$ cuts off from $\V$ by Lemma \ref{lemma-equivalent}, where this solid torus is uniquely determined by the assumption that the genus of $F$ is at least three.
But the uniqueness of a meridian disk of a solid torus up to isotopy forces $\Delta_1$ and $\Delta_2$ to be the same, violating the assumption.
Hence, $V_1$ is nonseparating and therefore $V_0$ and $V_2$ are separating by Lemma \ref{lemma-equivalent}.
This completes the proof.
\end{proof}

Considering Lemma \ref{lemma8}, we can say a $\V$- or $\W$-facial cluster having one equivalent class $\varepsilon$ \textit{has an equivalent class}, where it refers to the equivalent class of the GHSs obtained by weak reductions along the weak reducing pairs in $\varepsilon$.

\begin{definition}[analogue to Definition 2.14 of \cite{JungsooKim2012}]\label{definition-2-14}
Let $M$ and $F$ be as in Lemma \ref{lemma-five-GHSs}.
	Considering Lemma \ref{lemma8}, there is a unique weak reducing pair  in a $\V$-facial cluster having one equivalent class which can belongs to two or more faces in the cluster.
	We call it the \textit{center}.
	Lemma \ref{lemma-equivalent} and Lemma \ref{lemma8} induce any $\V$-disk in the cluster not belonging to the center cuts off a solid torus from $\V$, where this solid torus is uniquely determined by the disk, and the $\V$-disk  in the center is a meridian disk of the solid torus.
	The shape of a $\V$-facial cluster having one equivalent class is exactly the same as a \textit{$\V$-facial cluster} which was introduced by the author for the genus three case in \cite{JungsooKim2012} (see the third one of Figure \ref{fig-bbs}).
		Therefore, if a $\V$-face    in a $\V$-facial cluster having one equivalent class is represented by two weak reducing pairs, then one weak reducing pair is the center.
\end{definition}

\begin{lemma}[analogue to Lemma 2.15 of \cite{JungsooKim2012}]\label{lemma-DorE} 
Let $M$ and $F$ be as in Lemma \ref{lemma-five-GHSs}.
Every $\V$-face having one equivalent class belongs to a uniquely determined $\V$-facial cluster having one equivalent class. 
Moreover, every $\V$-facial cluster having one equivalent class has infinitely many weak reducing pairs.
\end{lemma}

\begin{proof}
A $\V$-face $\Delta$ has one equivalent class if and only if a $\V$-disk of $\Delta$ cuts off a solid torus $\V'$ from $\V$ and the other $\V$-disk is a meridian disk of $\V'$ by Lemma \ref{lemma-equivalent}.
Therefore, the proof is essentially the same as that of Lemma 2.15 of \cite{JungsooKim2012} except for the uniqueness of the $\V$-facial cluster having one equivalent class which contains $\Delta$.

Let us consider the uniqueness.
Suppose a $\V$-face $\Delta$ having one equivalent class belongs to two $\V$-facial clusters having one equivalent class, say  $\varepsilon_1$ and $\varepsilon_2$.
Then there is a weak reducing pair in $\Delta$ containing a nonseparating $\V$-disk by Lemma \ref{lemma-equivalent} and it must be the center of $\varepsilon_i$ for $i=1,2$ by Definition \ref{definition-2-14}.
This means all $\V$-faces of $\varepsilon_1$ and $\varepsilon_2$ share the common center, i.e. $\varepsilon_1$ and $\varepsilon_2$ correspond to the same connected component of the graph of $\V$-faces having one equivalent class and therefore we conclude $\varepsilon_1=\varepsilon_2$.

This completes the proof.
\end{proof}

The proof of Lemma \ref{lemma-DorE} induces the following corollary directly.

\begin{corollary}\label{corollary-DorE} 
Let $M$ and $F$ be as in Lemma \ref{lemma-five-GHSs}.
For a weak reducing pair, there is at most one $\V$-facial cluster having one equivalent class whose center is the weak reducing pair.
\end{corollary}

\begin{lemma}\label{lemma-cluster-A}
Let $M$ and $F$ be as in Lemma \ref{lemma-five-GHSs}.
For a $\V$-facial cluster having one equivalent class, say  $\varepsilon_\V$, if $V_1$ and $V_2$ are different $\V$-disks in $\varepsilon_\V$ not contained in the center, then there is no $1$-simplex in $\D(F)$ between them.
\end{lemma}

\begin{proof}
Let $(\bar{V},\bar{W})$ be the center of $\varepsilon_\V$.
For the sake of contradiction, assume that there is a $1$-simplex between $V_1$ and $V_2$ in $\D(F)$.
%Since there are already two $1$-simplices $\{\bar{V},V_1\}$ and $\{\bar{V},V_2\}$ in $\varepsilon_\V$,  we get a $2$-simplex in $\D(F)$ spanned by $\bar{V}$, $V_1$, and $V_2$. 
Considering Definition \ref{definition-2-14}, $\partial V_1$ and $\partial V_2$ cut off once-punctured tori $T_1$ and $T_2$ from $F$, respectively, where $T_i$ is the closure of the component of $F-\partial V_i$ containing $\partial \bar{V}$ for $i=1,2$.
Moreover, considering the assumption that $\partial V_1\cap\partial V_2=\emptyset$, one of $T_1$ and $T_2$ contains the other, say $T_1\subset T_2$.
Here, $\partial V_1$ divides $T_2$ into two pieces, where one is a genus one surface and the other is a planar surface.
Since $V_1$  is a compressing disk of $F$, $\partial V_1$ cannot cut off a disk from $T_2$.
This means  $\partial V_1$ cuts off an annulus from $T_2$ such that the other boundary component is $\partial V_2$, i.e. $\partial V_1$ is isotopic to $\partial V_2$ in $F$.
Therefore, $V_1$ is isotopic to $V_2$ in $\V$ by Lemma \ref{lemma-isotopic-bd}, leading to a contradiction.
This completes the proof.
\end{proof}

\begin{corollary}\label{corollary-cluster-A}
Let $M$ and $F$ be as in Lemma \ref{lemma-five-GHSs}.
For a $\V$-facial cluster having one equivalent class, say $\varepsilon_\V$, there is no $n$-simplex in $\D(F)$ spanned by vertices of $\varepsilon_\V$ for $n \geq 3$.
\end{corollary}

\begin{proof}
Let $(\bar{V},\bar{W})$ be the center of $\varepsilon_\V$.
For the sake of contradiction, assume that there is an $n$-simplex $\Delta$ in $\D(F)$ spanned by vertices of $\varepsilon_\V$ for $n\geq 3$.
Then we can obtain an at least $(n-2)$-subsimplex $\delta$ by removing $\bar{V}$ and $\bar{W}$ from $\Delta$, where $n-2\geq 1$.
Here, $\delta$ consists of $\V$-disks of $\varepsilon_\V$ other than $\bar{V}$, violating Lemma \ref{lemma-cluster-A}.
\end{proof}

\begin{lemma}\label{lemma-cluster-all-simplices}
Let $M$ and $F$ be as in Lemma \ref{lemma-five-GHSs}.
For a $\V$-facial cluster having one equivalent class, say $\varepsilon_\V$, the union of all simplices of $\D(F)$ spanned by the vertices of $\varepsilon_\V$ is $\varepsilon_\V$ itself.
\end{lemma}

\begin{proof}
Considering that each face of $\varepsilon_\V$ is spanned by three vertices of $\varepsilon_\V$, $\varepsilon_\V$ is contained in the union of all simplices of $\D(F)$ spanned by the vertices of $\varepsilon_\V$.
Hence, we will prove that the union belongs to $\varepsilon_\V$.
It is sufficient to show that every simplex of $\D(F)$ spanned by the vertices of $\varepsilon_\V$ is contained in $\varepsilon_\V$.

Let $(\bar{V},\bar{W})$ be the center of $\varepsilon_\V$.
Suppose there is a $1$-simplex $\sigma$ spanned by two vertices of $\varepsilon_\V$ which is not contained in $\varepsilon_\V$.
If there is a vertex of $\varepsilon_\V$, then it is either (i) $\bar{W}$, (ii) $\bar{V}$, or (iii) another $\V$-disk other than $\bar{V}$.
Considering all combinations of two different vertices of $\varepsilon_\V$ and the shape of $\varepsilon_\V$ (see the third one of Figure \ref{fig-bbs}), if they are of different types among the previous three types, then $\sigma\subset\varepsilon_\V$, violating the assumption.
This means $\sigma$ must connect two $\V$-disks other than $\bar{V}$, violating Lemma \ref{lemma-cluster-A}.
Hence, every $1$-simplex spanned by two vertices of $\varepsilon_\V$ must belong to $\varepsilon_\V$.

Suppose there is a $2$-simplex $\Delta$ spanned by three vertices of $\varepsilon_\V$ which is not contained in $\varepsilon_\V$.
Considering all combinations of three vertices of $\Delta$ similarly as in the previous paragraph, if three vertices are of mutually different types, then $\Delta$ belongs to $\varepsilon_\V$, violating the assumption of $\Delta$.
Hence, there must be two $\V$-disks other than $\bar{V}$ in $\Delta$, violating Lemma \ref{lemma-cluster-A}.
Therefore, every $2$-simplex spanned by three vertices of $\varepsilon_\V$ must belong to $\varepsilon_\V$.

Since there is no simplex of dimension at least three spanned by vertices of $\varepsilon_\V$ by Corollary \ref{corollary-cluster-A}, we don't need to consider such simplices.

This completes the proof.
\end{proof}

From now on, we will find a collection of connected subsets of $\DVW(F)$  such that (i) the GHSs obtained by weak reductions along the weak reducing pairs in such a subset are all equivalent and (ii) every weak reducing pair belongs to exactly one subset among such subsets (see Lemma \ref{lemma-BB-char}) in Definition \ref{definition-3-3-detail}, Definition \ref{definition-3-5-detail}, Definition \ref{definition-3-5-ii-detail} and Definition \ref{definition-3-6}. 

We will subdivide each of the five types of GHSs obtained by weak reductions in Lemma \ref{lemma-five-GHSs} as in the following definition.

\begin{definition}\label{definition-GHSs}
Let $M$ and $F$ be as in Lemma \ref{lemma-five-GHSs}.
Let $\mathbf{H}=(\V_1,\V_2;\bar{F}_V)\cup_{\bar{F}_{VW}}(\W_1,\W_2;\bar{F}_W)$ be the GHS obtained by weak reduction along a weak reducing pair $(V,W)$ from  $(\V,\W;F)$, where $\partial_-\V_2\cap \partial_-\W_1=\bar{F}_{VW}$, $\bar{F}_V\subset \V$, and $\bar{F}_W\subset \W$.
We subdivide each of the five types of GHSs of Lemma \ref{lemma-five-GHSs}, where we will differentiate (1) the case where $V$ or $W$ is nonseparating or it cuts off a solid torus from the relevant compression body and (2) the other cases in each subcase.

Note that if $V$ cuts off a solid torus $\V'$ from $\V$, then $\V'$ must miss $\partial W$ by Lemma \ref{lemma-2-8}.
Likewise, if $V$ cuts off a compression body $\V''$ containing $\partial W$ from $\V$, then $\V''$ is not a solid torus by Lemma \ref{lemma-2-8}.
The symmetric arguments also hold for $W$.

\begin{enumerate}[(a)]
\item Case: $\mathbf{H}$ is of type (a).
	\begin{enumerate}[(i)]
	\item If 
		\begin{enumerate}[(I)]
		\item both $V$ and $W$ are nonseparating or
		\item if there is a separating one, then it cuts off a solid torus from the relevant compression body, 
		\end{enumerate}
		then we say $\mathbf{H}$ is of \textit{type (a)-(i)}.
	\item (I) One of $V$ and $W$ is nonseparating or it cuts off a solid torus from the relevant compression body and (II) the other is separating and it does not cut off a solid torus from the relevant compression body.
	\begin{enumerate}[(i)]
	\item[($\W$)] If $V$ is separating and it does not cut off a solid torus from $\V$, i.e. it cuts off a handlebody $\V'$ of genus at least two from $\V$ such that $\partial W\cap \V'=\emptyset$ by Lemma \ref{lemma-five-GHSs}, then we say $\mathbf{H}$ is of \textit{type (a)-(ii)-$\W$}.
	\item[($\V$)] If $W$ is separating and it does not cut off a solid torus from $\W$, i.e. it cuts off a handlebody $\W'$ of genus at least two from $\W$ such that $\partial V\cap \W'=\emptyset$ by Lemma \ref{lemma-five-GHSs}, then we say $\mathbf{H}$ is of \textit{type (a)-(ii)-$\V$}.
	\end{enumerate}
	\item If (I) both $V$ and $W$ are separating and (II) each of them does not cut off a solid torus from the relevant compression body, i.e. they cut off  handlebodies $\V'$ and $\W'$ of genera at least two from $\V$ and $\W$, respectively, such that $\partial W\cap\V'=\emptyset$ and $\partial V\cap \W'=\emptyset$ by Lemma \ref{lemma-five-GHSs}, then we say $\mathbf{H}$ is of \textit{type (a)-(iii)}.	
	\end{enumerate}
\item Case: $\mathbf{H}$ is of type (b)-$\W$ or type (b)-$\V$.
	\begin{enumerate}[(1)]
	\item[($\W$)] Case: $\mathbf{H}$ is of type (b)-$\W$, i.e. $V$ cuts off a compression body $\V'$ of genus at least one such that $\partial_-\V'\neq\emptyset$ and $\partial W\cap\V'=\emptyset$ by Lemma \ref{lemma-five-GHSs}.
	
	If $W$ is nonseparating or it cuts off a solid torus from $\W$, then we say $\mathbf{H}$ is of \textit{type (b)-$\W$-(i)}.
	Otherwise, $W$ cuts off a handlebody $\W'$ of genus at least two from $\W$ such that $\partial V\cap \W'=\emptyset$ by Lemma \ref{lemma-five-GHSs} and we say $\mathbf{H}$ is of \textit{type (b)-$\W$-(ii)}.
	
	\item[($\V$)] Case: $\mathbf{H}$ is of type (b)-$\V$, i.e. $W$ cuts off a compression body $\W'$ of genus at least one such that $\partial_-\W'\neq\emptyset$ and $\partial V\cap\W'=\emptyset$ by Lemma \ref{lemma-five-GHSs}.
	
		If $V$ is nonseparating or it cuts off a solid torus from $\V$, then we say $\mathbf{H}$ is of \textit{type (b)-$\V$-(i)}.
	Otherwise, $V$ cuts off a handlebody $\V'$ of genus at least two from $\V$ such that $\partial W\cap \V'=\emptyset$ by Lemma \ref{lemma-five-GHSs} and we say $\mathbf{H}$ is of \textit{type (b)-$\V$-(ii)}.
	\end{enumerate}
\item Case: $\mathbf{H}$ is of type (c).

We do not subdivide this case.
\item Case: $\mathbf{H}$ is of type (d).

We do not subdivide this case.
\end{enumerate}
\end{definition}

For a given weak reducing pair $(V,W)$, we can see the GHS obtained by weak reduction along $(V,W)$ belongs to exactly one subcase among the ten types of Definition \ref{definition-GHSs}.

\begin{lemma}\label{lemma-types}
If two GHSs $\mathbf{H}_1=(\bar{F}_{V_1},\bar{F}_{V_1 W_1},\bar{F}_{W_1})$ and $\mathbf{H}_2=(\bar{F}_{V_2},\bar{F}_{V_2 W_2},\bar{F}_{W_2})$ obtained by weak reductions from $(\V,\W;F)$ are equivalent, then they are of the same type in the sense of Definition \ref{definition-GHSs}.
This induces that the GHSs obtained by weak reductions along the weak reducing pairs in a given $\V$- or $\W$-facial cluster having one equivalent class are of the same type in the sense of Definition \ref{definition-GHSs}.
\end{lemma}

\begin{proof}
Since $\mathbf{H}_1$ is isotopic to $\mathbf{H}_2$ in $M$, $\bar{F}_{V_1 W_1}$ is isotopic to $\bar{F}_{V_2 W_2}$ in $M$ as in the proof of Lemma \ref{lemma-isotopic-same-type}.
This means if one of $\mathbf{H}_1$ and $\mathbf{H}_2$ is of type (d), where this is the only case where the inner thin level is disconnected, then the other is also of type (d), leading to the result.
Hence, assume both $\mathbf{H}_1$ and $\mathbf{H}_2$ are not of type (d), i.e. both $\bar{F}_{V_1 W_1}$ and $\bar{F}_{V_2 W_2}$ are connected.

\begin{claim}\label{claim-1}
\begin{enumerate}[(i)]
\item $V_1$ is nonseparating or cuts off a solid torus missing $\partial W_1$ from $\V$ if and only if $V_2$ is nonseparating or cuts off a solid torus missing $\partial W_2$  from $\V$.\label{Claim-aaa-i}
\item $V_1$ cuts off a handlebody of genus at least two missing $\partial W_1$ from $\V$ if and only if $V_2$ cuts off a handlebody of genus at least two missing $\partial W_2$ from $\V$.
\item $V_1$ cuts off a compression body with nonempty negative boundary missing $\partial W_1$ from $\V$ if and only if $V_2$ cuts off a compression body with nonempty negative boundary missing $\partial W_2$ from $\V$.
\end{enumerate}
\end{claim}

\begin{proofN}{Claim \ref{claim-1}}
 
 \Case{a} At least one of $V_1$ and $V_2$ is nonseparating in $\V$.

In this case, we restrict our attention to (\ref{Claim-aaa-i}) of Claim \ref{claim-1} which is the only statement dealing with at least one nonseparating disk.

If both $V_1$ and $V_2$ are nonseparating, then there is nothing to prove.

Hence, assume $V_i$ is nonseparating and $V_j$ is separating in $\V$ for $i\neq j$.
Since $\bar{F}_{V_j}$ is isotopic to $\bar{F}_{V_i}$ in $\V$ by the assumption that $\mathbf{H}_j$ is equivalent to $\mathbf{H}_i$, Corollary \ref{corollary-five-GHSs} and (\ref{lemma-2-21-2}) of Lemma \ref{lemma-2-21} induce $V_j$ cuts off a solid torus $\V'$ from $\V$ and $\V'$ misses $\partial W_j$ by Lemma \ref{lemma-2-8}.
This means that we have proved  (\ref{Claim-aaa-i}) of Claim \ref{claim-1} except the case where both $V_1$ and $V_2$ are separating.

This completes the proof of  Case a.\\

 \Case{b} Both $V_1$ and $V_2$ are separating in $\V$.
 
By Corollary \ref{corollary-five-GHSs}, $\bar{F}_{V_i}$ is obtained by pushing the component of $F_{V_i}$ containing $\partial W_i$ slightly into $\mathrm{int}(\V)$ for $i=1,2$.
Therefore, there is a region cobounded by $V_i$, $\bar{F}_{V_i}$ and a component of $F-\partial V_i$ in $\V$, say $\tilde{\V}_i$, which is homeomorphic to $\bar{F}_{V_i}\times I$ such that  the $0$-level is $\bar{F}_{V_i}$ and the $1$-level contains $\partial W_i$ for $i=1,2$.
Let $\tilde{\V}'_1$ be the closure of the component of $\V-\bar{F}_{V_1}$ missing $\partial_+\V$  and $\tilde{\V}''_1$ the closure of the component of $\V-\tilde{\V}_1$ intersecting $\partial_+\V$.
Since $\mathbf{H}_1$ and $\mathbf{H}_2$ are equivalent, there is an isotopy $h_t:\V\to\V$, $t\in[0,1]$ such that $h_1(\bar{F}_{V_2})=\bar{F}_{V_1}$ and therefore $h_1(\tilde{\V}_2)$ is homeomorphic to $\bar{F}_{V_1}\times I$  whose $0$-level is $\bar{F}_{V_1}$.
Here, $\operatorname{cl}(\V-h_1(\tilde{\V}_2))$ consists of two components, where  the component  missing $\partial_+\V$ is $\tilde{\V}_1'$ and the other component, say $\tilde{\V}''_2$, intersects $\partial_+\V$.

Then we can see
$$\V=\tilde{\V}''_1\cup_{V_1} (\tilde{\V}_1\cup_{\bar{F}_{V_1}} \tilde{\V}'_1)=\tilde{\V}''_2\cup_{h_1(V_2)} (h_1(\tilde{\V}_2)\cup_{\bar{F}_{V_1}} \tilde{\V}'_1),$$
where $\partial W_1\cap\tilde{\V}''_1=\emptyset$ and $h_1(\partial W_2)\cap\tilde{\V}''_2=\emptyset$, and both $\tilde{\V}''_1$ and $\tilde{\V}''_2$ are compression bodies such that $\partial_-\tilde{\V}''_1\subset\partial_-\V$ and $\partial_-\tilde{\V}''_2\subset\partial_-\V$ by Lemma 1.3 of \cite{ScharlemannThompson1993} because we get $\tilde{\V}''_1$ and $\tilde{\V}''_2$  by cutting $\V$ off  along $V_1$ and $h_1(V_2)$, respectively.
Moreover, $\tilde{\V}_1\cup_{\bar{F}_{V_1}} \tilde{\V}'_1$ and $h_1(\tilde{\V}_2)\cup_{\bar{F}_{V_1}} \tilde{\V}'_1$ are compression bodies such that one is homeomorphic to the other and they share the common negative boundary.
This means (1) the genus of $\tilde{\V}''_1$ is the same as that of $\tilde{\V}''_2$ and (2) $\tilde{\V}''_1$ is a handlebody if and only if $\tilde{\V}''_2$ is a handlebody.
That is, $\tilde{\V}''_1$ is 
\begin{enumerate}[(i)]
\item a solid torus ($=V_1$ cuts off a solid torus missing $\partial W_1$ from $\V$),
\item a handlebody of genus at least two ($=V_1$ cuts off a handlebody of genus at least two missing $\partial W_1$ from $\V$), or 
\item a compression body with nonempty negative boundary ($=V_1$ cuts off a compression body with nonempty negative boundary missing $\partial W_1$ from $\V$)
\end{enumerate}
 if and only if $\tilde{\V}''_2$ is 
\begin{enumerate}[(i)]
\item a solid torus ($=h_1(V_2)$ cuts off a solid torus missing $h_1(\partial W_2)$ from $\V$),
\item a handlebody of genus at least two ($=h_1(V_2)$ cuts off a handlebody of genus at least two missing $h_1(\partial W_2)$ from $\V$), or 
\item a compression body with nonempty negative boundary ($=h_1(V_2)$ cuts off a compression body with nonempty negative boundary missing $h_1(\partial W_2)$ from $\V$),
\end{enumerate}
respectively. 
Since $h_t$ is an isotopy defined on $\V$, considering the latter part of the previous ``if and only if'' statement, it is equivalent to the statement that $V_2$ cuts off (i) a solid torus missing $\partial W_2$, (ii) a handlebody of genus at least two missing $\partial W_2$, or (iii) a compression body with nonempty negative boundary missing $\partial W_2$ from $\V$, respectively.
This completes the proof of Case b.\\

This completes the proof of Claim \ref{claim-1}.
\end{proofN}

Similarly, we get Claim \ref{claim-2}.

\begin{claim}\label{claim-2}
\begin{enumerate}[(i)]
\item $W_1$ is nonseparating or cuts off a solid torus missing $\partial V_1$ from $\W$ if and only if $W_2$ is nonseparating or cuts off a solid torus missing $\partial V_2$ from $\W$.
\item $W_1$ cuts off a handlebody of genus at least two missing $\partial V_1$ from $\W$ if and only if $W_2$ cuts off a handlebody of genus at least two missing $\partial V_2$ from $\W$.
\item $W_1$ cuts off a compression body with nonempty negative boundary missing $\partial V_1$ from $\W$ if and only if $W_2$ cuts off a compression body with nonempty negative boundary missing $\partial V_2$ from $\W$.
\end{enumerate}
\end{claim}

Recall that $\mathbf{H}_1$ and $\mathbf{H}_2$ are of the same type in the sense of Lemma \ref{lemma-five-GHSs} by Lemma \ref{lemma-isotopic-same-type}. 
Hence, considering each case of Definition \ref{definition-GHSs} for $\mathbf{H}_1$ and $\mathbf{H}_2$ awaring of Claim \ref{claim-1} and Claim \ref{claim-2} (excluding type (d)), we conclude $\mathbf{H}_1$ and $\mathbf{H}_2$ are of the same type in the sense of Definition \ref{definition-GHSs}.

This completes the proof.
\end{proof}

Considering Lemma \ref{lemma-types}, we can say an equivalent class of $\mathcal{GHS}_F^\ast/\sim$ is of one type among the ten types in Definition \ref{definition-GHSs}.
Moreover, we can say a $\V$- or $\W$-facial cluster having one equivalent class $\varepsilon$ is of one type among these ten types by the second statement of Lemma \ref{lemma-types} (but we will exclude type (a)-(iii), type (b)-$\W$-(ii), type (b)-$\V$-(ii), type (c) and type (d) cases later).

\begin{lemma}\label{lemma-restrict} 
Let $M$ and $F$ be as in Lemma \ref{lemma-five-GHSs}.
For a $k$-simplex $\Sigma$ in $\DVW(F)$ having at least one weak reducing pair, if the GHSs obtained by weak reductions along the weak reducing pairs in $\Sigma$ are all equivalent, then $k\leq 3$.
Moreover, if $k=3$, then $\Sigma=\{V_1, V_2, W_1, W_2\}$, where $V_2$ and $W_2$ cut off solid tori $\V'$ and $\W'$ from $\V$ and $\W$, respectively, and  $V_1$ and $W_1$ are meridian disks of $\V'$ and $\W'$, respectively.
This means if $k\geq 4$, then  there are at least two weak reducing pairs in $\Sigma$ such that they give nonequivalent GHSs after weak reductions.
\end{lemma}

\begin{proof}
Suppose the GHSs obtained by weak reductions along the weak reducing pairs in the $k$-simplex $\Sigma$ are all equivalent, where $\Sigma=\{V_1, \cdots, V_m, W_1, \cdots, W_n\}$, where $V_i\subset \V$ for $1\leq i \leq m$ $(1\leq m)$, $W_i\subset \W$ for $1\leq i \leq n$ $(1\leq n)$, and $m+n=k+1$.
Assume $m>2$.
Then each of the three $\V$-faces $\Delta_1=\{V_1,V_2,W_1\}$, $\Delta_2=\{V_2,V_3,W_1\} $ and $\Delta_3=\{V_1,V_3,W_1\}$ has one equivalent class by the assumption.
Considering $\Delta_1$ and $\Delta_2$, $V_2$ is nonseparating and $V_1$ and $V_3$ are separating in $\V$ by Lemma \ref{lemma8}.
But considering $\Delta_2$ and $\Delta_3$, $V_3$ is nonseparating in $\V$ by Lemma \ref{lemma8}, leading to a contradiction.
Hence, we get $m\leq 2$.
Similarly, we get $n\leq 2$ and therefore $k+1=m+n\leq 4$, i.e. $k\leq 3$.
The second statement is obtained by applying Lemma \ref{lemma-equivalent} to the $\V$-face $\{V_1,V_2,W_1\}$ and the $\W$-face $\{V_1,W_1,W_2\}$, where each of them has one equivalent class by the assumption.
This completes the proof.
\end{proof}

\begin{lemma}\label{lemma-3-2}
Let $M$ and $F$ be as in Lemma \ref{lemma-five-GHSs}.
Let $\varepsilon_\V$ and $\varepsilon_\W$ be a $\V$-facial cluster having one equivalent class and a $\W$-facial cluster having one equivalent class such that they share the common center.
Then $\varepsilon_\V$ and $\varepsilon_\W$ have  the same equivalent class of type (a)-(i).
\end{lemma}

\begin{proof}
Considering Lemma \ref{lemma8}, the GHSs obtained by weak reductions along the weak reducing pairs in $\varepsilon_\V\cup\varepsilon_\W$ are all equivalent to the GHS corresponding to the common center.
Moreover, they are of the same type in the sense of  Definition \ref{definition-GHSs} by Lemma \ref{lemma-types}.
Since $\varepsilon_\V$ is a $\V$-facial cluster having one equivalent class and $\varepsilon_\W$ is a $\W$-facial cluster having one equivalent class, the common center consists of nonseparating disks by Definition \ref{definition-2-14}, i.e.  the center gives a GHS of type (a)-(i) or type (d) after weak reduction by Definition \ref{definition-GHSs}.
But every weak reducing pair in $\varepsilon_\V\cup\varepsilon_\W$ other than the common center has a separating disk by Definition \ref{definition-2-14} and therefore it cannot give a type (d) GHS after weak reduction by Lemma  \ref{lemma-five-GHSs}.
Hence, we conclude both $\varepsilon_\V$ and $\varepsilon_\W$ are of type (a)-(i).
This completes the proof.
\end{proof}

\begin{lemma}
\label{lemma-type-a-i}
Let $M$ and $F$ be as in Lemma \ref{lemma-five-GHSs} and let $(V,W)$ be a weak reducing pair such that the GHS obtained by weak reduction along $(V,W)$ is of type (a)-(i).
Then there are a naturally determined $\V$-facial cluster having one equivalent class, say  $\varepsilon_\V$, and a naturally determined $\W$-facial cluster having one equivalent class, say $\varepsilon_\W$, satisfying the following:
\begin{enumerate}
\item $\varepsilon_\V$ and $\varepsilon_\W$  share the common center.\label{lemma-type-a-i-1}
\item Let $\Sigma$ be the union of all simplices of $\D(F)$ spanned by the vertices of $\varepsilon_\V\cup\varepsilon_\W$.
Then $\Sigma$ contains $(V,W)$.\label{lemma-type-a-i-2}
\item The GHSs obtained by weak reductions along the weak reducing pairs in $\Sigma$ are all equivalent.\label{lemma-type-a-i-3}
\end{enumerate} 
\end{lemma}

\begin{proof}
Take $\bar{V}$ (resp $\bar{W}$) as $V$ (resp $W$) itself if $V$ (resp $W$) is nonseparating.
If $V$ is separating, then it cuts off a solid torus $\V'$ from $\V$ such that $\partial W\cap\V'=\emptyset$ by Definition \ref{definition-GHSs} and therefore we can choose a meridian disk $\bar{V}$ of $\V'$ missing $V\cup W$.
Similarly, if $W$ is separating, then it cuts off a solid torus $\W'$ from $\W$ such that $\partial V\cap\W'=\emptyset$ and therefore we can choose a meridian disk $\bar{W}$ of $\W'$ missing $V\cup W$.
If both $V$ and $W$ are separating, then $\V'\cap\W'=\emptyset$ because the assumption that $\partial W\cap \V'=\emptyset$ and $\partial V\cap \W'=\emptyset$ in Definition \ref{definition-GHSs} means the two once-punctured tori cut by $\partial W$ and $\partial V$ from $F$ are pairwise disjoint, i.e. the four disks $V$, $\bar{V}$, $\bar{W}$ and $W$ are pairwise disjoint.

Hence, we can find a weak reducing pair $(\bar{V},\bar{W})$ such that either
\begin{enumerate}
\item $(V,W)$ is $(\bar{V},\bar{W})$ itself,\label{definition-3-3-1}
\item $(V,W)$ is contained in a $\V$-face $\Delta_\V=\{V,\bar{V},\bar{W}=W\}$ and it has one equivalent class by Lemma \ref{lemma-equivalent}\label{definition-3-3-2},
\item $(V,W)$ is contained in a $\W$-face $\Delta_\W=\{\bar{V}=V,\bar{W},W\}$ and it has one equivalent class by Lemma \ref{lemma-equivalent}\label{definition-3-3-3}, or
\item $(V,W)$ is contained in a $3$-simplex $\Sigma_{VW}=\{V,\bar{V},\bar{W},W\}$ such that $\bar{V}$ and $\bar{W}$ are meridian disks of the solid tori that $V$ and $W$ cut off from $\V$ and $\W$, respectively. \label{definition-3-3-4}
\end{enumerate}
Here, $(\bar{V},\bar{W})$ is uniquely determined in $\D(F)$ by $(V,W)$ because the meridian disk of a solid torus is unique up to isotopy.

From now on, if we say a $3$-simplex $\Sigma_{VW}$ \textit{satisfies the condition ($\ast$)} in the proof of this lemma, then it will mean that $V\subset \V$ and $W\subset \W$ cut off solid tori from $\V$ and $\W$, respectively, and the other two disks of $\Sigma_{VW}$ are  meridian disks of the solid tori, respectively, as in the case (\ref{definition-3-3-4}), where it is the same condition for the $3$-simplex $\Sigma$ in the statement of Lemma \ref{lemma-restrict}.

\begin{claim}\label{claim-AB}
$(V,W)$ is contained in a $3$-simplex of the form $$\Sigma_{V'W'}=\{V',\bar{V},\bar{W},W'\}$$ satisfying the condition ($\ast$) in any case as well as the case (\ref{definition-3-3-4}).
\end{claim}
\begin{proofN}{Claim \ref{claim-AB}}
Let us consider the case (\ref{definition-3-3-2}).
Then (i) $\partial V$ divides $F$ into a once-punctured genus at least two surface $F'$ and a once-punctured torus $F''$ and (ii) $\partial \bar{V}\subset F''$ and $\partial\bar{W}\subset F'$ because $\partial\bar{W}$ $(=\partial W)$ misses the solid torus $\V'$.
Here, $\partial\bar{W}$ is nonseparating in $F'$ by Lemma \ref{lemma-pre}.
Hence, we can find a band-sum of two parallel copies of $\bar{W}$ in $\W$, say $W'$,  such that $\partial W'\subset F'$ (if we use the expression ``\textit{two parallel copies of $\bar{W}$ in $\W$}'', then we will assume that there is a product neighborhood $\bar{W}\times I$ of $\bar{W}$ in $\W$ such that the disks  $\bar{W}\times \{0,1\}$ are the two parallel copies of $\bar{W}$ and $\bar{W}$ is  $\bar{W}\times \{1/2\}$ and see Figure 2 of \cite{JungsooKim2012} for the idea of ``\textit{band-sum}''), i.e. $W'$ cuts off a solid torus $\W'$ from $\W$ missing the $\V$-disks of $\Delta_\V$ and $\bar{W}$ is a meridian disk of $\W'$.
Hence, $(V,W)$ belongs to the $3$-simplex $\Sigma_{V'W'}=\{V=V',\bar{V},W=\bar{W},W'\}$ satisfying the condition  ($\ast$).
The symmetric argument also holds for the case (\ref{definition-3-3-3}).

Let us consider the case (\ref{definition-3-3-1}).
Then we can find a band-sum  of two parallel copies of $V$ (resp $W$) in $\V$ (resp $\W$), say $V'$ (resp $W'$), such that  $V'\cap W'=\emptyset$ by using two disjoint simple arcs in $F$ connecting two parallel copies of $V$ and $W$, respectively, because $\partial V\cup\partial W$ is nonseparating in $F$ by Lemma \ref{lemma-five-GHSs} and therefore we can find such two disjoint arcs in the connected surface $F_{VW}$ whose genus is at least one, where the arcs connect two scars of $V$ and those of  $W$, respectively. 
Hence, $\Sigma_{V'W'}=\{V',V=\bar{V},W=\bar{W},W'\}$ forms a $3$-simplex satisfying the condition  ($\ast$).

This completes the proof of Claim \ref{claim-AB}.
\end{proofN}

Hence, considering the $\V$-face $\{V',\bar{V},\bar{W}\}$ and the $\W$-face $\{\bar{V},\bar{W},W'\}$  contained in the $3$-simplex $\Sigma_{V'W'}$, we can see $V$ and $W$ belong to these $\V$- and $\W$-faces, respectively, and each face has one equivalent class  by Lemma \ref{lemma-equivalent}.
Using Lemma \ref{lemma-DorE}, we can guarantee the existence of the $\V$-facial cluster  having one equivalent class, say $\varepsilon_\V$, and the $\W$-facial cluster  having one equivalent class, say $\varepsilon_\W$, such that each of them contains $(\bar{V},\bar{W})$ and both  $V$ and $W$ belong to $\varepsilon_\V\cup\varepsilon_\W$.
Moreover, the GHSs obtained by weak reductions along the weak reducing pairs in $\varepsilon_\V\cup\varepsilon_\W$ are all equivalent because $\varepsilon_\V$ and $\varepsilon_\W$ share $(\bar{V},\bar{W})$.

We can observe the following (see Figure \ref{fig-bb-a} and Definition \ref{definition-2-14}):
\begin{enumerate}[(a)]
\item Only $\bar{V}$ and $\bar{W}$ are nonseparating disks among the vertices of $\varepsilon_\V\cup\varepsilon_\W$.\label{definition-3-3-preclaim-a}
\item Each of $\bar{V}$ and $\bar{W}$ is connected to every other vertex of $\varepsilon_\V\cup\varepsilon_\W$ by a $1$-simplex in $\varepsilon_\V\cup\varepsilon_\W$. \label{definition-3-3-preclaim-b}
\end{enumerate}
\begin{figure}
\includegraphics[width=4cm]{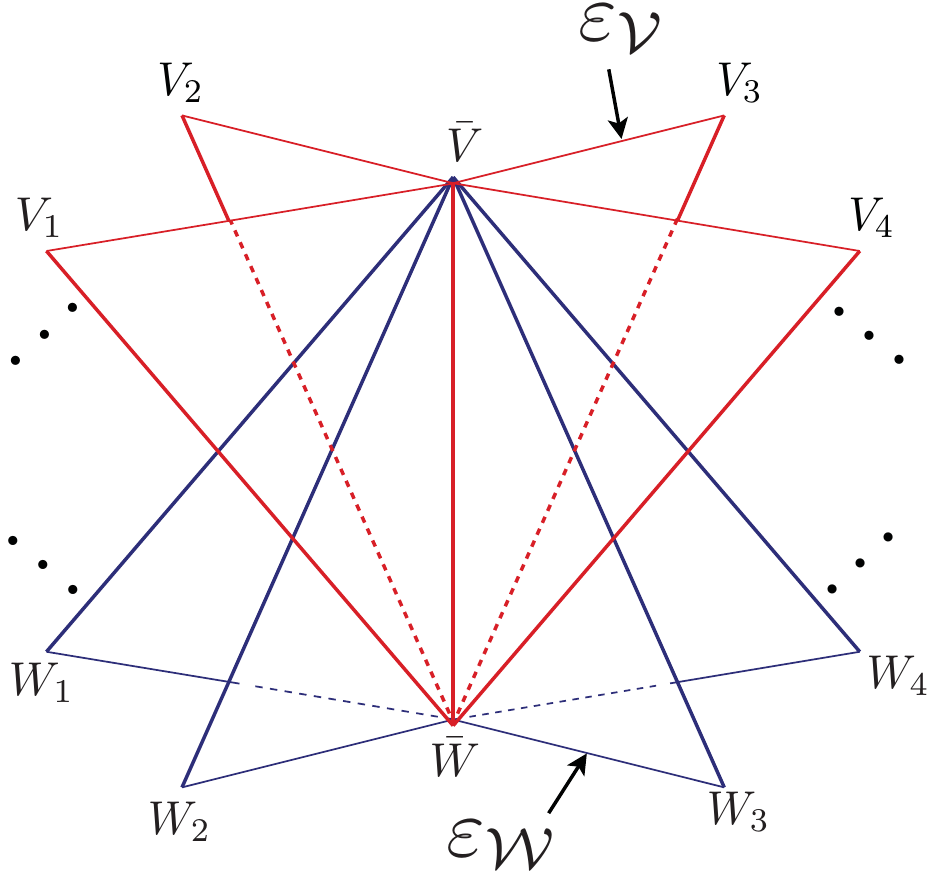}
\caption{$\varepsilon_\V\cup\varepsilon_\W$ \label{fig-bb-a}}
\end{figure}

Let $\Sigma$ be the union of all simplices of $\D(F)$ spanned by the vertices of $\varepsilon_\V\cup\varepsilon_\W$ as in the statement of this lemma.

\begin{claim}\label{claim-4}
There is no $1$-simplex in $\D(F)$ spanned by two separating disks of $\Sigma\cap\DV(F)$ (resp $\Sigma\cap\DW(F)$).
\end{claim}

\begin{proofN}{Claim \ref{claim-4}}
For the sake of contradiction, assume there is a $1$-simplex $\sigma$ in $\D(F)$ spanned by two separating disks of $\Sigma\cap\DV(F)$.
By (\ref{definition-3-3-preclaim-a}), the vertices of $\sigma$ are two disks of $\varepsilon_\V\cap\DV(F)$ other than $\bar{V}$, violating Lemma \ref{lemma-cluster-A}.
This completes the proof of Claim \ref{claim-4}.
\end{proofN}

\begin{claim}\label{claim-5}
$\Sigma$ is equal to the union of all $\Sigma_{V'W'}$ satisfying the condition ($\ast$) and  containing $(\bar{V},\bar{W})$.
\end{claim}

\begin{proofN}{Claim \ref{claim-5}}
Let us consider the union of all $\Sigma_{V'W'}$ satisfying the condition ($\ast$) and containing  $(\bar{V},\bar{W})$.
Considering the condition ($\ast$) for each $\Sigma_{V'W'}$ in the union, $V'$ and $W'$ cut off solid tori from $\V$ and $\W$, respectively, and $\bar{V}$ and $\bar{W}$ are meridian disks of the solid tori, respectively, and therefore each of the $\V$-face $\Delta_\V'=\{V',\bar{V},\bar{W}\}$ and the $\W$-face $\Delta_\W'=\{\bar{V},\bar{W},W'\}$ has one equivalent class by Lemma \ref{lemma-equivalent}.
Considering the $\V$-facial cluster having one equivalent class $\varepsilon_\V'$ which contains $\Delta_\V'$, where the existence and uniqueness are guaranteed by Lemma \ref{lemma-DorE}, the center must be $(\bar{V},\bar{W})$ by Definiton \ref{definition-2-14} because $\bar{V}$ is nonseparating.
But Corollary \ref{corollary-DorE} forces $\varepsilon_\V'$ to be $\varepsilon_\V$ because they share the common center, i.e. $\Delta_\V'\subset \varepsilon_\V$.
Likewise, we get $\Delta_\W'\subset \varepsilon_\W$.
This means $\Sigma_{V'W'}$ itself is a $3$-simplex spanned by four vertices of $\varepsilon_\V\cup\varepsilon_\W$ and therefore the union belongs to $\Sigma$.
Hence, we will prove $\Sigma$ belongs to the union.

It is sufficient to show that each simplex in $\Sigma$ belongs to some $\Sigma_{V'W'}$ in the union.
We will consider $1$- or more simplices in $\Sigma$. 
(Since every vertex of $\Sigma$ is contained in a $\V$-face of $\varepsilon_\V$ or a $\W$-face of $\varepsilon_\W$ which is a $2$-simplex in $\Sigma$, if we can prove it for every $2$-simplex in $\Sigma$, then all vertices in $\Sigma$ would belong to the union.)\\

\Case{a} $\sigma\subset\Sigma$ is a $1$-simplex.

If $\sigma$ is $(\bar{V},\bar{W})$, then there is nothing to prove.
Assume $\sigma\neq(\bar{V},\bar{W})$.
Suppose $\sigma$ contains a nonseparating disk, i.e. it consists of $\bar{V}$ or $\bar{W}$ and a separating disk by (\ref{definition-3-3-preclaim-a}).
Since the separating disk of $\sigma$ is also a vertex of $\varepsilon_\V\cup\varepsilon_\W$, $\sigma$ is the $1$-simplex in $\varepsilon_\V\cup\varepsilon_\W$  described in (\ref{definition-3-3-preclaim-b}), i.e. there is a $\V$- or $\W$-face  having one equivalent class containing $\sigma$, say $\Delta$,  which is contained in $\varepsilon_\V$ or $\varepsilon_\W$, respectively, and therefore it also contains the common center $(\bar{V},\bar{W})$.
If $\Delta$ is a $\V$-face, then one of the $\V$-disks of $\Delta$ cuts off a solid torus from $\V$ and the other is the meridian disk of it by Lemma \ref{lemma-equivalent} and vise versa. 
Hence, if we use the proof of Claim \ref{claim-AB} corresponding to the case (\ref{definition-3-3-2}) or (\ref{definition-3-3-3}), then we can find a $3$-simplex $\Sigma_{V'W'}$ containing $\Delta$ such that it satisfies the condition  ($\ast$) and contains  $(\bar{V},\bar{W})$, leading to the result.

Hence, assume $\sigma$ consists of two separating disks.
If $\sigma$ is not a weak reducing pair, then this violates Claim \ref{claim-4}.
Hence, $\sigma$ is a weak reducing pair, say  $(V',W')$, i.e. we can find a $\V$-face $\Delta_\V^\sigma=\{V',\bar{V},\bar{W}\}$ in  $\varepsilon_\V$ intersecting $\sigma$ and a $\W$-face $\Delta_\W^\sigma=\{\bar{V},\bar{W},W'\}$ in  $\varepsilon_\W$ intersecting $\sigma$.
This means $\Delta_\V^\sigma$, $\Delta_\W^\sigma$ and $\sigma$ form  a $3$-simplex $\Sigma_{V'W'}=\{V',\bar{V},\bar{W},W'\}$.
Since each of $\Delta_\V^\sigma$ and $\Delta_\W^\sigma$ has one equivalent class, considering Lemma \ref{lemma-equivalent}, we conclude $\Sigma_{V'W'}$ satisfies the condition ($\ast$), leading to the result.\\

\Case{b} $\Delta\subset\Sigma$ is a $2$-simplex.

If $\Delta$ is not a $\V$- or $\W$-face, then either $\Delta\subset \DV(F)$ or $\Delta\subset\DW(F)$.
Without loss of generality, assume $\Delta\subset\DV(F)$.
Then there is at least one $1$-simplex $\sigma\subset\Delta$ spanned by two separating disks of $\Sigma\cap\DV(F)$ by (\ref{definition-3-3-preclaim-a}), violating Claim \ref{claim-4}.

Hence, $\Delta$ is a $\V$- or $\W$-face.

If $\Delta$ belongs to $\varepsilon_\V\cup\varepsilon_\W$, then  we can find a $3$-simplex $\Sigma_{V'W'}$ satisfying the condition ($\ast$) and containing $\Delta$ by the same argument for $\Delta$ in Case a.
Since $\Delta$ contains $(\bar{V},\bar{W})$, this leads to the result.

If $\Delta$ does not belong to $\varepsilon_\V\cup\varepsilon_\W$, then exactly one of $\bar{V}$ and $\bar{W}$ does not belong to $\Delta$. ((i) If $\Delta$ contains both $\bar{V}$ and $\bar{W}$, then $\Delta$ is formed by the common center of $\varepsilon_\V$ and $\varepsilon_\W$ and the other disk coming from $\varepsilon_\V$ or $\varepsilon_\W$, i.e. $\Delta$ is a $\V$-face of $\varepsilon_\V$ or a $\W$-face of $\varepsilon_\W$, leading to a contradiction, and (ii) if $\Delta$ misses both $\bar{V}$ and $\bar{W}$, then $\Delta$ consists of two separating $\V$-disks and one separating $\W$-disk by (\ref{definition-3-3-preclaim-a}) or vise versa.
Then there is a $1$-simplex $\sigma\subset\Delta$ spanned by two separating disks of $\Sigma\cap\DV(F)$ or $\Sigma\cap\DW(F)$, respectively, violating Claim \ref{claim-4}.)

Without loss of generality, assume $\bar{V}\notin\Delta$ and $\bar{W}\in\Delta$.
In this case, $\Delta$ must be a $\W$-face otherwise
there is a $1$-simplex $\sigma\subset\Delta$ spanned by two separating disks of $\Sigma\cap\DV(F)$ by (\ref{definition-3-3-preclaim-a}), violating Claim \ref{claim-4}.
Moreover, considering (\ref{definition-3-3-preclaim-b}), there is a $1$-simplex connecting $\bar{V}$ and each vertex of $\Delta$.
This means $\Delta$ forms a $3$-simplex together with $\bar{V}$, where  this $3$-simplex contains $(\bar{V},\bar{W})$ and consists of two disks from $\V$ and two disks from $\W$.
Here, the two disks of this $3$-simplex other than $\bar{V}$ and $\bar{W}$ cut off solid tori from the relevant compression bodies and $\bar{V}$ and $\bar{W}$ are meridian disks of the solid tori, respectively,  by Definition \ref{definition-2-14} (because the vertices come from $\varepsilon_\V\cup\varepsilon_\W$) and therefore this $3$-simplex has the form $\Sigma_{V'W'}=\{V',\bar{V},\bar{W},W'\}$ satisfying the condition  ($\ast$), leading to the result.\\

\Case{c} $\Sigma'\subset\Sigma$ is a $3$-simplex.

If $\Sigma'$ contains a $2$-simplex $\Delta$ which is not a $\V$- or $\W$-face, then we get a contradiction as in the start of Case b.
This means $\Sigma'$ has the form $\{V_1,V_2,W_1,W_2\}$ for $V_i\subset\V$ and $W_i\subset \W$ for $i=1,2$.
Moreover, applying Claim \ref{claim-4} to the $1$-simplices $\{V_1,V_2\}$ and $\{W_1,W_2\}$, we can assume $V_1=\bar{V}$ and $W_1=\bar{W}$ by (\ref{definition-3-3-preclaim-a}).
Hence, considering the last sentence of Case b, we can see $\Sigma'$ itself has the form $\Sigma_{V_2 W_2}=\{V_2,\bar{V}=V_1,\bar{W}=W_1,W_2\}$ satisfying the condition  ($\ast$), leading to the result.\\

If there is a simplex of dimension at least four in $\Sigma$, i.e. it consists of at least five disks, then there must be a $2$-simplex in $\Sigma$  which is neither a $\V$-face nor a $\W$-face, leading to a contradiction as in the start of Case b.
Therefore, we do not need to consider more high dimensional simplices.

This completes the proof of Claim \ref{claim-5}.
\end{proofN}

Considering the proof of Claim \ref{claim-5}, the $1$-simplex $(V,W)$ (which is spanned by two vertices of $\varepsilon_\V\cup\varepsilon_\W$) belongs to $\Sigma$.
Since the common center $(\bar{V},\bar{W})$ of $\varepsilon_\V$ and $\varepsilon_\W$ is uniquely determined by $(V,W)$,  $\varepsilon_\V$ and $\varepsilon_\W$ are also uniquely determined by $(V,W)$ by Corollary \ref{corollary-DorE}.
This completes the proof of all statements of this Lemma other than (\ref{lemma-type-a-i-3}).

\begin{claim}\label{claim-3}
The GHSs obtained by weak reductions along the weak reducing pairs in a $3$-simplex  $\Sigma_{V'W'}$ satisfying the condition ($\ast$) are all equivalent.
\end{claim}

\begin{proofN}{Claim \ref{claim-3}}
Let $\Sigma_{V'W'}=\{V',\tilde{V},\tilde{W},W'\}$ and suppose $(D, E)$ is a weak reducing pair in $\Sigma_{V'W'}$.
We will prove the GHS obtained by weak reduction along $(D,E)$ is equivalent to that obtained by weak reduction along $(\tilde{V},\tilde{W})$.

First, considering $\Sigma_{V'W'}$ satisfies the condition ($\ast$), we can see each of the $\V$-face $\Delta_{\V}=\{V',\tilde{V},\tilde{W}\}$ and the $\W$-face $\Delta_{\W}=\{\tilde{V},\tilde{W},W'\}$ has one equivalent class by Lemma \ref{lemma-equivalent}.

If one of $D$ and $E$ is $\tilde{V}$ or $\tilde{W}$, then $(D,E)$ belongs to $\Delta_{\W}$ or $\Delta_{\V}$, respectively, leading to the result.

Suppose $D=V'$ and $E=W'$.
In this case, we get the $\V$-face $\Delta_{\V}'=\{V',\tilde{V},W'\}$ containing $(V',W')$ and  $\Delta_{\V}'$ has one equivalent class by Lemma \ref{lemma-equivalent}.
Moreover, $\Delta_{\V}'$ shares the weak reducing pair $(\tilde{V},W')$ with $\Delta_{\W}$, i.e. $(D,E)$ gives a GHS equivalent to that obtained by using $(\tilde{V},\tilde{W})$ after weak reduction.

This completes the proof of Claim \ref{claim-3}.
\end{proofN}

Considering the proof of Claim \ref{claim-5}, any weak reducing pair $(D,E)$ in $\Sigma$ belongs to a $3$-simplex $\Sigma_{V'W'}$ satisfying the condition  ($\ast$) and containing $(\bar{V},\bar{W})$.
Moreover, considering  Claim \ref{claim-3}, we conclude the GHS obtained by weak reduction along $(D,E)$ is equivalent to the GHS obtained by weak reduction along $(\bar{V},\bar{W})$, i.e. the GHSs obtained by weak reductions along the weak reducing pairs in $\Sigma$ are all equivalent to that corresponding to $(\bar{V},\bar{W})$.
This completes the proof of (\ref{lemma-type-a-i-3}).

This completes the proof.
\end{proof}

\begin{definition}
\label{definition-3-3-detail}
Let $M$ and $F$ be as in Lemma \ref{lemma-five-GHSs}. 
An \textit{equivalent cluster of type (a)-(i)} is the union of all simplices of $\DVW(F)$ spanned by the vertices of $\varepsilon_\V\cup\varepsilon_\W$, where $\varepsilon_\V$ is a $\V$-facial cluster having one equivalent class, $\varepsilon_\W$ is a $\W$-facial cluster having one equivalent class, and they share the common center.
We call the common center of $\varepsilon_\V$ and $\varepsilon_\W$  the \textit{center} of the equivalent cluster.
Lemma \ref{lemma-type-a-i} means every weak reducing pair giving a type (a)-(i) GHS after weak reduction belongs to a naturally determined equivalent cluster of type (a)-(i).
\end{definition}

\begin{lemma}
An equivalent cluster of type (a)-(i) is contractible.
\end{lemma}

\begin{proof}
Considering Claim \ref{claim-5}, the shape of an equivalent cluster of type (a)-(i) is essentially the same as that of $\DVW(F)$ of Case (a) in the proof of Lemma 3.5 of \cite{JungsooKim2014} and it is contractible.
Therefore, we omit the proof.
\end{proof}

\begin{lemma}
\label{lemma-type-a-ii}
Let $M$ and $F$ be as in Lemma \ref{lemma-five-GHSs} and let $(V,W)$ be a weak reducing pair such that  the GHS obtained by weak reduction along $(V,W)$ is of type (a)-(ii)-$\W$ or type (b)-$\W$-(i).
Then there is a naturally determined $\W$-facial cluster having one equivalent class containing $(V,W)$.
Likewise,  if the GHS obtained by weak reduction along $(V,W)$ is of type (a)-(ii)-$\V$ or type (b)-$\V$-(i), then there is a naturally determined $\V$-facial cluster having one equivalent class containing $(V,W)$.
\end{lemma}

\begin{proof}
Suppose the GHS obtained by weak reduction along $(V,W)$  is of type (a)-(ii)-$\W$ or type (b)-$\W$-(i).
We can see the following from Definition \ref{definition-GHSs}:
\begin{enumerate}
\item $W$ is nonseparating or $W$ cuts off a solid torus $\W'$ missing $\partial V$ from $\W$.
\item Either $V$ cuts off a handlebody of genus at least two from $\V$ missing $\partial W$ (the type (a)-(ii)-$\W$ case) or a compression body with nonempty negative boundary missing $\partial W$ from $\V$ (the type (b)-$\W$-(i) case).
\end{enumerate}
Let $\bar{V}$ be $V$ itself.
Take $\bar{W}$ as $W$ itself if $W$ is nonseparating in $\W$.
If $W$ is separating, then we take $\bar{W}$ as a meridian disk of $\W'$.
Hence, we can find a weak reducing pair $(\bar{V},\bar{W})$ such that either
\begin{enumerate}
\item $(V,W)$ is $(\bar{V},\bar{W})$ itself, or\label{definition-3-5-1}
\item $(V,W)$ is contained in the $\W$-face $\Delta_\W=\{\bar{V},\bar{W},W\}$ and $\Delta_\W$ has one equivalent class by Lemma \ref{lemma-equivalent}.\label{definition-3-5-2}
\end{enumerate}
Here, $(\bar{V},\bar{W})$ is uniquely determined in $\D(F)$ by $(V,W)$ because the meridian disk of a solid torus is unique up to isotopy.

Consider the case (\ref{definition-3-5-1}).
Then $\partial \bar{W}$ is nonseparating in the relevant component of $F_{\bar{V}}$ by Lemma \ref{lemma-pre}, where this component is of genus at least two by Lemma \ref{lemma-2-8}, and therefore we can take a band-sum of two parallel copies of $W=\bar{W}$ in $\W$, say $W'$, such that $\Delta=\{V=\bar{V},W=\bar{W},W'\}$ is a $\W$-face.
Moreover, $\Delta$  has one equivalent class by Lemma \ref{lemma-equivalent}.
Hence, we can guarantee the existence of the $\W$-facial cluster having one equivalent class, say $\varepsilon_\W$, such that $\varepsilon_\W$ contains $(V,W)$ in the case (\ref{definition-3-5-1})  by Lemma \ref{lemma-DorE} as well as the case (\ref{definition-3-5-2}).
Since the center $(\bar{V},\bar{W})$ of $\varepsilon_\W$ is uniquely determined by $(V,W)$,  $\varepsilon_\W$ is uniquely determined by $(V,W)$ by Corollary \ref{corollary-DorE}.

This completes the proof.
\end{proof}

\begin{definition}
\label{definition-3-5-detail}
Let $M$ and $F$ be as in Lemma \ref{lemma-five-GHSs}.
Let $\varepsilon_\W$ be a $\W$-facial cluster having one equivalent class such that the GHSs obtained by weak reductions along the weak reducing pairs in $\varepsilon_\W$ are either of type (a)-(ii)-$\W$ or type (b)-$\W$-(i).
We call $\varepsilon_\W$ an \textit{equivalent cluster of type (a)-(ii)-$\W$} or \textit{(b)-$\W$-(i)}, respectively.
We call the center of $\varepsilon_\W$ the \textit{center} of the equivalent cluster.

Lemma \ref{lemma-type-a-ii} means every weak reducing pair giving a type (a)-(ii)-$\W$ or (b)-$\W$-(i) GHS after weak reduction belongs to a naturally determined equivalent cluster of type (a)-(ii)-$\W$ or type (b)-$\W$-(i), respectively.
\end{definition}

Likewise,  we get the following definition.

\begin{definition}
\label{definition-3-5-ii-detail}
Let $M$ and $F$ be as in Lemma \ref{lemma-five-GHSs}.
Let $\varepsilon_\V$ be a $\V$-facial cluster having one equivalent class such that the GHSs obtained by weak reductions along the weak reducing pairs in $\varepsilon_\V$ are either of type (a)-(ii)-$\V$ or type (b)-$\V$-(i).
We call $\varepsilon_\V$ an \textit{equivalent cluster of type (a)-(ii)-$\V$} or \textit{type (b)-$\V$-(i)}, respectively.
We call the center of $\varepsilon_\V$ the \textit{center} of the equivalent cluster.

Lemma \ref{lemma-type-a-ii} means every weak reducing pair giving a type (a)-(ii)-$\V$ or (b)-$\V$-(i) GHS after weak reduction belongs to a naturally determined equivalent cluster of type (a)-(ii)-$\V$ or type (b)-$\V$-(i), respectively.
\end{definition}

Since a $\W$- or $\V$-facial cluster having one equivalent class is contractible, an equivalent cluster of type (a)-(ii)-$\W$,  type (b)-$\W$-(i), type (a)-(ii)-$\V$ or type (b)-$\V$-(i) is contractible.
Moreover, Lemma \ref{lemma-cluster-all-simplices} means an equivalent cluster of type (a)-(ii)-$\W$, type (b)-$\W$-(i), type (a)-(ii)-$\V$ or type (b)-$\V$-(i) is the union of all simplices of $\D(F)$ spanned by the vertices of the cluster.

\begin{definition}\label{definition-3-6}
Let $M$ and $F$ be as in Lemma \ref{lemma-five-GHSs}.
Suppose the GHS obtained by weak reduction along $(V,W)$ is of type (a)-(iii), type (b)-$\W$-(ii), type (b)-$\V$-(ii), type (c), or type (d).
In this case, we call the weak reducing pair $(V,W)$ itself \textit{an equivalent cluster of type (a)-(iii)}, \textit{type (b)-$\W$-(ii)}, \textit{type (b)-$\V$-(ii)}, \textit{type (c)}, or \textit{type (d)}, respectively.
We also define the \textit{center} of the equivalent cluster $(\bar{V},\bar{W})$ as $(V,W)$ itself.
\end{definition}

See Figure \ref{fig-bbs} for the shapes of equivalent clusters.
\begin{figure}
\includegraphics[width=10cm]{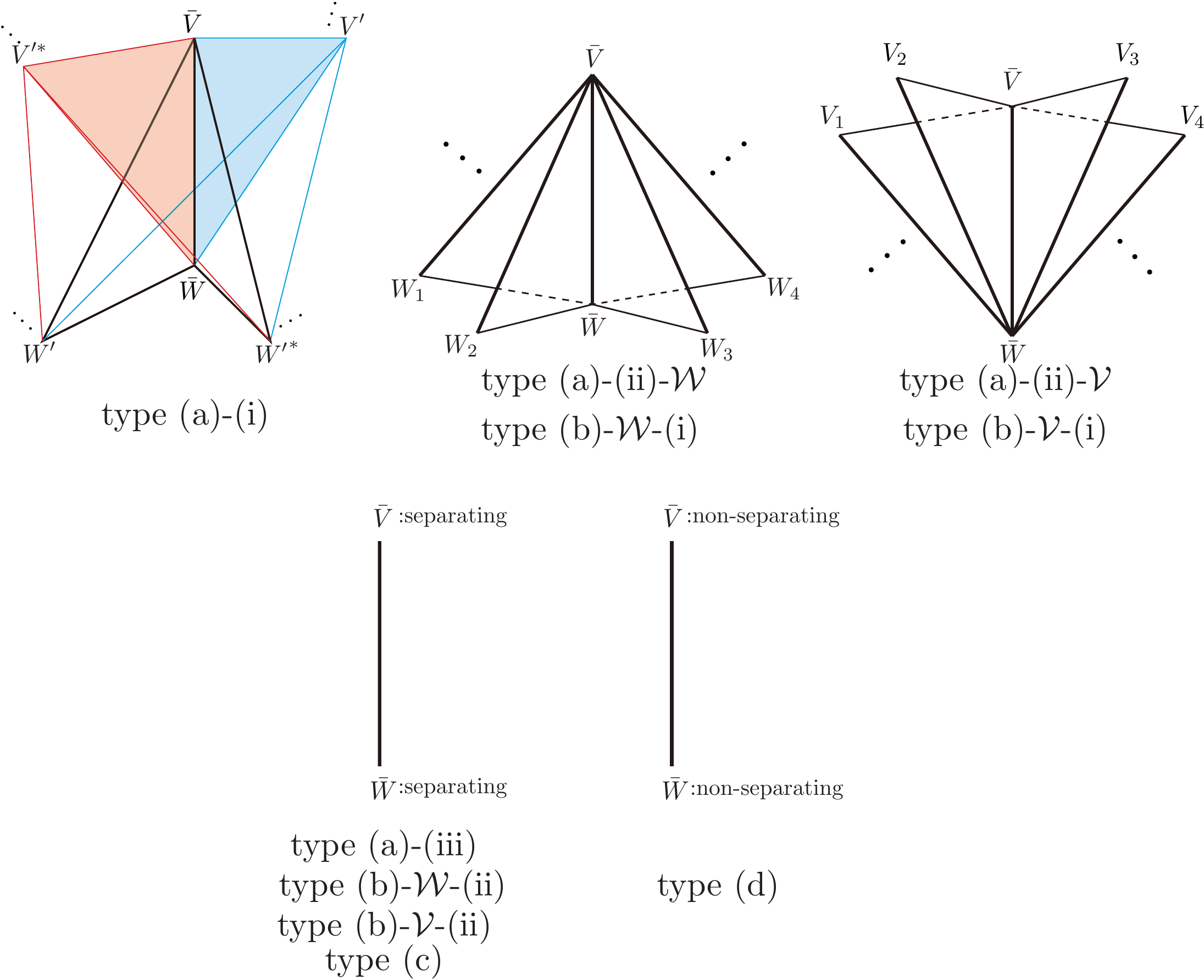}
\caption{the shapes of equivalent clusters\label{fig-bbs}}
\end{figure}

In summary, we get the following two lemmas.

First, we can understand the \textit{type} of an equivalent cluster as the type of GHSs obtained by weak reductions in the cluster and therefore we summarize all cases in Lemma  \ref{lemma-3-9-ii}.
Since any weak reducing pair gives exactly one type of GHS after weak reduction in the sense of Definition \ref{definition-GHSs}, the second statement of Lemma \ref{lemma-3-9-ii} comes from  Lemma \ref{lemma-types} directly.
\begin{lemma}\label{lemma-3-9-ii}
Let $M$ and $F$ be as in Lemma \ref{lemma-five-GHSs} and let $\mathcal{B}$ be an equivalent cluster.
Then the GHSs obtained by weak reductions along the weak reducing pairs in $\mathcal{B}$ are all equivalent.
Moreover, the GHSs are of the same type in the sense of Definition \ref{definition-GHSs}, where the type is uniquely determined.
\end{lemma}

Next, Lemma \ref{lemma-3-9} means every equivalent cluster is contractible and the intersection of an equivalent cluster and $\DV(F)$ or $\DW(F)$ is contractible.

\begin{lemma}\label{lemma-3-9}
Let $M$ and $F$ be as in Lemma \ref{lemma-five-GHSs}.
Then every equivalent cluster is contractible.
Moreover, every equivalent cluster is the union of all simplices of $\D(F)$ spanned by the vertices of the cluster.
When an equivalent cluster intersects $\DV(F)$ or $\DW(F)$, the dimension of the intersection is at most one.
Indeed, the intersection is either a vertex coming from the center if the dimension is zero or a star-shaped graph with infinitely many edges if the dimension is one.
Moreover, if the dimension of the intersection is one, then the star-shaped graph is equal to the intersection of the $\V$- or $\W$-facial cluster having one equivalent class in the equivalent cluster whose center is the center of the equivalent cluster and $\DV(F)$ or $\DW(F)$, respectively, i.e. the vertex positioned at the center of the star-shaped graph comes from the center of the equivalent cluster.
\end{lemma}

Note that Lemma \ref{lemma-cluster-A} induces that (i) the dimension of the intersection of an equivalent cluster and $\DV(F)$ or $\DW(F)$ is at most one and (ii) the shape of the intersection must be a star-shaped graph for the dimension one case in the statement  of Lemma \ref{lemma-3-9}.

Considering Lemma \ref{lemma-type-a-i}, Lemma \ref{lemma-type-a-ii} and Definition \ref{definition-3-6}, any weak reducing pair $(V,W)$ belongs to a naturally determined equivalent cluster, say the \textit{canonical equivalent cluster} for $(V,W)$.

The next lemma gives a criterion to determine whether a disk in an equivalent cluster belongs to the center of the cluster or not.

\begin{lemma}\label{lemma-center}
Let $M$ and $F$ be as in Lemma \ref{lemma-five-GHSs} and $D$ is a disk in an equivalent cluster $\mathcal{B}$.
Then $D$ belongs to the center of $\mathcal{B}$ if and only if one of the following holds.
\begin{enumerate}
\item it is nonseparating or 
\item if it is separating, then it does not cut off a solid torus from the relevant compression body.
\end{enumerate}
\end{lemma}

\begin{proof}
Let $(\bar{V},\bar{W})$ be the center of $\mathcal{B}$.
Then the following hold:
\begin{enumerate}
\item \textbf{Type (a)-(i)}: $\bar{V}$ and $\bar{W}$ are nonseparating, and another disk cuts off a solid torus from the relevant compression body (see Definition \ref{definition-3-3-detail} and Definition  \ref{definition-2-14}).
\item \textbf{Type (a)-(ii)-$\W$ and Type (b)-$\W$-(i)}: $\bar{W}$ is nonseparating and another $\W$-disk cuts off a solid torus from $\W$ (see Definition \ref{definition-3-5-detail} and Definition  \ref{definition-2-14}).
$\bar{V}$ is separating in $\V$ and it does not cut off a solid torus from $\V$  (see Definition \ref{definition-GHSs}).
\item \textbf{Type (a)-(ii)-$\V$ and Type (b)-$\W$-(i)}: Use the symmetric argument of the previous case (see Definition \ref{definition-3-5-ii-detail}, Definition  \ref{definition-2-14} and Definition \ref{definition-GHSs}).
\item \textbf{Type (a)-(iii), Type (b)-$\W$-(ii), Type (b)-$\V$-(ii) and Type (c)}:
Each of $\bar{V}$ and $\bar{W}$ is separating, but it does not cut off a solid torus from the relevant compression body (see Lemma \ref{lemma-five-GHSs},  Definition \ref{definition-GHSs} and Definition \ref{definition-3-6}).
\item \textbf{Type (d)}: $\bar{V}$ and $\bar{W}$ are nonseparating (see Lemma \ref{lemma-five-GHSs} and Definition \ref{definition-3-6}).
\end{enumerate}
This completes the proof.
\end{proof}

\begin{lemma}\label{lemma-6x}
Let $M$ and $F$ be as in Lemma \ref{lemma-five-GHSs}.
Let $\mathcal{B}_i$ be an equivalent cluster and $(\bar{V}_i,\bar{W}_i)$  the center of $\mathcal{B}_i$ for $i=1,2$.
If $(\mathcal{B}_1\cap\mathcal{B}_2)\cap \DV(F)\neq\emptyset$, then $\bar{V}_1=\bar{V}_2$. 
Likewise, if $(\mathcal{B}_1\cap\mathcal{B}_2)\cap \DW(F)\neq\emptyset$, then $\bar{W}_1=\bar{W}_2$.
\end{lemma}

\begin{proof}
Let $V$ be a vertex of $(\mathcal{B}_1\cap\mathcal{B}_2)\cap \DV(F)$.

Assume $V=\bar{V}_1$.
Then (1) $V$ is nonseparating or (2) it does not cut off a solid torus from $\V$ if  it is separating  as a disk in $\mathcal{B}_1$ by Lemma \ref{lemma-center}.
Since $V$ also belongs to $\mathcal{B}_2$, $V=\bar{V}_2$  by Lemma \ref{lemma-center}, i.e. $\bar{V}_1=\bar{V}_2$.

Assume $V\neq\bar{V}_1$.
Then $V$ cuts off a solid torus $\V'$ from $\V$ as a disk in $\mathcal{B}_1$ by Lemma \ref{lemma-center} and therefore $V\neq\bar{V}_2$  as a disk in $\mathcal{B}_2$ by Lemma \ref{lemma-center}.
Here, each $\mathcal{B}_i$ contains a $\V$-facial cluster having one equivalent class whose center is $(\bar{V}_i,\bar{W}_i)$, say $\varepsilon_\V^i$,  and $V$ is a vertex of the star-shaped graph $\varepsilon_\V^i\cap \DV(F)$ which is not the center of the graph  by Lemma \ref{lemma-3-9} for $i=1,2$.
This means $\bar{V}_1$ and $\bar{V}_2$ are meridian disks of the solid torus $\V'$ by Definition \ref{definition-2-14}, i.e. they are isotopic in $\V$. 
Therefore, we get $\bar{V}_1=\bar{V}_2$.

This completes the proof.
\end{proof}

The next lemma means the canonical equivalent cluster for $(V,W)$ is the only equivalent cluster containing $(V,W)$.

\begin{lemma}\label{lemma-BB-char}
Let $M$ and $F$ be as in Lemma \ref{lemma-five-GHSs}.
Then every weak reducing pair belongs to a uniquely determined equivalent cluster.
This means if two different equivalent clusters intersect each other, then the intersection cannot contain a weak reducing pair.
\end{lemma}

\begin{proof}
By definitions of equivalent clusters, every weak reducing pair belongs to some equivalent cluster.
Let us consider a weak reducing pair $(V,W)$.
Suppose there are two equivalent clusters containing $(V,W)$, say $\mathcal{B}_1$ and $\mathcal{B}_2$.
Let $(\bar{V}_i,\bar{W}_i)$ be the center of $\mathcal{B}_i$ for $i=1,2$.
We will prove $\mathcal{B}_1=\mathcal{B}_2$.

\begin{claim}
$\mathcal{B}_1$ and $\mathcal{B}_2$ are of the same type and they share the common center.\label{claim-6}
\end{claim}

\begin{proofN}{Claim \ref{claim-6}}
Since $\mathcal{B}_1$ and $\mathcal{B}_2$ share $(V,W)$, the GHSs obtained by weak reductions along the weak reducing pairs in $\mathcal{B}_1\cup\mathcal{B}_2$ are all equivalent to that corresponding to $(V,W)$ and they are of the same type in the sense of Definition \ref{definition-GHSs}  by Lemma \ref{lemma-3-9-ii}.  
This means $\mathcal{B}_1$ and $\mathcal{B}_2$ are of the same type.

Since $\mathcal{B}_1$ and $\mathcal{B}_2$ share $(V,W)$, $(\mathcal{B}_1\cap\mathcal{B}_2)\cap \DV(F)\neq\emptyset$ and $(\mathcal{B}_1\cap\mathcal{B}_2)\cap \DW(F)\neq\emptyset$.
Therefore, Lemma \ref{lemma-6x} means $\bar{V}_1=\bar{V}_2$ and $\bar{W}_1=\bar{W}_2$, i.e. $(\bar{V}_1,\bar{W}_1)=(\bar{V}_2,\bar{W}_2)$.

This completes the proof of Claim \ref{claim-6}.
\end{proofN}

By Claim \ref{claim-6}, $\mathcal{B}_1$ and $\mathcal{B}_2$ are of the same type (so the shapes of them are  the same in the sense of Figure \ref{fig-bbs}) and they share the common center.
Hence, if each of $\mathcal{B}_1$ and $\mathcal{B}_2$ consists of a weak reducing pair, then we conclude $\mathcal{B}_1=\mathcal{B}_2$, leading to the result.
Hence, assume each of $\mathcal{B}_1$ and $\mathcal{B}_2$ does not consist of a weak reducing pair.
Considering $\mathcal{B}_i\cap\DV(F)$ and $\mathcal{B}_i\cap\DW(F)$, if both are of dimension zero, then each of them consists of a vertex by Lemma \ref{lemma-3-9}, i.e. $\mathcal{B}_i$ is weak reducing pair, a contradiction.
Hence, at least one of $\mathcal{B}_i\cap\DV(F)$ and $\mathcal{B}_i\cap\DW(F)$ is of dimension one by Lemma \ref{lemma-3-9} for $i=1,2$.

If $\dim(\mathcal{B}_i\cap\DV(F))=1$ for $i=1,2$, then  each of $\mathcal{B}_1$ and $\mathcal{B}_2$ contains a $\V$-facial cluster having one equivalent class whose center is $(\bar{V}_i,\bar{W}_i)$, say $\varepsilon_\V^i$,  by Lemma \ref{lemma-3-9} for $i=1,2$.
Here, Claim \ref{claim-6} means the center of $\varepsilon_\V^1$ is equal to that of $\varepsilon_\V^2$ and therefore we get $\varepsilon_\V^1=\varepsilon_\V^2$ by Corollary \ref{corollary-DorE}.
Hence, say $\varepsilon_\V:=\varepsilon_\V^1=\varepsilon_\V^2$ in this case.

Likewise, if $\dim(\mathcal{B}_i\cap\DW(F))=1$ for $i=1,2$, then  each of $\mathcal{B}_1$ and $\mathcal{B}_2$ contains a $\W$-facial cluster having one equivalent class whose center is $(\bar{V}_i,\bar{W}_i)$, say $\varepsilon_\W^i$, and therefore we get $\varepsilon_\W^1=\varepsilon_\W^2$ likewise. 
Hence, say $\varepsilon_\W:=\varepsilon_\W^1=\varepsilon_\W^2$ in this case.

Considering the uniqueness of the $\V$- or $\W$-facial cluster having one equivalent class whose center is $(\bar{V}_i, \bar{W}_i)$ by Corollary \ref{corollary-DorE}, we can see either
\begin{enumerate}
\item $\varepsilon_\V\cup \varepsilon_\W$ determines $\mathcal{B}_i$ as in Definition \ref{definition-3-3-detail} ($\mathcal{B}_i$ is of type (a)-(i)) for i=1,2,
\item $\varepsilon_\W$ is itself $\mathcal{B}_i$ as in Definition \ref{definition-3-5-detail} ($\mathcal{B}_i$ is of type (a)-(ii)-$\W$ or type (b)-$\W$-(i)) for i=1,2, or
\item $\varepsilon_\V$ is itself $\mathcal{B}_i$ as in Definition \ref{definition-3-5-ii-detail} ($\mathcal{B}_i$ is of type (a)-(ii)-$\V$ or type (b)-$\V$-(i)) for i=1,2.
\end{enumerate}
Hence, we conclude $\mathcal{B}_1=\mathcal{B}_2$.

This completes the proof.
\end{proof}

\begin{definition}\label{definition-Phi}
Considering Lemma \ref{lemma-3-9-ii} and Lemma \ref{lemma-BB-char}, there is the \textit{canonical function $\Phi$} from the set of equivalent clusters to the set of equivalence classes $\mathcal{GHS}_F^\ast/\sim$, where $\Phi(\mathcal{B})$ is the equivalent class of the GHSs obtained by weak reductions along the weak reducing pairs in $\mathcal{B}$ for a given equivalent cluster $\mathcal{B}$.
For a representative $\mathbf{H}\in\mathcal{GHS}_F^\ast$ of an equivalent class $[\mathbf{H}]^\ast$, there exists an weak reducing pair giving $\mathbf{H}$  after weak reduction.
Moreover, this weak reducing pair belongs to an equivalent cluster $\mathcal{B}'$ by Lemma \ref{lemma-BB-char}, i.e. $\Phi(\mathcal{B}')=[\mathbf{H}]^\ast$, and therefore $\Phi$ is surjective.
\end{definition}

As well as a weak reducing pair, a $\V$- or $\W$-facial cluster having one equivalent class also belongs to a uniquely determined equivalent cluster by Lemma \ref{lemma-BB-some}.

\begin{lemma}\label{lemma-BB-some}
Let $M$ and $F$ be as in Lemma \ref{lemma-five-GHSs}.
If there is a $\V$- or $\W$-facial cluster having one equivalent class, then it belongs to a uniquely determined equivalent cluster.
Therefore, every $\V$- or $\W$-face having one equivalent class belongs to a uniquely determined equivalent cluster.
\end{lemma}

\begin{proof}
Without loss of generality, assume $\varepsilon$ is a $\V$-facial cluster having one equivalent class.
Let $(V,W)$ be the center of $\varepsilon$.
Then $V$ is nonseparating in $\V$ by Definition \ref{definition-2-14}.

\begin{claim}\label{claim-7}
$(V,W)$ itself is not an equivalent cluster.
\end{claim}

\begin{proofN}{Claim \ref{claim-7}}
For the sake of contradiction, suppose $(V,W)$ is an equivalent cluster.
Let us consider the cases where a weak reducing pair itself is an equivalent cluster.
In the cases of type (a)-(iii), type (b)-$\W$-(ii), type (b)-$\V$-(ii) and type (c), both disks are separating by Definition \ref{definition-GHSs}.
Therefore, considering that $V$ is nonseparating in $\V$,  $(V,W)$ is of type (d), i.e. both $V$ and $W$ are nonseparating in $\V$ and $\W$, respectively, and $\partial V\cup\partial W$ is separating in $F$.

Let $D$ be another disk of $\varepsilon$ other than $V$ and $W$.
Then $D$ cuts off solid torus $\V'$ from $\V$ and $V$ is a meridian disk of $\V'$  by Definition \ref{definition-2-14}, i.e. $\partial V$ is nonseparating in the relevant component of $F_D$ by Lemma \ref{lemma-pre}.
Moreover, $\partial W$ misses $\V'$ by Lemma \ref{lemma-2-8} and $\partial W$ is nonseparating in the relevant component of $F_D$ by Lemma \ref{lemma-pre}.
Let us compress $F_D$ along $V$ and $W$.
Then the resulting one $(F_D)_{VW}$ consists of two components, i.e. $F_{VW}$ consists of only one component by recovering $F_{VW}$ from the two components of $(F_D)_{VW}$.
This means $\partial V\cup\partial W$ is nonseparating in $F$, violating the assumption.

This completes the proof of Claim \ref{claim-7}.
\end{proofN}

By Claim \ref{claim-7}, $(V,W)$  itself cannot be an equivalent cluster, i.e. the GHS obtained by weak reduction along $(V,W)$ is either of type (a)-(i), type (a)-(ii)-$\W$, type (a)-(ii)-$\V$, type (b)-$\W$-(i) or type (b)-$\V$-(i).
Moreover, we can exclude the cases of type (a)-(ii)-$\W$ and type (b)-$\W$-(i) because $V$ is nonseparating in $\V$.
Hence, considering the proofs of Lemma  \ref{lemma-type-a-i} or Lemma  \ref{lemma-type-a-ii}, there is a weak reducing pair $(\bar{V},\bar{W})$ uniquely determined by $(V,W)$ such that it becomes the center of  an equivalent cluster $\mathcal{B}$ of type (a)-(i) (if $W$ is nonseparating or cuts off a solid torus from $\W$), type (a)-(ii)-$\V$ (if $W$ cuts off a handlebody of genus at least two missing $\partial V$ from $\W$) or type (b)-$\V$-(i) (if $W$ cuts off a compression body with nonempty negative boundary missing $\partial V$ from $\W$) and $\mathcal{B}$ contains $(V,W)$.

Since there cannot be two different equivalent clusters sharing $(V,W)$ by Lemma \ref{lemma-BB-char}, if there is an equivalent cluster containing $\varepsilon$, then it must be the only one.
Therefore, we only need to prove $\varepsilon$ belongs to $\mathcal{B}$.

Since $V$ is nonseparating in $\V$ and $V\in\mathcal{B}$, $V$ belongs the center of $\mathcal{B}$ by Lemma \ref{lemma-center}.

Suppose $W$ does not cut off a solid torus from $\W$, i.e. $(V,W)$ is the center of $\mathcal{B}$ as well as that of $\varepsilon$ by Lemma \ref{lemma-center}.
In this case, the $\V$-facial cluster $\varepsilon_\V$ forming $\mathcal{B}$ in Definition \ref{definition-3-3-detail} or Definition \ref{definition-3-5-ii-detail} must be equal to $\varepsilon$ by Corollary \ref{corollary-DorE}, i.e. $\varepsilon\subset \mathcal{B}$, leading to the result.

Suppose $W$ cuts off a solid torus $\W'$ from $\W$, i.e. the case where $\mathcal{B}$ is of a type (a)-(i).
Let $\tilde{W}$ be a meridian disk of $\W'$ missing $W$. 
Then $V=\bar{V}$ and $\tilde{W}=\bar{W}$ by the proof of  Lemma  \ref{lemma-type-a-i} and $\tilde{W}$ misses all vertices of $\varepsilon$ because every $\V$-disks of $\varepsilon$ cannot intersect the once-punctured torus component of $F-\partial W$  by Lemma \ref{lemma-2-8}.
This means each $\V$-face $\{V', V, W\}$ of $\varepsilon$  belongs to the $3$-simplex of the form $\Sigma_{V'W'}=\{V',\bar{V}=V,\bar{W}=\tilde{W},W'=W\}$, where $V'$ cuts off a solid torus from $\V$ and $V$ is a meridian disk of the solid torus by Definition \ref{definition-2-14}.
Since $\mathcal{B}$ contains all such $3$-simplices by Claim \ref{claim-5}, we conclude $\varepsilon$ belongs to $\mathcal{B}$, leading to the result.

Since  a $\V$-face having one equivalent class belongs to a uniquely determined  $\V$-facial cluster having one equivalent class by Lemma \ref{lemma-DorE}, the second statement comes from the first statement  directly.

This completes the proof.
\end{proof}

\begin{lemma}\label{lemma-equiv-BB}
Let $M$ and $F$ be as in Lemma \ref{lemma-five-GHSs}.
For a $3$-simplex $\Sigma$ of $\D(F)$ containing at least one weak reducing pair, if the GHSs obtained by weak reductions along the weak reducing pairs in $\Sigma$ are all equivalent, then $\Sigma$ belongs to a uniquely determined equivalent cluster of type (a)-(i).
\end{lemma}

\begin{proof}
Considering Lemma \ref{lemma-restrict}, $\Sigma$ has the form $\{V',V,W,W'\}$, where  $V'$ and $W'$ cut off solid tori $\V'$ and $\W'$ from $\V$ and $\W$, respectively, and  $V$ and $W$ are meridian disks of $\V'$ and $\W'$, respectively, and therefore each of the $\V$-face $\Delta_\V=\{V',V,W\}$ and the $\W$-face $\Delta_\W=\{V,W,W'\}$ has one equivalent class by Lemma \ref{lemma-equivalent}.
This gives the $\V$- and $\W$-facial clusters having one equivalent class containing $\Delta_\V$ and $\Delta_\W$, say $\varepsilon_\V$ and $\varepsilon_\W$, respectively, by Lemma \ref{lemma-DorE}, and they share the common center $(V,W)$ by Definition \ref{definition-2-14}.
Using $\varepsilon_\V\cup\varepsilon_\W$, we can find an equivalent cluster $\mathcal{B}$ of type (a)-(i) by Definition \ref{definition-3-3-detail} and $\mathcal{B}$ contains $\Sigma$ by Claim \ref{claim-5}.

If there is another equivalent cluster containing $\Sigma$, say $\mathcal{B}'$, then $\mathcal{B}$ and $\mathcal{B}'$ share the weak reducing pairs in $\Sigma$, violating Lemma \ref{lemma-BB-char}.

This completes the proof.
\end{proof}

\begin{definition}
Let $M$ and $F$ be as in Lemma \ref{lemma-five-GHSs}.
Let $\Sigma$ be a simplex of $\D(F)$ containing at least one weak reducing pair.
If there is only one equivalent class of the GHSs obtained by weak reductions along the weak reducing pairs in $\Sigma$, then we call $\Sigma$ an \textit{equivalent simplex}.
If there are two or more equivalent classes of the GHSs obtained by weak reductions along the weak reducing pairs in $\Sigma$, then we call $\Sigma$ a \textit{nonequivalent simplex}.
\end{definition}

In summary, we get the following theorem.

\begin{theorem}
\label{theorem-DVWF}
Let $M$ and $F$ be as in Lemma \ref{lemma-five-GHSs}.
Then $\DVW(F)$ consists of equivalent clusters and nonequivalent simplices.
Moreover, each equivalent simplex belongs to a uniquely determined equivalent cluster.
\end{theorem}

\begin{proof}
Recall that $\DVW(F)$ consists of all simplices of $\D(F)$ intersecting both $\DV(F)$ and $\DW(F)$, i.e. it is the union of all equivalent simplices and all nonequivalent simplices.

Considering Lemma \ref{lemma-BB-char}, Lemma \ref{lemma-BB-some} and Lemma \ref{lemma-equiv-BB}, every equivalent simplex of dimension at most three belongs to a uniquely determined equivalent cluster.
Moreover, Lemma \ref{lemma-restrict} means there is no equivalent simplex of dimension at least four.
Hence, all equivalent simplices are contained  in the union of all equivalent clusters.
Therefore, the union of all equivalent clusters and  all nonequivalent simplices covers $\DVW(F)$.
Since every equivalent cluster is a subset of $\DVW(F)$, we conclude the previous union  is exactly the same as $\DVW(F)$.
 
This completes the proof.
\end{proof}

\section{the proof of Theorem \ref{theorem-main-a}}
In this section, we prove Theorem \ref{theorem-main-a}.

The next lemma describes the way how two different equivalent clusters intersect each other.

\begin{lemma}\label{lemma-BB-intersect}
Suppose $M$ is an irreducible $3$-manifold and $(\V,\W;F)$ is a weakly reducible, unstabilized Heegaard splitting of $M$ of genus $n\geq 3$.
If two different equivalent clusters $\mathcal{B}_1$ and $\mathcal{B}_2$ intersect each other, then  the center of $\mathcal{B}_1$ intersects that of $\mathcal{B}_2$ in a vertex.
\end{lemma}

\begin{proof}
Let $(\bar{V}_i,\bar{W}_i)$ be the center of $\mathcal{B}_i$ for $i=1,2$.
If $(\mathcal{B}_1\cap\mathcal{B}_2)\cap \DV(F)\neq \emptyset$, then $\bar{V}_1=\bar{V}_2$ by Lemma \ref{lemma-6x}.
Likewise, if $(\mathcal{B}_1\cap\mathcal{B}_2)\cap \DW(F)\neq \emptyset$, then $\bar{W}_1=\bar{W}_2$.
But at least one of $(\mathcal{B}_1\cap\mathcal{B}_2)\cap \DV(F)$ and $(\mathcal{B}_1\cap\mathcal{B}_2)\cap \DW(F)$ must be empty otherwise $(\bar{V}_1,\bar{W}_1)=(\bar{V}_2,\bar{W}_2)$, violating Lemma \ref{lemma-BB-char}.

This completes the proof.
\end{proof}

\begin{lemma}\label{lemma-equivalent-GHSs}
Let $M$ and $F$ be as in Lemma \ref{lemma-BB-intersect}.
Then every component of $\DVW(F)$ is an equivalent cluster if and only if every $\V$- or $\W$-face has one equivalent class.
\end{lemma}

\begin{proof}
$(\Rightarrow)$ Suppose every component of $\DVW(F)$ is an equivalent cluster.
Let $\Delta$ be a $\V$-face $\{V,V',W\}$ without loss of generality.
Then $\Delta$ belong to an equivalent cluster $\mathcal{B}$ by the assumption, i.e. $\{V,V'\}=\Delta\cap\DV(F)\subset\mathcal{B}\cap \DV(F)$.
Considering Lemma \ref{lemma-3-9}, there is a $\V$-facial cluster having one equivalent class in $\mathcal{B}$, say $\varepsilon_\V$, such that $\varepsilon_\V\cap\DV(F)=\mathcal{B}\cap\DV(F)$ and the intersection is a star-shaped graph.
This means the $1$-simplex $\{V,V'\}$ is an edge of the star-shaped graph $\varepsilon_\V\cap\DV(F)$, i.e. $\{V,V'\}$ belongs to a $\V$-face having one equivalent class in $\varepsilon_\V$, say $\Delta_\V$.
Therefore, we can assume $V'$ cuts off a solid torus $
\V'$ from $\V$ and $V$ is a meridian disk of $\V'$ by Lemma \ref{lemma-equivalent}.
Hence, $\Delta$ also has one equivalent class by Lemma \ref{lemma-equivalent}.\\

$(\Leftarrow)$ Suppose every $\V$- or $\W$-face has one equivalent class.

\begin{claim}\label{claim-9}
There is no simplex of dimension at least four in $\DVW(F)$.
Moreover, if there is a $3$-simplex in $\DVW(F)$, then it has the form in the statement of Lemma \ref{lemma-restrict}.
\end{claim}

\begin{proof}
Note that every highest dimensional simplex of $\DVW(F)$ must intersect both $\DV(F)$ and $\DW(F)$.
(Otherwise, there would be a highest dimensional simplex $\Delta$ of $\DVW(F)$ such that $\Delta\subset\DV(F)$ or $\DW(F)$ and therefore $\Delta$ cannot be a proper subsimplex of a simplex intersecting both $\DV(F)$ and $\DW(F)$, violating the definition of $\DVW(F)$.)

Suppose $\dim(\DVW(F))=k$ and there is a $k$-simplex $\Sigma=\{V_1, \cdots, V_m, W_1, \cdots, W_n\}$ $\subset \DVW(F)$, where $V_i\subset \V$ for $1\leq i \leq m$ $(1\leq m)$, $W_i\subset \W$ for $1\leq i \leq n$ $(1\leq n)$, and $m+n=k+1$.
Considering the assumption that every $\V$- or $\W$-face has one equivalent class and the arguments in the proof of Lemma \ref{lemma-restrict}, we get $m,n\leq 2$ and confirm the desired shape of a $3$-simplex in $\DVW(F)$.

This completes the proof of Claim \ref{claim-9}.
\end{proof}

\begin{claim}\label{claim-10}
Every simplex of $\D(F)$ intersecting both $\DV(F)$ and $\DW(F)$ is contained in an equivalent cluster.
\end{claim}

\begin{proofN}{Claim \ref{claim-10}}
By Claim \ref{claim-9}, it is sufficient to consider simplices of dimension at most three.
Considering Lemma \ref{lemma-BB-char}, every weak reducing pair belongs to an equivalent cluster.
Moreover, the assumption that every $\V$- or $\W$-face has one equivalent class induces that every $\V$- or $\W$-face belongs to an equivalent cluster by Lemma \ref{lemma-BB-some}.
Suppose $\Sigma$ is a $3$-simplex intersecting both $\DV(F)$ and $\DW(F)$.
By Claim \ref{claim-9}, $\Sigma$ has the form in the statement of Lemma \ref{lemma-restrict}.
Hence, considering the argument in the proof of Lemma \ref{lemma-equiv-BB}, we can find an equivalent cluster containing $\Sigma$.

This completes the proof of Claim \ref{claim-10}.
\end{proofN}

Let $\mathcal{C}$ be a component of $\DVW(F)$.
Then $\mathcal{C}=\cup_{\alpha\in\mathcal{A}} \sigma_\alpha$, where each $\sigma_\alpha$ is a simplex intersecting both $\DV(F)$ and $\DW(F)$ by the definition of $\DVW(F)$ and $\mathcal{A}$ is an index set.
Hence, considering Claim \ref{claim-10}, $\mathcal{C}\subset \cup_{\alpha\in\mathcal{A}} \mathcal{B}_\alpha$, where $\mathcal{B}_\alpha$ is an equivalent cluster containing $\sigma_\alpha$.
Moreover, every equivalent cluster is a connected subset of $\DVW(F)$ and therefore $ \mathcal{B}_\alpha\subset \mathcal{C}$ for all $\alpha\in\mathcal{A}$, i.e. $\cup_{\alpha\in\mathcal{A}} \mathcal{B}_\alpha\subset \mathcal{C} $, and therefore we conclude $\mathcal{C}=\cup_{\alpha\in\mathcal{A}} \mathcal{B}_\alpha$.

It is sufficient to show that $\mathcal{C}$ consists of only one $\mathcal{B}_\alpha$.
For the sake of contradiction, assume there exist  $\mathcal{B}_{\alpha_1}$ and $\mathcal{B}_{\alpha_2}$ such that $\mathcal{B}_{\alpha_1}\neq\mathcal{B}_{\alpha_2}$.
If $\mathcal{B}_\alpha\cap\mathcal{B}_\beta=\emptyset$ for all mutually different $\mathcal{B}_\alpha$ and $\mathcal{B}_\beta$, then $\mathcal{C}$ is disconnected, leading to a contradiction, i.e.  we can assume $\mathcal{B}_{\alpha_1}\cap\mathcal{B}_{\alpha_2}\neq\emptyset$.
Let $(\bar{V}_i,\bar{W}_i)$ be the center of $\mathcal{B}_{\alpha_i}$ for $i=1,2$.
Then $(\bar{V}_1,\bar{W}_1)$ intersects $(\bar{V}_2,\bar{W}_2)$ in a vertex by Lemma \ref{lemma-BB-intersect}.
Therefore, there is a sequence of $\V$- and $\W$-faces $\Delta_0$, $\Delta_1$, $\cdots$, $\Delta_n$ such that the following hold by Lemma 8.4 of \cite{Bachman2008}:
\begin{enumerate}
\item $\Delta_{i-1}\cap \Delta_i$ is a weak reducing pair for $1\leq i \leq n$ and
\item $(\bar{V}_1,\bar{W}_1)\subset \Delta_0$ and $(\bar{V}_2,\bar{W}_2)\subset \Delta_n$.
\end{enumerate}
By the assumption that every $\V$- or $\W$-face has one equivalent class, each $\Delta_i$ has one equivalent class for $0\leq i \leq n$.
Therefore, each $\Delta_i$ belongs to a uniquely determined equivalent cluster $\mathcal{B}^i$ for $0\leq i \leq n$ by Lemma \ref{lemma-BB-some}.
Since two different equivalent clusters cannot share a weak reducing pair by Lemma \ref{lemma-BB-char}, we conclude $\mathcal{B}_{\alpha_1}=\mathcal{B}^0=\cdots=\mathcal{B}^n=\mathcal{B}_{\alpha_2}$, violating the assumption that $\mathcal{B}_{\alpha_1}\neq\mathcal{B}_{\alpha_2}$.

This completes the proof.
\end{proof}

Finally, we reach Theorem \ref{theorem-not-minimal-critical}.

\begin{theorem}\label{theorem-not-minimal-critical}
Suppose $M$ is an orientable, irreducible $3$-manifold and $(\V,\W;F)$ is a weakly reducible, unstabilized Heegaard splitting of $M$ of genus at least three.
If every $\V$- or $\W$-face has one equivalent class, then $F$ is either not topologically minimal or of topological index two  if it is topologically minimal.
\end{theorem}

\begin{proof}
Since every $\V$- or $\W$-face has one equivalent class, each component of $\DVW(F)$ is an equivalent cluster by Lemma \ref{lemma-equivalent-GHSs}.

Suppose that $\DVW(F)$ is disconnected, i.e. there are at least two mutually disjoint equivalent clusters $\mathcal{B}_0$ and $\mathcal{B}_1$.
Let $\bar{\D}(F)$ be the set of all compressing disks of $F$.
Then $\bar{\D}(F)$ can be partitioned into two subsets $C_0$ and $C_1$ as the following:
\begin{enumerate}
\item $C_0$ consists of compressing disks of $F$ whose isotopy classes are the vertices of $\mathcal{B}_0$.
\item $C_1=\bar{\D}(F)-C_0$.
\end{enumerate}
Then we can see $C_0$ and $C_1$ contain the compressing disks whose isotopy classes belong to $\mathcal{B}_0$ and $\mathcal{B}_1$, respectively, i.e. there is a weak reducing pair $(V_i,W_i)$ coming from $C_i$ for $i=0,1$.

We claim that for any $D\in C_0$ and $E\in C_1$, the pair $(D,E)$ cannot be a weak reducing pair.
For the sake of contradiction, suppose that $(D,E)$ is a weak reducing pair.
Then $(D,E)$ (as a pair of isotopy classes) is contained in some component of $\DVW(F)$.
But the assumption $D\in C_0$ means $(D,E)$ belongs to the component $\mathcal{B}_0$ and therefore this forces $E$ to belong to $C_0$,  violating the assumption that $C_0\cap C_1=\emptyset$.

Hence, the existence of the partition $\{C_0,C_1\}$ of $\bar{\D}(F)$ means $F$ is critical in the sense of \cite{Bachman2008} and therefore $F$ is of topological index two by Theorem 2.5 of \cite{Bachman2010}.

Suppose $\DVW(F)$ is connected.
Then $\DVW(F)$ is just an equivalent cluster, where it is contractible and $\DVW(F)\cap\DV(F)$ and $\DVW(F)\cap\DW(F)$ are also contractible by Lemma \ref{lemma-3-9}.
Therefore, $\D(F)$ is contractible by using the same arguments in the proof of Theorem 1.1 of \cite{JungsooKim2014}. 
Hence, $F$ is not topologically minimal, leading to the result.

This completes the proof.
\end{proof}

\section{the proof of Theorem \ref{theorem-1-1}}

In this section, we will prove Theorem \ref{theorem-1-1}.

Suppose there are two equivalent clusters such that the corresponding equivalent classes determined by $\Phi$ are the same.
The next lemma gives a sufficient condition that these two equivalent clusters are the same.

\begin{lemma}\label{lemma-3-11}
Suppose $M$ is an orientable, irreducible $3$-manifold and $(\V,\W;F)$ is a weakly reducible, unstabilized Heegaard splitting of $M$ of genus at least three.
Let $\mathcal{B}_1$ and $\mathcal{B}_2$ be equivalent clusters such that $\Phi(\mathcal{B}_1)=\Phi(\mathcal{B}_2)$ and $(\bar{V}_i,\bar{W}_i)$ the center of $\mathcal{B}_i$ for $i=1,2$.
Suppose if $\bar{V}_i$ (resp $\bar{W}_i$) is separating, then the closure of the component of $\V-\bar{V}_i$ (resp $\W-\bar{W}_i$) missing $\partial\bar{W}_i$ (resp $\partial\bar{V}_i$) is $S\times I$, where $S$ is a component of $\partial_-\V$ (resp $\partial_-\W$) for $i=1,2$.
Then $\mathcal{B}_1=\mathcal{B}_2$.
\end{lemma}

\begin{proof}
Let $\mathbf{H}_i=(\bar{F}_{\bar{V}_i},\bar{F}_{\bar{V}_i \bar{W}_i},\bar{F}_{\bar{W}_i})$ be the GHS obtained by weak reduction from $(\V,\W;F)$ along $(\bar{V}_i,\bar{W}_i)$ for $i=1,2$.
Since $\mathbf{H}_2$ is equivalent to $\mathbf{H}_1$ by the assumption that $\Phi(\mathcal{B}_1)=\Phi(\mathcal{B}_2)$, we get the following:
\begin{enumerate}
\item $\bar{F}_{V_2}$ and $\bar{F}_{W_2}$ are isotopic to $\bar{F}_{V_1}$ and $\bar{F}_{W_1}$ in $\V$ and $\W$, respectively, and\label{aaaaa}
\item $\mathcal{B}_1$ and $\mathcal{B}_2$ are of the same type by Lemma \ref{lemma-types}.\label{aaaab}
\end{enumerate}
By (\ref{aaaab}),  $\bar{V}_1$ is nonseparating in $\V$ if and only if $\bar{V}_2$ is nonseparating in $\V$ and the symmetric argument also holds for $\bar{W}_1$ and $\bar{W}_2$ in $\W$.

Suppose $\bar{V}_1$ is nonseparating in $\V$, i.e. both $\bar{V}_1$ and $\bar{V}_2$ are nonseparating  in $\V$ by the previous argument.
Since  $\bar{F}_{\bar{V}_2}$ is isotopic to $\bar{F}_{\bar{V}_1}$ in $\V$ by (\ref{aaaaa}), considering Corollary \ref{corollary-five-GHSs}, $\bar{V}_2$ is isotopic to $\bar{V}_1$ in $\V$ by (\ref{lemma-2-21-1}) of Lemma \ref{lemma-2-21}.

Suppose $\bar{V}_1$ is separating in $\V$, i.e. both $\bar{V}_1$ and $\bar{V}_2$ are separating  in $\V$.
By the assumption, $\bar{V}_1$ cuts off $S\times I$ missing $\partial \bar{W}_1$ from $\V$ for a component $S\subset\partial_-\V$ of genus $k\geq 1$.
Here, the genus of $\bar{F}_{\bar{V}_1}$ is $n-k$ by Corollary \ref{corollary-five-GHSs}.
Let $\tilde{\V}$ be the closure of the component of $\V-\bar{F}_{\bar{V}_1}$ intersecting $F$.
Then $\tilde{\V}$ is a compression body of genus $n$ with negative boundary consisting of the genus $k$ component $S$ and the genus $n-k$ component $\bar{F}_{\bar{V}_1}$, where $\partial_+\tilde{\V}=\partial_+\V$, by Lemma \ref{lemma-region}.
By (\ref{aaaaa}), we can isotope $\bar{F}_{\bar{V}_2}$ into $\bar{F}_{\bar{V}_1}$ by an isotopy $h_t:\V\to\V$, $t\in[0,1]$ such that  $h_1(\bar{F}_{\bar{V}_2})=\bar{F}_{\bar{V}_1}$.
Here, $\bar{V}_1$ and $h_1(\bar{V}_2)$ are compressing disks in $\tilde{\V}$, i.e. $h_1(\bar{V}_2)$ is isotopic to $\bar{V}_1$ in $\tilde{\V}$ by Lemma \ref{lemma-2-20}.
Therefore, $h_1(\bar{V}_2)$ is isotopic to $\bar{V}_1$ in $\V$ by an isotopy $g_t$ defined on $\V$ by Corollary \ref{corollary-isotopic-bd}.
Hence, the sequence of isotopies consisting of $h_t$ and $g_t$ gives an isotopy defined on $\V$ that takes $\bar{V}_2$ into $\bar{V}_1$.

Therefore, $\bar{V}_2$ is isotopic to $\bar{V}_1$ in $\V$ in any case.
Likewise, we conclude $\bar{W}_2$ is isotopic to $\bar{W}_1$ in $\W$.
Hence, $(\bar{V}_1,\bar{W}_1)=(\bar{V}_2,\bar{W}_2)$ in $\D(F)$.
Therefore, we get $\mathcal{B}_1=\mathcal{B}_2$ by Lemma \ref{lemma-BB-char}.

This completes the proof.
\end{proof}

Finally, we reach Theorem \ref{theorem-1-1-main}.

\begin{theorem}\label{theorem-1-1-main}
Let $(\V,\W;F)$ be a weakly reducible, unstabilized Heegaard splitting of genus three  in an orientable, irreducible $3$-manifold $M$.
Then the domain of $\Phi$ is the set of components of $\DVW(F)$, $\Phi$ is bijective, and there is a canonically induced function $\Omega$ from the set of components of $\DVW(F)$ to the set of the isotopy classes of the generalized Heegaard splittings obtained by weak reductions from $(\V,\W;F)$.
The number of components of the preimage of an isotopy class of $\Omega$ is the number of ways to embed the thick level contained in $\V$ into $\V$ (or in $\W$ into $\W$).
This means if we consider a generalized Heegaard splitting $\mathbf{H}$ obtained by weak reduction from $(\V,\W;F)$, then the way to embed the thick level of $\mathbf{H}$ contained in $\V$ into $\V$ determines the way to embed the thick level of $\mathbf{H}$ contained in $\W$ into $\W$ up to isotopy and vise versa.
\end{theorem}

\begin{proof} 
By Corollary \ref{lemma-equivalent-genus3}, every $\V$- or $\W$-face has one equivalent class.
Therefore, Lemma \ref{lemma-equivalent-GHSs} means every component of $\DVW(F)$ is an equivalent cluster.

Let $\mathcal{B}_1$ and $\mathcal{B}_2$ be equivalent clusters such that $\Phi(\mathcal{B}_1)=\Phi(\mathcal{B}_2)$ and $(\bar{V}_i,\bar{W}_i)$ the center of $\mathcal{B}_i$ for $i=1,2$.
Considering Lemma \ref{lemma-center}, if $\bar{V}_i$ is separating, then it cannot cut off a solid torus from $\V$.
Moreover, since the genus of $F$ is three, the closure of the component of $\V-\bar{V}_i$ containing $\partial \bar{W}_i$ is a genus two compression body by Lemma \ref{lemma-2-8} and Lemma 1.3 of \cite{ScharlemannThompson1993}.
This means the closure of the other component  must be $(\text{torus})\times I$ if $\bar{V}_i$ is separating.
The symmetric argument also holds for $\bar{W}_i$.
Therefore, Lemma \ref{lemma-3-11} induces $\mathcal{B}_1=\mathcal{B}_2$, i.e. $\Phi$ is an injective function from the set of components of $\DVW(F)$ to $\mathcal{GHS}^\ast_F/\sim$.
Considering Definition \ref{definition-Phi}, we conclude $\Phi$ is bijective.

Recall that $\mathcal{GHS}_F$ is the set of isotopy classes of the GHSs obtained by weak reductions from $(\V,\W;F)$.
Since all GHSs in an equivalent class of $\mathcal{GHS}^\ast/\sim$ correspond to the same isotopy class,  there is a function $\Omega$ from the set of components of $\DVW(F)$ to $\mathcal{GHS}_F$ canonically induced from $\Phi$.

Let us consider the preimage $\Omega^{-1}([\mathbf{H}])$ for an isotopy class $[\mathbf{H}]\in\mathcal{GHS}_F$.
Let $\mathcal{B}_1$ and $\mathcal{B}_2$ be two equivalent clusters in $\Omega^{-1}([\mathbf{H}])$, $(V_i,W_i)$  a weak reducing pair in $\mathcal{B}_i$ for $i=1,2$, and $\mathbf{H}_i$ the GHS obtained by weak reduction along $(V_i,W_i)$ for $i=1,2$.

If  $\bar{F}_{V_2}:=\Thick(\mathbf{H}_2)\cap\V$ is not isotopic to $\bar{F}_{V_1}:=\Thick(\mathbf{H}_1)\cap\V$ in $\V$, then $\Phi(\mathcal{B}_1)=[\mathbf{H}_1]^\ast\neq[\mathbf{H}_2]^\ast=\Phi(\mathcal{B}_2)$ and therefore $\mathcal{B}_1\neq\mathcal{B}_2$.

Suppose  $\bar{F}_{V_2}$ is isotopic to $\bar{F}_{V_1}$ in $\V$.
Without changing the equivalent classes $[\mathbf{H}_1]^\ast$ and $[\mathbf{H}_2]^\ast$, we can assume $(V_i,W_i)$ is the center of $\mathcal{B}_i$ for $i=1,2$ by Lemma \ref{lemma-3-9-ii}.

We claim that $\mathcal{B}_1=\mathcal{B}_2$.
For the sake of contradiction, suppose $\mathcal{B}_1$ and $\mathcal{B}_2$ are different.
By the assumption, we can isotope $\bar{F}_{V_2}$ into $\bar{F}_{V_1}$ by an isotopy $h_t:\V\to\V$, $t\in[0,1]$  such that $h_1(\bar{F}_{V_2})=\bar{F}_{V_1}$.
Let $\V'$ be the closure of the component of $\V-\bar{F}_{V_1}$ intersecting $F$.
Then $\V'$ is a genus three compression body with at least one negative boundary component $\bar{F}_{V_1}$ whose genus is two by Corollary \ref{corollary-five-GHSs} and Lemma \ref{lemma-region} and $V_1$ and $h_1(V_2)$ are compressing disks in $\V'$, where $\partial_+\V=\partial_+\V'$.

If $\partial_-\V'$ is connected, then every separating compressing disk in $\V'$ cuts off a solid torus from $\V'$ and also does in $\V$.
Therefore, considering that each $V_i$ belongs to the center of $\mathcal{B}_i$ for $i=1,2$, Lemma \ref{lemma-center} forces $V_1$ and $V_2$ to be nonseparating disks in $\V$, i.e. $\partial V_1$ and $\partial V_2$ are nonseparating in $\partial_+\V=\partial_+\V'$.
This means $V_1$ and $h_1(V_2)$ are nonseparating compressing disks in $\V'$.
But there is a unique nonseparating disk in $\V'$ up to isotopy by Lemma \ref{lemma-2-19}, i.e. $h_1(V_2)$ is isotopic to $V_1$ in $\V'$.
Therefore, $h_1(V_2)$ is isotopic to $V_1$ in  $\V$ by an isotopy $g_t$ defined on $\V$ by Corollary \ref{corollary-isotopic-bd}.
Hence, the sequence of isotopies consisting of $h_t$ and $g_t$ gives an isotopy defined on $\V$ taking $V_2$ into $V_1$.
But this means the centers of $\mathcal{B}_1$ and $\mathcal{B}_2$ intersect each other, violating the assumption that each component of $\DVW(F)$ is an equivalent cluster itself.
If $\partial_-\V'$ is disconnected, then $V_1$ and $h_1(V_2)$ are isotopic  in $\V'$ by Lemma \ref{lemma-2-20}.
Hence, we also get a contradiction similarly.

Therefore, $\mathcal{B}_1=\mathcal{B}_2$, i.e. $\Phi(\mathcal{B}_1)=[\mathbf{H}_1]^\ast=[\mathbf{H}_2]^\ast=\Phi(\mathcal{B}_2)$ and therefore $\Thick(\mathbf{H}_2)\cap\W$ is also isotopic to $\Thick(\mathbf{H}_1)\cap\W$ in $\W$.

Hence, we only need to count the number of ways to embed the thick level contained in $\V$ into $\V$ in order to count the number of elements of $\Omega^{-1}([\mathbf{H}])$.
 
This completes the proof.
\end{proof}


\begin{thebibliography}{0}

\bibitem{Bachman2002} D. Bachman, Critical Heegaard surfaces, {\it Trans. Amer. Math. Soc}. {\bf 354} (2002), 4015--4042.

\bibitem{Bachman2008} D. Bachman, Connected sums of unstabilized Heegaard splittings are unstabilized, 
{\it Geom. Topol}. {\bf 12} (2008), 2327--2378.

\bibitem{Bachman2010} D. Bachman, Topological index theory for surfaces in 3--manifolds, {\it Geom. Topol}. {\bf 14} (2010), 585--609.

\bibitem{Bachman2012-1} D. Bachman, Normalizing Topologically Minimal Surfaces I: Global to Local Index, arXiv:1210.4573.

\bibitem{Bachman2012-2} D. Bachman, Normalizing Topologically Minimal Surfaces II: Disks,  arXiv:1210.4574.

\bibitem{Bachman2013-1} D. Bachman, Normalizing Topologically Minimal Surfaces III: Bounded Combinatorics,  arXiv:1303.6643.

\bibitem{IdoJangKobayashi2014} A. Ido, Y. Jang and T. Kobayashi, Heegaard splittings of distance exactly $n$, {\it Algebr. Geom. Topol}. {\bf 14} (2014), 1395--1411.

\bibitem{JungsooKim2013} J. Kim, On critical Heegaard splittings of tunnel number two composite knot exteriors, {\it J. Knot Theory Ramifications} \textbf{22} (2013), 1350065, 11 pp.

\bibitem{JungsooKim2012} J. Kim, On unstabilized genus three critical Heegaard surfaces, {\it Topology Appl}.  \textbf{165} (2014), 98--109.

\bibitem{JungsooKim2014} J. Kim, A topologically minimal, weakly reducible, unstabilized Heegaard splitting of genus three is critical, {\it Algebr. Geom. Topol}. \textbf{16} (2016) 1427--1451.

\bibitem{7} M. Lustig and Y. Moriah, Closed incompressible surfaces in complements of wide knots and links, 
{\it Topology Appl}. {\bf 92} (1999), 1--13.

\bibitem{MS2013} H. Masur and S. Schleimer, The geometry of the disk complex, {\it J. Amer. Math. Soc}. {\bf 26} (2013), 1--62. 

\bibitem{8} D. McCullough, Virtually geometrically finite mapping class groups of $3$-manifolds, 
{\it J. Differential Geom}. {\bf 33} (1991) 1--65.

\bibitem{Morimoto2015} K. Morimoto, On Heegaard splittings of knot exteriors with tunnel number degenerations, {\it Topology Appl}. {\bf 196} (2015), 719--728.

\bibitem{Gelca2014} R. Gelca, {\it Theta functions and knots}, World Scientific Publishing Co. Pte. Ltd., Hackensack, NJ, (2014). 


\bibitem{SaitoScharlemannSchultens2005} T. Saito, M. Scharlemann and J. Schultens, Lecture notes on generalized Heegaard splittings, arXiv:math/0504167v1.

\bibitem{ScharlemannThompson1993} M. Scharlemann and A. Thompson, Heegaard splittings of $(\text{surface})\times I$ are standard, {\it Math. Ann}. \textbf{295} (1993), 549--564.

\bibitem{ScharlemannThompson1994} M. Scharlemann and A. Thompson, Thin position for $3$-manifolds, \textit{AMS Contemp. Math.} {\bf 164} (1994), 231--238.

\end{thebibliography}
\end{document}